\documentclass[11pt]{article}
\pdfoutput=1
\usepackage[UKenglish]{isodate}
\usepackage[T1]{fontenc}
\usepackage[noBBpl]{mathpazo}

\usepackage{amsmath, amssymb, amsfonts, amsthm, dsfont, zlmtt, graphicx, array, float, url, xpatch, stmaryrd}

\usepackage[a4paper,left=2.8cm,right=2.8cm,top=2.7cm,bottom=2.7cm]{geometry}
\usepackage{microtype}
\usepackage[shortcuts]{extdash}
\cleanlookdateon

\usepackage{titlesec}
\titlespacing*{\section}{0pt}{\baselineskip}{0pt}
\titlespacing*{\subsection}{0pt}{0.66\baselineskip}{0pt}

\usepackage[shortlabels]{enumitem}
\setlist{leftmargin=0.8cm,topsep=0pt,itemsep=-2pt}
\setlist[enumerate]{label=\rm{(\roman*)}}

\numberwithin{equation}{section}

\usepackage[tableposition=top, figureposition=bottom, skip=0.5\baselineskip]{caption}

\makeatletter
\g@addto@macro\normalsize{%
  \setlength\abovedisplayskip{0.4\baselineskip plus 0.4\baselineskip}
  \setlength\belowdisplayskip{0.4\baselineskip plus 0.4\baselineskip}
  \setlength\abovedisplayshortskip{-0.3\baselineskip}
  \setlength\belowdisplayshortskip{0.4\baselineskip plus 0.4\baselineskip}
}
\makeatother

\makeatletter\def\blfootnote{\gdef\@thefnmark{}\@footnotetext} \makeatother
\newcommand{\dateline}[1]{\enlargethispage{16pt}\blfootnote{\phantom{\Large M}\hspace{-1em}\hspace{-20pt}\emph{Date} #1}}

\renewenvironment{thebibliography}[1]
{ \begin{oldthebibliography}{#1}
  \setlength{\parskip}{0pt}
  \setlength{\itemsep}{2pt plus 0.3ex}
  \bgroup\footnotesize }
{ \egroup \end{oldthebibliography} }

\makeatletter
\renewenvironment{proof}[1][\proofname]{\par
  \pushQED{\qed}%
  \normalfont
  \topsep2pt \partopsep1pt 
  \trivlist
  \item[\hskip\labelsep
        \itshape
    #1\@addpunct{.}]\ignorespaces
}{%
  \popQED\endtrivlist\@endpefalse
  \addvspace{6pt plus 6pt}
}
\makeatother

\newtheoremstyle{shdefinition}{8pt}{4pt}{}{}{\bfseries\boldmath}{.}{0.3em}{} 
\newtheoremstyle{shplain}{8pt}{4pt}{\itshape}{}{\bfseries\boldmath}{.}{0.3em}{} 
\theoremstyle{shdefinition}
\newtheorem{definition}{Definition}[section]
\newtheorem{remark}[definition]{Remark}
\newtheorem{example}[definition]{Example}
\newtheorem*{example*}{Example}
\newtheorem*{organisation*}{Organisation}
\newtheorem*{notation*}{Notation}
\newtheorem*{acknowledgements*}{Acknowledgements}
\newtheorem*{furtheracknowledgements*}{Further acknowledgements (rights retention statement)}
\theoremstyle{shplain}
\newtheorem{theorem}[definition]{Theorem}
\newtheorem{shtheorem}{Theorem}
\newtheorem*{theorem*}{Theorem}
\newtheorem{corollary}[definition]{Corollary}
\newtheorem{proposition}[definition]{Proposition}
\newtheorem{lemma}[definition]{Lemma}

\renewcommand{\a}{\alpha}
\renewcommand{\b}{\beta}
\newcommand{\g}{\gamma}
\renewcommand{\d}{\delta}
\newcommand{\e}{\varepsilon}
\newcommand{\p}{\varphi}
\newcommand{\s}{\sigma}
\renewcommand{\l}{\lambda}
\renewcommand{\t}{\tau}

\renewcommand{\wp}{\widetilde{\varphi}}
\newcommand{\ws}{\widetilde{\sigma}}
\newcommand{\wrho}{\widetilde{\rho}}

\newcommand{\C}{\mathcal{C}}
\renewcommand{\L}{\mathcal{L}}

\renewcommand{\S}{\mathcal{S}}

\newcommand{\<}{\langle}
\renewcommand{\>}{\rangle}
\renewcommand{\leq}{\leqslant}
\renewcommand{\geq}{\geqslant}
\newcommand{\leqn}{\trianglelefteqslant}

\newcommand{\Aut}{\mathrm{Aut}}

\newcommand{\Out}{\mathrm{Out}}
\newcommand{\Inndiag}{\mathrm{Inndiag}}
\newcommand{\tr}{\mathsf{T}}
\newcommand{\F}{\mathbb{F}}
\newcommand{\FF}{\overline{\F}}
\renewcommand{\:}{\colon}
\renewcommand{\mod}[1]{\mathrm{ \ } (\mathrm{mod\ } #1)}

\renewcommand{\div}{\mid}
\newcommand{\ndiv}{\nmid}

\newcommand{\Sym}{\mathrm{Sym}}
\newcommand{\Alt}{\mathrm{Alt}}

\newcommand{\SL}{\mathrm{SL}}
\newcommand{\GL}{\mathrm{GL}}

\newcommand{\PSL}{\mathrm{PSL}}
\newcommand{\PGL}{\mathrm{PGL}}
\newcommand{\Sp}{\mathrm{Sp}}
\newcommand{\PSp}{\mathrm{PSp}}
\newcommand{\SU}{\mathrm{SU}}
\newcommand{\GU}{\mathrm{GU}}
\newcommand{\PSU}{\mathrm{PSU}}
\newcommand{\SO}{\mathrm{SO}}
\renewcommand{\O}{\mathrm{O}}
\newcommand{\POm}{\mathrm{P}\Omega}
\newcommand{\PO}{\mathrm{PO}}
\newcommand{\PGO}{\mathrm{PGO}}
\newcommand{\PGaO}{\mathrm{P}\Gamma\mathrm{O}}

\begin{document}
 
\begin{center} 
{\LARGE \textbf{Totally deranged elements of almost simple groups \\[2pt] and invariable generating sets}} \\[11pt]
{\Large Scott Harper}                                                                                      \\[22pt]
\end{center}

\begin{center}
\begin{minipage}{0.8\textwidth}
\small By a classical theorem of Jordan, every faithful transitive action of a nontrivial finite group has a derangement (an element with no fixed points). The existence of derangements with additional properties has attracted much attention, especially for faithful primitive actions of almost simple groups. In this paper, we show that an almost simple group can have an element that is a derangement in \emph{every} faithful primitive action, and we call these elements \emph{totally deranged}. In fact, we classify the totally deranged elements of all almost simple groups, showing that an almost simple group $G$ contains a totally deranged element only if the socle of $G$ is $\mathrm{Sp}_4(2^f)$ or $\mathrm{P}\Omega^+_n(q)$ with $n=2^l \geqslant 8$. Using this, we classify the invariable generating sets of a finite simple group $G$ of the form $\{ x, x^a \}$ where $x \in G$ and $a \in \mathrm{Aut}(G)$, answering a question of Garzoni. As a final application, we classify the elements of almost simple groups that are contained in a unique maximal subgroup $H$ in the case where $H$ is not core-free, which complements the recent work of Guralnick and Tracey addressing the case where $H$ is core-free.\par
\end{minipage}
\end{center}

\dateline{26 March 2024 \ \emph{MSC 2020} Primary: 20E32; Secondary: 20B05, 20E28, 20F05}

\section{Introduction} \label{s:intro} 

Let $G$ be a finite group with a transitive action on a set $\Omega$ of size $|\Omega|>1$. Then there exists an element of $G$ that fixes no point of $\Omega$ in this action; such an element is known as a \emph{derangement}. This follows quickly from the orbit counting lemma and was first observed by Jordan \cite{ref:Jordan72}. As highlighted by Serre \cite{ref:Serre03}, this basic fact has consequences for number theory and topology. Recently, an extensive field of research has emerged on counting derangements and finding derangements with specified properties. For instance, Fulman and Guralnick \cite{ref:FulmanGuralnick18} proved the Boston--Shalev conjecture that the proportion of derangements in any nontrivial transitive action of a finite simple group is bounded away from zero by an absolute constant. Many problems in this area reduce to faithful primitive actions of almost simple groups, and we refer the reader to the introductory chapter of \cite{ref:BurnessGiudici16} for an overview of recent work. (Recall that a group $G$ is \emph{almost simple} if $T \leq G \leq \Aut(T)$ for a nonabelian simple group $T$.) 

Since every faithful primitive action of an almost simple group admits a derangement, it is natural to reverse the quantifiers and ask whether it is possible for an almost simple group to contain an element that is a derangement in \emph{every} faithful primitive action? Despite how restrictive this condition is, the following example shows that it is possible.

\begin{example*}
Let $G = \O^+_{2m}(q)$ where $2m \geq 8$ and $q$ are powers of two. Then any element of $G$ of order $q^m-1$ is a derangement in every faithful primitive action of $G$.
\end{example*}

Motivated by this example, we make the following definition. An element $x$ of a group $G$ is \emph{totally deranged} if $x$ is a derangement in every faithful primitive action of $G$.

Our first main theorem classifies the almost simple groups with totally deranged elements.

\begin{shtheorem} \label{thm:derangement_basic}
Let $G$ be an almost simple group with socle $T$. Then $G$ contains a totally deranged element if and only if $T = \Sp_4(2^f)$ and $G$ contains a graph-field automorphism, or $T = \POm^+_n(q)$ with $2m=2^l \geq 8$ and $G$ contains a graph automorphism or a product of graph and field automorphisms.
\end{shtheorem}

In fact, with some more notation, we can give an expanded version of Theorem~\ref{thm:derangement_basic} that classifies all totally deranged elements of almost simple groups. For a finite simple group $T$ of Lie type, we write $\p$ for the standard field automorphism and $\Inndiag(T)$ for the subgroup of $\Aut(T)$ generated by inner and diagonal automorphisms. (We formally define $\Inndiag(T)$ in Section~\ref{ss:p_groups}, but let us note that if $T$ is $\Sp_4(2^f)$ or $\POm^+_{2m}(2^f)$, then $\Inndiag(T)  = T$, and if $T = \POm^+_{2m}(p^f)$ for odd $p$, then $\Inndiag(T)$ is a particular index two subgroup of $\PGO^+_{2m}(p^f)$.)

\begin{shtheorem} \label{thm:derangement}
Let $G$ be an almost simple group with socle $T$. Then $x \in G$ is totally deranged if and only if all of the following conditions hold, where $p$ is prime and $q=p^f$:
\begin{enumerate}
\item $T$ is either $\POm^+_{2m}(q)$ with $2m=2^l \geq 8$, or $\Sp_{2m}(q)$ with $2m=4$ and $p=2$
\item $G \not\leq \< \Inndiag(T), \p \>$
\item $x \in \Inndiag(T)\a$ for $\a \in \< \p \>$ of odd order $e$, and, writing $q=q_0^e$, we have $x^e = su = us$ for:
\begin{enumerate}[{\rm (a)}]
\item $u = 1$ or $|u| = p > 2$
\item $|s| \div (q_0^{m}-1)$ and $|x^e| \ndiv 2(q_0^{m/2} + 1)$
\item $|s| \ndiv (q_0^{m/k}-1)$ for all prime divisors $k$ of $2f$ that do not divide $e$
\item $k \div |s|$ and $k \ndiv (q_0^{m}-1)/|s|$ for all prime divisors $k$ of $e$.
\end{enumerate}
\end{enumerate}
\end{shtheorem}

Observe that Theorem~\ref{thm:derangement} is simpler if $x \in T$, since in this case $e=1$, so part~(d) is vacuous, $x=su$ and $q_0=q$. Theorem~\ref{thm:derangement_basic} and the example above immediately follow from Theorem~\ref{thm:derangement}.

Theorem~\ref{thm:derangement} has an application to generation. A subset $\{ x_i \mid i \in I \}$ of a group $G$ is an \emph{invariable generating set} of $G$ if $\{ x_i^{g_i} \mid i \in I \}$ is a generating set of $G$ for all choices of $g_i \in G$. This notion was first introduced by Dixon \cite{ref:Dixon92} motivated by computational Galois theory. There is a close connection between derangements and invariable generation given by the easy observation that a group $G$ admits an invariable generating set if and only if every transitive action of $G$ on a set of size at least two has a derangement. In particular, every finite group has an invariable generating set, and much attention has been dedicated to finding invariable generating sets of minimal size. For instance, every finite simple group is known to have an invariable generating set of size two \cite{ref:GuralnickMalle12JLMS,ref:KantorLubotzkyShalev11}, which strengthens the famous result that every finite simple group has a generating set of size two (see \cite{ref:AschbacherGuralnick84,ref:Steinberg62}).

In \cite[Question~2.14]{ref:Garzoni20}, Garzoni asks whether there exists a nonabelian finite simple group $T$ that is invariably generated by $\{ x, x^a \}$ for some $x \in T$ and $a \in \Aut(T)$. This is clearly impossible if $\Out(T) = 1$. A short elementary argument shows that if $\{ x, x^a \}$ is an invariable generating set for a nonabelian finite simple group $T$, then $x$ must be a totally deranged element of the almost simple group $\< T, a \>$  (see Lemma~\ref{lem:invariable}). This leads to our next theorem, which gives an affirmative answer to Garzoni's question.

\begin{shtheorem} \label{thm:invariable}
Let $T$ be a nonabelian finite simple group, let $a \in \Aut(T)$ and let $x \in T$. Then the following are equivalent:
\begin{enumerate}
\item $T$ is invariably generated by $\{x,x^a\}$
\item $x$ is a totally deranged element of $\< T, a \>$
\item the following hold, where $p$ is prime and $q=p^f$:
\begin{enumerate}[{\rm (a)}]
\item $T$ is either $\POm^+_{2m}(q)$ with $2m=2^l \geq 8$, or $\Sp_{2m}(q)$ with $2m=4$ and $p=2$
\item $a \not\in \< \Inndiag(T), \p \>$
\item $x=su=us$ where $|s| \div (q^m-1)$, $|x| \ndiv 2(q^{m/2}+1)$, $|s| \ndiv (q^{m/k}-1)$ for all prime divisors $k$ of $2f$ and either $u=1$ or $|u|=p>2$.
\end{enumerate}
\end{enumerate}
\end{shtheorem}

Theorem~\ref{thm:derangement} gives (ii)$\implies$(iii). Interestingly, while (i)$\implies$(ii) is elementary (see Lemma~\ref{lem:invariable}), we prove the converse (ii)$\implies$(i) by directly showing that (i) satisfied by the groups in (iii).

In \cite{ref:GuralnickTracey}, Guralnick and Tracey classify the elements of almost simple groups that are contained in a unique maximal subgroup $H$ under the assumption that $H$ is core-free (they use this to classify the finite groups $G$ and elements $x$ satisfying $\< x^G \> = G$ such that $x$ is contained in a unique maximal subgroup of $G$). However, examples also occur when $H$ is not core-free and our next theorem complements \cite{ref:GuralnickTracey} by solving the problem in this case.

\begin{shtheorem} \label{thm:unique}
Let $G$ be almost simple with socle $T$ and let $x \in G$. Then the following are equivalent:
\begin{enumerate}
\item $x$ is contained in a unique maximal subgroup $H$ of $G$ and $H$ is not core-free
\item $x$ is a totally deranged element of $G$ and $Tx$ is contained in a unique maximal subgroup of $G/T$
\item $(G,x)$ appears in Theorem~\ref{thm:derangement} and, in addition, the following hold:
\begin{enumerate}[{\rm (a)}]
\item if $T = \Sp_4(q)$, then $G = \< T, \rho^i \>$ for a graph-field automorphism $\rho$ that satisfies $\rho^2=\p$ and where $i$ divides $f/e$ and $f/ei$ is a power of $2$
\item if $T = \POm^+_{2m}(q)$, then $G = \< T, x, y\p^i \>$ for a duality (respectively, triality) graph automorphism $y$ and where $i$ divides $f/e$ and $f/ei$ is a power of $2$ (respectively, $3$).
\end{enumerate}
\end{enumerate}
\end{shtheorem}

The equivalence (i)$\iff$(ii) is elementary, so to complete the proof we have to determine when $Tx$ is contained in a unique maximal subgroup of $G/T$ for the groups in Theorem~\ref{thm:derangement}. Notice that if $x \in T$, then we simply require that $G/T$ be cyclic of prime power order. In Proposition~\ref{prop:unique} we will see that when $x$ is contained in a unique maximal subgroup $H$ of $G$ and $H$ is not core-free, it is always the case that $H = G \cap \< \Inndiag(T), \p \>$.

We now discuss our methods. Let $G$ be an almost simple group with socle $T$. If $G$ acts transitively with point stabiliser $H \leq G$, then $x \in G$ is a derangement if and only if $x^G \cap H$ is empty. Therefore, $x \in G$ is totally deranged if and only if $x$ is contained in no core-free maximal subgroup of $G$. This gives a framework for proving Theorem~\ref{thm:derangement}. Indeed, to show that $G$ contains no totally deranged elements, we show that every element of $G$ is contained in a maximal subgroup not containing $T$. For alternating $T$ the proof is short, and for sporadic $T$ we use computational methods exploiting the Character Table Library \cite{ref:CTblLib} in \textsf{GAP} \cite{ref:GAP}.

Most of our effort is devoted to the case where $T$ is a finite simple group of Lie type. The most interesting cases are when $T \in \{ B_2(2^f), G_2(3^f), F_4(2^f) \}$ and $G$ contains a graph-field automorphism or $T \in \{\PSL_n(q), \POm^+_{2m}(q), E_6(q) \}$ and $G$ contains a graph or graph-field automorphism, since graph(-field) automorphisms affect the maximality of subgroups. We use a variety of geometric and Lie theoretic methods to study the elements of $\Inndiag(T)$, but the elements in $\Aut(T) \setminus \Inndiag(T)$ are more opaque and we apply \emph{Shintani descent}. This technique was introduced by Shintani \cite{ref:Shintani76} in 1976 and has its origins in character theory. However, in recent years it has seen powerful applications in group theory, for example in the series of papers leading to the classification of the finite $\frac{3}{2}$-generated groups \cite{ref:BurnessGuralnickHarper21}.

In this paper, we establish new general results on Shintani descent, which we expect will have further applications (indeed they have already been used in \cite{ref:GuralnickTracey}, see below). For a connected linear algebraic group $X$ over $\FF_{p}$ and two commuting Steinberg endomorphisms $\s_1$ and $\s_2$, the \emph{Shintani map} is a bijection between the $X_{\s_1}$-classes in $X_{\s_1}\s_2$ and the $X_{\s_2}$-classes in $X_{\s_2}\s_1$ (here $X_{\s_i}$ is the set of fixed points of $X$ under $\s_i$). For instance, if $X = \PGL_n(\FF_p)$ and $\p$ is the standard Frobenius map, then for divisors $i$ of $f$ we have a bijection between the conjugacy classes in $\PGL_n(p^f)\p^i$ and in $\PGL_n(p^i)$. This bijection preserves important structure, including information about maximal overgroups (see \cite[Theorem~4]{ref:Harper21}). 

Corollary~\ref{cor:shintani_subfields} is a new result describing how Shintani descent relates to maximal subfield subgroups. In the previous example, it implies that if $k$ is a prime dividing $f$ but not $f/i$, then an element of $\PGL_n(p^f)\p^i$ is contained in a degree $k$ subfield subgroup of $\< \PGL_n(p^f), \p^i \>$ if and only if its image under the Shintani map is contained in a degree $k$ subfield subgroup of $\PGL_n(p^i)$. Corollary~\ref{cor:shintani_subfields} has recently been used by Guralnick and Tracey \cite{ref:GuralnickTracey} in their classification of elements of finite groups contained in a unique maximal subgroup.

\begin{organisation*}
Section~\ref{s:prelims} records preliminary results on groups of Lie type and Shintani descent, with our new results on Shintani descent and subfields being presented in Section~\ref{ss:p_shintani_subfields}. Theorem~\ref{thm:derangement} is proved in Section~\ref{s:proof}, then Theorems~\ref{thm:invariable} and~\ref{thm:unique} are proved in Section~\ref{s:applications}.
\end{organisation*}

\begin{notation*}
We follow \cite{ref:GorensteinLyonsSolomon98,ref:KleidmanLiebeck}. In particular, $\Omega_n(\FF_p)$ is the connected component $\O_n(\FF_p)^\circ$. For a Lie type $\Sigma$, by $\Sigma(\FF_p)$ we mean the adjoint algebraic group and by $\Sigma(q)$ we mean the (typically) simple group, so $A_m(\FF_p) = \PGL_{m+1}(\FF_p)$ and ${}^2A_m(q) = \PSU_{m+1}(q)$. We also write $\SL^+_n(q) = \SL_n(q)$ and $\SL^-_n(q) = \SU_n(q)$ and similarly $E_6^+(q) = E_6(q)$ and $E_6^-(q) = {}^2E_6(q)$. For $H,K \leq G$, we write $N_H(K) = \{ h \in H \mid K^h = K\}$ even when $K$ is not a subgroup of $H$. We write $(a,b)$ for the greatest common divisor of natural numbers $a$ and $b$. 
\end{notation*}

\begin{acknowledgements*}
The author wrote this paper when he was first a Leverhulme Early Career Fellow (ECF-2022-154) and then an EPSRC Postdoctoral Fellow (EP/X011879/1), and he thanks the Leverhulme Trust and the Engineering and Physical Sciences Research Council. He thanks Jay Taylor, Adam Thomas and the anonymous referee for their helpful comments.
\end{acknowledgements*}

\section{Preliminaries} \label{s:prelims}

\subsection{Groups of Lie type} \label{ss:p_groups}

We take this section to introduce our notation for the almost simple groups of Lie type. Our notation is consistent with \cite{ref:Harper21}, to which we will refer later. 

Let $X$ be a linear algebraic group over $\FF_p$, with $p$ prime, which we call an \emph{algebraic group}. For a Steinberg endomorphism $\s$ of $X$, write $X_{\s} = \{ x \in X \mid x^\s = x \}$. Let $\L$ be the set of finite groups $T$ such that $T = O^{p'}(X_\s)$ for a simple algebraic group $X$ of adjoint type and a Steinberg endomorphism $\s$. (Recall that $O^{p'}(G)$ is the subgroup generated by the $p$-elements of $G$.) Usually, if $T \in \L$, then $T$ is simple and we call it a \emph{finite simple group of Lie type}. (For us, the Tits group ${}^2F_4(2)'$ is not a finite simple group of Lie type.) We call $A_m(q)$, \dots, $G_2(q)$ \emph{untwisted}, ${}^2A_m(q)$, ${}^2D_m(q)$, ${}^2E_6(q)$, ${}^3D_4(q)$ \emph{twisted} and ${}^2B_2(2^f)$, ${}^2F_4(2^f)$, ${}^2G_2(3^f)$ \emph{very twisted}.

We say that an element $x \in X$ is \emph{semisimple} if $|x|$ is prime to $p$ and \emph{unipotent} if $|x|$ is a power of $p$. In particular, the identity is both semisimple and unipotent. Every element of $X$ is expressible uniquely as $x = su = us$ where $s$ is semisimple and $u$ unipotent (this is the \emph{Jordan decomposition} of $x$). We say that $x$ is \emph{mixed} if it is neither semisimple nor unipotent.

Let $X$ be a simple algebraic group of adjoint type. Let us fix some notation:
\begin{enumerate}[ ]
\setlength{\itemsep}{-3pt}
\item $\p$ is the Frobenius endomorphism of $X$ fixing $\F_p$
\item $\g$ is the standard involutory graph automorphism of $X  \in \{ A_m, D_m, E_6 \}$
\item $\t$ is the standard triality graph automorphism of $X = D_4$
\item $\rho$ is the graph-field endomorphism of $X$ fixing $\F_p$ if $(X,p) \in \{ (B_2, 2), (F_4,2), (G_2,3) \}$.
\end{enumerate}
\enlargethispage{3pt}

Write $\Sigma(X)$ for the group generated by the maps $\p$, $\g$, $\t$, $\rho$ when they are defined. For $T = O^{p'}(X_\s)\in \L$, note that $\Aut(T) \cong \Inndiag(T){:}\Sigma(T)$, where $\Inndiag(T) = X_\s$ and $\Sigma(T) = \{ g|_{X_\s} \mid \text{$g \in \Sigma(X)$ and $g^\s = g$} \}$ (see \cite[Theorem~2.5.4]{ref:GorensteinLyonsSolomon98}). 

It will be useful to record the following immediate consequence of \cite[Theorem~2.5.12]{ref:GorensteinLyonsSolomon98}.
\begin{lemma} \label{lem:groups_automorphism}
Let $T$ be a finite simple group of Lie type. Then $\Aut(T) = \< \Inndiag(T), \p \>$ or
\begin{enumerate}
\item $\Aut(T) = \< T, \rho \>$ and $T \in \{ B_2(2^f), F_4(2^f), G_2(3^f) \}$
\item $\Aut(T) = \< \Inndiag(T), \p, \g \>$ and $T \in \{ A_m(q) \, \text{\rm ($m \geq 2$)}, \ D_m(q) \, \text{\rm ($m \geq 5$)}, \ E_6(q) \}$
\item $\Aut(T) = \< \Inndiag(T), \p, \g, \t \>$ and $T = D_4(q)$.
\end{enumerate}
\end{lemma}

\begin{remark} \label{rem:groups_automorphism}
Let us comment on Lemma~\ref{lem:groups_automorphism}.
\begin{enumerate}
\item In Lemma~\ref{lem:groups_automorphism}, we are abusing notation by, for example, conflating the map $\p\:X \to X$ with its restriction $\p|_T\:T \to T$. Usually this causes no problems, but when we need more precision, for an endomorphism $\a$ of $X$ that restricts to an automorphism of $T = O^{p'}(X_{\s})$, we write $\widetilde{\a} = \a|_{\Inndiag(T)}$. For instance, if $X=B_2$, $p=2$ and $f$ is odd, then $\<\p\>$ is a proper subgroup of $\<\rho\>$, but for $T = {}^2B_2(2^f) = X_{\rho^f}$, we have $\<\wp\> = \<\wrho\>$.
\item Lemma~\ref{lem:groups_automorphism} demonstrates that $\< \Inndiag(T), \p \>$ is a normal subgroup of $\Aut(T)$. Indeed, if $T \neq \POm^+_8(q)$, then $\< \Inndiag(T), \p \>$ has index at most two in $\Aut(T)$. If $T =  \POm^+_8(q)$, then $\Aut(T)/\< \Inndiag(T), \p \> \cong S_3$, and, for $T \leq G \leq \Aut(T)$, we say that the almost simple group \emph{$G$ contains triality} if $|G/(G \cap \<\Inndiag(T),\p\>)|$ is divisible by $3$.
\end{enumerate}
\end{remark}

Let us now discuss the the subgroup structure of the finite groups of Lie type. Let $T$ be a finite simple group of Lie type and let $T \leq G \leq \Aut(T)$. The core-free maximal subgroups of $G$ are of central importance to this paper and they are described by the following theorem, which combines \cite[Theorem~2]{ref:LiebeckSeitz90} and \cite[Theorem~2]{ref:LiebeckSeitz98} of Liebeck and Seitz .

\begin{theorem} \label{thm:maximal}
Let $G$ be an almost simple group of Lie type with socle $T$. Write $T= O^{p'}(X_\s)$ for a simple algebraic group $X$ of adjoint type and a Steinberg endomorphism $\s$ of $X$. Let $H$ be a maximal subgroup of $G$ not containing $T$. Then $H$ is one of
\begin{enumerate}[{\rm (I)}]
\item $N_G(Y_\s \cap T)$ for a maximal closed $\s$-stable positive-dimensional subgroup $Y$ of $X$
\item $N_G(X_\a \cap T)$ for a Steinberg endomorphism $\a$ of $X$ such that $\a^k=\s$ for a prime $k$
\item a local subgroup not in (I)
\item an almost simple group not in (I) or (II)
\item the Borovik subgroup: $H \cap T = (\Alt_5 \times \Alt_6).2^2$ with $T = E_8(q)$ and $p \geq 7$.
\end{enumerate}
\end{theorem} 

\begin{remark} \label{rem:maximal}
Let us comment on Theorem~\ref{thm:maximal}.
\begin{enumerate}
\item An important class of examples arising in (I) are parabolic subgroups. For a simple algebraic group or a finite group of Lie type, we write $P_{i_1, \dots, i_k}$ for the parabolic subgroup obtained by deleting nodes $i_1$, \dots, $i_k$ from the Dynkin diagram.
\item The subgroups in (II) are easily determined via the possible Steinberg endomorphisms. 
\item In addition to those in (i) and (ii), for exceptional groups, the only other maximal subgroups we need detailed information about are the subgroups in (I) arising from reductive maximal rank subgroups $Y \leq X$, and our reference for these is \cite{ref:LiebeckSaxlSeitz92}.
\item For classical groups, one usually categorises maximal subgroups via Aschbacher's subgroup structure theorem \cite{ref:Aschbacher84}, where each maximal subgroup is either contained in one of eight geometric classes $\C_1$, \dots, $\C_8$ (with (II) and (III) broadly overlapping with $\C_5$ and $\C_6$, respectively) or is an absolutely irreducible almost simple group in the class $\S$. We follow Kleidman and Liebeck's notation in \cite{ref:KleidmanLiebeck} for geometric subgroups (including their definition of \emph{type}), and we refer to \cite[p.3]{ref:KleidmanLiebeck} for a definition of the class $\S$.
\end{enumerate}
\end{remark}

\subsection{Shintani descent} \label{ss:p_shintani}

We now introduce Shintani descent, following \cite{ref:Harper21}. Let $X$ be a connected algebraic group and let $\s_1$ and $\s_2$ be commuting Steinberg endomorphisms. For $\{ i,j \} = \{ 1,2 \}$, write $\ws_i = \s_i|_{X_{\s_j}}$ and assume that $\<\ws_i\> \cap X_{\s_j} = 1$. The \emph{Shintani map} of $(X,\s_1,\s_2)$, written $F\:X_{\s_1}\ws_2 \to X_{\s_2}\ws_1$, is 
\begin{gather*}
F\: \{ g^{X_{\s_1}} \mid g \in X_{\s_1}\ws_2 \} \to \{ h^{X_{\s_2}} \mid h \in X_{\s_2}\ws_1 \} \\
(x\ws_2)^{X_{\s_1}} \mapsto (y\ws_1)^{X_{\s_2}} \iff (x\s_2,\s_1)^X = (\s_2,y\s_1)^X. 
\end{gather*}
By \cite[Theorem~2.1]{ref:Harper21}, $F$ is a well-defined bijection and $C_{X_{\s_1}}(g) \cong C_{X_{\s_2}}(h)$ if $F(g^{X_{\s_1}}) = h^{X_{\s_2}}$. We frequently abuse notation and write $F$ on elements rather than conjugacy classes.

\begin{remark} \label{rem:shintani}
Let us collect together some basic properties of the Shintani map.
\begin{enumerate}
\item If $F$ is the Shintani map of $(X,\s_1,\s_2)$ and $F(g) = h$, then there exists $a \in X$ such that $g = aa^{-\s_2^{-1}}\ws_2$ and $h = a^{-1}a^{\s_1^{-1}}\ws_1$ (see \cite[Lemma~2.2]{ref:Harper21}).
\item Part~(i) implies that if $F$ is the Shintani map of $(X,\s^e,\s)$ for $e > 0$, then for all $g \in X_{\s^e}\ws$, the image $F(g)$ is $X$-conjugate to $g^{-e}$ (see \cite[Lemma~2.20]{ref:Harper21}). In particular, $|g| = e|F(g)|$.
\item If $F\:X_{\s_1}\ws_2 \mapsto X_{\s_2}\ws_1$ is the Shintani map of $(X,\s_1,\s_2)$, then define $F'\:X_{\s_1}\ws_2 \mapsto X_{\s_2}\ws_2^{-1}$ as $F'(g) = F(g)^{-1}$. In particular, if $F$ is the Shintani map of $(X,\s^e,\s)$, then $F'$ is what is called the Shintani map of $(X,\s,e)$ elsewhere (see \cite[Definition~3.18]{ref:BurnessGuralnickHarper21} for example).
\end{enumerate}
\end{remark}

Let us explain how the Shintani map relates to closed subgroups of $X$. This is studied comprehensively in \cite{ref:Harper21}, but apart from in a few instances, all we need is a simpler result. We have chosen to give the proof since it is short and enlightening.

\begin{theorem} \label{thm:shintani_subgroups}
Let $Y$ be a closed connected $\< \s_1, \s_2 \>$-stable subgroup of $X$. Let $g \in X_{\s_1}\ws_2$ and $h \in X_{\s_2}\ws_1$ with $F(g^{X_{\s_1}}) = h^{X_{\s_2}}$. Then $g$ is contained in an $\<X_{\s_1},\ws_2\>$-conjugate of $\<Y_{\s_1},\ws_2\>$ if and only if $h$ is contained in an $\<X_{\s_2},\ws_1\>$-conjugate of $\<Y_{\s_2},\ws_1\>$.
\end{theorem}

\begin{proof}
First note that $g \in X_{\s_1}\ws_2$ is contained in an $\<X_{\s_1},\ws_2\>$-conjugate of the subgroup $\<Y_{\s_1},\ws_2\>$ if and only if $g$ is contained in an $\<X_{\s_1},\ws_2\>$-conjugate of the coset $Y_{\s_1}\ws_2$, which, in turn, is true if and only if $g^{\<X_{\s_1},\ws_2\>} \cap Y_{\s_1}\ws_2 = g^{X_{\s_1}} \cap Y_{\s_1}\ws_2$ is nonempty. Similarly, $h \in X_{\s_2}\ws_1$ is contained in an $\<X_{\s_2},\ws_1\>$-conjugate of $\<Y_{\s_2},\ws_1\>$ if and only if $h^{X_{\s_2}} \cap Y_{\s_2}\ws_1$ is nonempty. Let $E$ be the Shintani map of $(Y,\s_1,\s_2)$. Observe that if $y \in Y_{\s_1}\ws_2$ and $z \in Y_{\s_2}\ws_1$ satisfy $E(y^{Y_{\s_1}}) = z^{Y_{\s_2}}$, then $F(y^{X_{\s_1}}) = z^{X_{\s_2}}$. In particular, if $g^{X_{\s_1}} \cap Y_{\s_1}\ws_2$ is nonempty, and contains $y^{Y_{\s_1}}$, say, then $h^{X_{\s_2}} \cap Y_{\s_2}\ws_1$ is nonempty since it contains $E(y^{Y_{\s_1}})$; the converse similarly holds. Therefore, $g^{X_{\s_1}} \cap Y_{\s_1}\ws_2$ is nonempty if and only $h^{X_{\s_2}} \cap Y_{\s_2}\ws_1$ is nonempty. The result follows.
\end{proof}

\begin{example} \label{ex:shintani_subgroups}
Let $X = \GL_n$ and let $Y$ be a maximal $P_k$ parabolic subgroup of $X$ for an integer $0 < k < n$ (that is, $Y$ is the stabiliser of a $k$-space of the natural module for $X$). Write $q=p^f$ and $q_0 = p^i$ where $i \div f$. Let $x \in \GL_n(q) = X_{\p^f}$ and $x_0 \in \GL_n(q_0) = X_{\p^i}$ satisfy $F((x\p^i)^{\GL_n(q)}) = {x_0}^{\GL_n(q_0)}$ where $F\: \GL_n(q)\p^i \to \GL_n(q_0)$ is the Shintani map of $(X,\p^f,\p^i)$. Then Theorem~\ref{thm:shintani_subgroups} implies that $x\p^i$ is contained in a $P_k$ parabolic subgroup of $\<\GL_n(q), \p \>$ if and only if $x_0$ is contained in a $P_k$ parabolic subgroup of $\GL_n(q_0)$. Applying this to all $0 < k < n$, we deduce that $x\p^i$ stabilises no subspace of $\F_q^n$ in the semilinear action of $\< \GL_n(q), \p \>$ on $\F_q^n$ if and only if $x_0$ acts irreducibly on $\F_{q_0}^n$.
\end{example}

The following more technical result \cite[Theorem~2.8(ii)]{ref:Harper21} will be useful in some proofs. Here, for an element $x \in X$ and a Steinberg endomorphism $\s$ of $X$, we write 
\begin{equation}
x_\s = xx^{-\s^{-1}}. \label{eq:lang}
\end{equation}
By the Lang--Steinberg theorem, every $s \in X$ may written as $s = x_\s$ for some $x \in X$.

\begin{theorem} \label{thm:shintani_cosets}
Let $Y$ be a closed $\< \s_1, \s_2 \>$-stable subgroup of $X$ and let $s,t \in N_X(Y^\circ)$. Assume that $N_{X_{s\s_1}}(Y^\circ_{s\s_1}) = N_X(Y^\circ)_{s\s_1}$, $N_{X_{t\s_2}}(Y^\circ_{t\s_2}) = N_X(Y^\circ)_{t\s_2}$ and $[s\s_1,t\s_2] = 1$. Let $g \in X_{\s_1}\ws_2$ and $h \in X_{\s_2}\ws_1$ with $F(g^{X_{\s_1}}) = h^{X_{\s_2}}$. Then the number of $\<X_{\s_1},\ws_2\>$-conjugates of $(Y^\circ_{s\ws_1} t\ws_2)^{s_{\s_1}}$ that contain $g$ equals the number of $\<X_{\s_2},\ws_1\>$-conjugates of $(Y^\circ_{t\ws_2} s\ws_1)^{t_{\s_2}}$ that contain $h$.
\end{theorem}

\subsection{Shintani descent and subfields} \label{ss:p_shintani_subfields}

Referring to the cases in Theorem~\ref{thm:maximal}, to see how Shintani descent relates to subgroups in case~(I) (the ``algebraic subgroups''), the key tool is Theorem~\ref{thm:shintani_subgroups}, which we introduced in the previous section. However, we will need new techniques to handle the subgroups in case~(II) (the ``subfield subgroups''). This is what we introduce in this section. (Other methods are available to study the subgroups in cases~(III)--(V).)

\begin{lemma} \label{lem:shintani_power}
Let $X$ be a connected algebraic group and let $\s_1$ and $\s_2$ be commuting Steinberg endomorphisms of $X$. Let $d \geq 1$. Let $F_1$ be the Shintani map of $(X,\s_1,\s_2)$ and let $F_2$ be the Shintani map of $(X,\s_1^d,\s_2)$. Let $x \in X_{\s_1}$. Then $F_2(x\ws_2) = F_1(x\ws_2)^d$.
\end{lemma}

\begin{proof}
Let $x = a_{\s_2}$. By Remark~\ref{rem:shintani}(i), $F_1(x\ws_2)^d = (a^{-1}a^{\s_1^{-1}}\ws_1)^d = a^{-1}a^{\s_1^{-d}}\ws_1^d = F_2(x\ws_2)$. 
\end{proof} 

\begin{theorem} \label{thm:shintani_subfields}
Let $X$ be a connected algebraic group and let $\s$ be a Steinberg endomorphism of $X$. Let $F$ be the Shintani map of $(X,\s^m,\s^l)$ where $l \div m$. Let $x \in X_{\s^m}$. Let $k$ be a prime divisor of $m$. Then $x\ws^l$ is contained in an $X_{\s^m}$-conjugate of $\<X_{\s^{m/k}},\ws\>$ if and only if $F(x\ws^l) = z^k$ for $z \in X_{\s^l}\ws^{m/k}$.
\end{theorem}

\begin{proof}
Let $E\:X_{\s^{m/k}}\ws^l \to X_{\s^l}\ws^{m/k}$ be the Shintani map of $(X,\s^{m/k},\s^l)$. 

For the forward direction, assume that $x\ws^l$ is contained in an $X_{\s^m}$-conjugate of $\<X_{\s^{m/k}},\ws\>$. Then $x\ws^l$ is $X_{\s^m}$-conjugate to $g\ws^l$ for some $g \in X_{\s^{m/k}}$. By Lemma~\ref{lem:shintani_power}, $F(x\ws^l) = F(g\ws^l) = E(g\ws^l)^k$ and $E(g\ws^l) \in X_{\s^l}\ws^{m/k}$ as required. 

For the reverse direction, assume that there exists $z \in X_{\s^l}\ws^{m/k}$ such that $z^k=F(x\ws^l)$. Let $g \in X_{\s^{m/k}}$ such that $E(g\ws^l) = z$. By Lemma~\ref{lem:shintani_power}, $F(g\ws^l) = E(g\ws^l)^k = z^k = F(x\ws^l)$, so $x\ws^l$ is $X_{\s^m}$-conjugate to $g\ws^l$, so $x\ws^l$ is contained in an $X_{\s^m}$-conjugate of $\< X_{\s^{m/k}}, \ws \>$. 
\end{proof}

In practice, we will use a corollary of Theorem~\ref{thm:shintani_subfields} that is slightly easier to apply. To obtain this corollary we first need the following lemma.

\begin{lemma} \label{lem:shintani_subfields}
Let $X$ be a connected algebraic group and let $\s$ be a Steinberg endomorphism of $X$. Let $e,k \geq 1$ be coprime. Let $y \in X_{\s^k}$. Then $y$ is $X$-conjugate to an element of $X_{\s}$ if and only if $y$ is $X$-conjugate to $z^k$ for an element $z \in X_{\s^k}\ws^e$.
\end{lemma}

\begin{proof}
Let $F\:X_{\s^k}\ws \to X_{\s}$ be the Shintani map of $(X,\s^k,\s)$.

For the forward direction, assume that $y$ is $X$-conjugate to $w \in X_{\s}$. Fix $a \geq 1$ such that $(a,|y|) = 1$ and $a \equiv e \mod{k}$ and fix $b \geq 1$ such that $ab \equiv 1 \mod{|y|}$. Then $y = y^{ab}$ is $X$-conjugate to $w^{ab} = (w^b)^a$, which, by Remark~\ref{rem:shintani}(ii), is $X$-conjugate to $(F^{-1}(w^b)^k)^a = (F^{-1}(w^b)^a)^k$, and $F^{-1}(w^b)^a \in X_{\s^k}\ws^a = X_{\s^k}\ws^e$.

For the reverse direction, assume that $y$ is $X$-conjugate to $z^k$ for some $z \in X_{\s^k}\ws^e$. Fix $c \geq 1$ such that $(c,|y|) = 1$ and $ce \equiv 1 \mod{k}$ and fix $d \geq 1$ such that $cd \equiv 1 \mod{|y|}$. Then $y = y^{cd}$ is $X$-conjugate to $z^{kcd} = (z^{ck})^d$. Now $z^c \in X_{\s^k}\ws^{ce} = X_{\s^k}\ws$, and by Remark~\ref{rem:shintani}(ii), $(z^{ck})^d$ is $X$-conjugate to $F(z^c)^d \in X_{\s}$. 
\end{proof}

We can now state the main result we use to relate Shintani descent and subfield subgroups.

\begin{corollary} \label{cor:shintani_subfields}
Let $X$ be a connected algebraic group and let $\s$ be a Steinberg endomorphism of $X$. Let $F$ be the Shintani map of $(X,\s^m,\s^l)$ where $l \div m$. Let $x \in X_{\s^m}$. Let $k$ be a prime divisor of $m$. 
\begin{enumerate}
\item If $k$ divides $m/l$, then $x\ws^l$ is contained in an $X_{\s^m}$-conjugate of $\<X_{\s^{m/k}},\ws\>$ if and only if $F(x\ws^l) = z^k$ for some $z \in X_{\s^l}$.
\item Assume $F(x\ws^l)^X \cap X_{\s^l} = F(x\ws^l)^{X_{\s^l}}$. If $k$ does not divide $m/l$, then $x\ws^l$ is contained in an $X_{\s^m}$-conjugate of $\<X_{\s^{m/k}},\ws\>$ if and only if $F(x\ws^l)$ is contained in an $X_{\s^l}$-conjugate of $X_{\s^{l/k}}$.
\end{enumerate}
\end{corollary}

\begin{proof}
If $k$ divides $m/l$, then $\ws^{m/k}=1$, so the result holds by Theorem~\ref{thm:shintani_subfields}. Now assume that $k$ does not divide $m/l$. By assumption, $F(x\ws^l)^X \cap X_{\s^l} = F(x\ws^l)^{X_{\s^l}}$, so the result holds by combining Theorem~\ref{thm:shintani_subfields} and Lemma~\ref{lem:shintani_subfields}.
\end{proof}

The following example illustrates the utility of Corollary~\ref{cor:shintani_subfields}.

\begin{example} \label{ex:shintani_subfields}
Let $X = \GL_n$ and write $q=p^f$ and $q_0=p^i$ where $i \div f$. Let $x \in \GL_n(q) = X_{\p^f}$ and $x_0 \in \GL_n(q_0) = X_{\p^i}$ satisfy $F((x\p^i)^{\GL_n(q)}) = {x_0}^{\GL_n(q_0)}$ where $F\: \GL_n(q)\p^i \to \GL_n(q_0)$ is the Shintani map of $(X,\p^f,\p^i)$. Since conjugacy in both $X = \GL_n$ and $X_{\p^i} = \GL_n(q_0)$ is determined by the rational canonical form, the condition $x_0^X \cap X_{\p^i} = {x_0}^{X_{\p^i}}$ is satisfied.

Assume that $f$ is even and $e=f/i$ is odd (so $i$ is even), and write $m=2f/i$ and $l=2$. We consider two examples.
\begin{enumerate}
\item Let $\s=\p^{i/2}$. Then $F$ is the Shintani map of $(X,\s^m,\s^l)$, and by Corollary~\ref{cor:shintani_subfields}(ii), $x\p^i$ is contained in a subgroup of $\< \GL_n(q), \p \>$ of type  $\< \GL_n(q^{1/2}), \p \>$ if and only if $x_0$ is contained in a subgroup of $\GL_n(q_0)$ of type $\GL_n(q_0^{1/2})$.
\item Let $\s = \g\p^{i/2}$ where $\g$ is the standard graph automorphism of $X$. Then, again, $F$ is the Shintani map of $(X,\s^m,\s^l)$, but, in this case, by Corollary~\ref{cor:shintani_subfields}(ii), $x\p^i$ is contained in a subgroup of $\< \GL_n(q), \p \>$ of type $\< \GU_n(q^{1/2}), \p \>$ if and only if $x_0$ is contained in a subgroup of $\GL_n(q_0)$ of type $\GU_n(q_0^{1/2})$.
\end{enumerate}
\end{example}

\subsection{Reductive maximal rank subgroups and the related finite subgroups} \label{ss:p_subgroups}

This final preliminary section is dedicated to the reductive maximal rank subgroups that arise in part~(I) of Theorem~\ref{thm:maximal}, which will play a particular role in our proofs. 

Let $X$ be a simple algebraic group and let $\s$ be a Steinberg endomorphism of $X$. For $x \in X$, we will regularly use the notation $x_\s =  xx^{-\s^{-1}}$ introduced in \eqref{eq:lang}.

Let $S$ be a $\s$-stable maximally split maximal torus of $X$ and let $Y$ be a closed connected reductive $\s$-stable subgroup of $X$ containing $S$. Let $x \in X$ such that $Y^x$ is $\s$-stable. Since $Y^x$ contains a $\s$-stable maximal torus, without loss of generality, we may assume that $S^x$ is also $\s$-stable. The possibilities for $(Y^x)_\s$ up to $X_\s$-conjugacy are in bijection with the conjugacy classes in the coset $(N_X(Y)/Y)\s$, via $(Y^x)_\s \mapsto Y x_\s\s$ (see \cite[Theorem~21.11]{ref:MalleTesterman11} for example).

The following easy observation will be useful.

\begin{lemma} \label{lem:subgroups_power}
Let $X$ be a connected algebraic group and let $\s$ be a Steinberg endomorphism of $X$. Let $Y \leq X$ be a closed $\s$-stable subgroup, let $x \in X$ and let $e \geq 0$. Then $(Y^x)_{\s^e} = Y_{(x_\s\s)^e}^x$.
\end{lemma}

\begin{proof}
Here $(Y^x)_{\s^e} = Y_{\psi}^x$ for $\psi = (\s^e)^{x^{-1}} = (\s^{x^{-1}})^e =  (x_\s\s)^e$.
\end{proof}

We can now present a technical result that we use in the proof of Proposition~\ref{prop:standard}.

\begin{lemma} \label{lem:subgroups_closed}
Let $X$ be a simple algebraic group, let $\s$ be a Steinberg endomorphism of $X$, let $M$ be a closed connected reductive $\s$-stable maximal rank subgroup of $X$ and let $e$ be a positive integer. Assume that one of the following holds:
\begin{enumerate}
\item $X \in \{ G_2, F_4, E_6, E_7, E_8 \}$ and $\s = \p^i$ 
\item $X \in \{ B_2, G_2, F_4 \}$,  $\s = \rho^i$ (odd $i$) and $2 \ndiv e$
\item $X = E_6$, $\s = \g\p^i$ and $2 \ndiv e$
\item $X = D_4$, $\s = \t\p^i$   and $3 \ndiv e$.
\end{enumerate}
If $N_{X_\s}(M_\s)$ is a maximal subgroup of $X_\s$, then $N_{X_{\s^e}}(M_{\s^e})$ is a maximal subgroup of $X_{\s^e}$.
\end{lemma}

Before we prove Lemma~\ref{lem:subgroups_closed}, let us make a modification to the above setup that will make our calculations easier (both in the following proof and later in the paper). 

Fix the Weyl groups $W_X = N_X(S)/S$ and $W_Y = N_Y(S)/S$, and for $n \in N_X(S)$, write $\bar{n} = Sn \in W_X$. Write $A =  N_{W_X}(W_Y)/W_Y$ and note that $N_X(Y)/Y \cong A$. The endomorphism $\s$ induces an automorphism $\alpha = \s|_A$ on $A$, and by \cite[Corollary~3]{ref:Carter78}, the possibilities for $(Y^x)_\s$ up to $X_\s$-conjugacy are also in bijection with the conjugacy classes in the coset $A\a$, via $(Y^x)_\s \mapsto W_Y\bar{x}_\s\a$. Let us write
\[
\mathcal{A}_\s = \{ W_Y\bar{x}_\s \in A \mid \text{$N_{X_\s}((Y^x)_\s)$ is maximal in $X_\s$} \}.
\]

We are now in a position to prove Lemma~\ref{lem:subgroups_closed}.

\begin{proof}[Proof of Lemma~\ref{lem:subgroups_closed}]
Let $S$ be a $\s$-stable maximally split maximal torus of $X$, and let $Y$ be a closed connected reductive $\s$-stable subgroup of $X$ containing $S$. Let $x \in X$ such that $M = Y^x$ and $S^x$ are $\s$-stable. We will prove that if $N_{X_\s}((Y^x)_\s)$ is maximal in $X_\s$, then $N_{X_{\s^e}}((Y^x)_{\s^e})$ is maximal in $X_{\s^e}$. From Lemma~\ref{lem:subgroups_power}, recall that $N_{X_\s}((Y^x)_\s) = N_{X_\s}(Y_{x_\s\s}^x)$ and $N_{X_{\s^e}}((Y^x)_{\s^e}) = N_{X_{\s^e}}(Y_{(x_\s\s)^e}^x)$. We will prove that if $a \in \mathcal{A}_\s \alpha$, then $a^e \in \mathcal{A}_{\s^e}\alpha^e$, which establishes the claim.

First consider part~(i). Here $\s$ centralises $W_X$, so $\alpha = 1$ and we need to verify that if $a \in A_\s$ then $a^e \in A_{\s^e}$. This is easily deduced by inspecting \cite[Tables~5.1 \&~5.2]{ref:LiebeckSaxlSeitz92}. For example, let us consider in detail the case where $X = E_8$. Then either $Y$ yields no maximal subgroups of $X_\s$ (so $A_\s = \emptyset$) or all the possible types of subgroup arising from $Y$ are maximal in $X_\s$ and $X_{\s^e}$ (so $A_\s = A_{\s^e} = A$) or, inspecting the proof of \cite[Lemma~2.5]{ref:LiebeckSaxlSeitz92}, $Y \in \{ D_4^2, A_2^4, A_1^8, S \}$. When $Y$ is $D_4^2$, $A_2^4$ or $A_1^8$, we see that $A_\s = A_{\s^e}$ is a proper subgroup of $A$ (we have $C_6 < \Sym_3 \times C_2$, $C_8 < \GL_2(3)$, $1 < \mathrm{AGL_3(2)}$, respectively). 

Finally assume that $Y = S$. Here $A = W_{E_8}$ and, writing $I = \{ 1, 2, 3, 4, 5, 6, 10, 12, 15, 30 \}$, for each $i \in I$ we may fix $a_i \in A$ satisfying $|a_i| = i$ and $\{ a_i \mid i \in I \}^A = (\< a_{30} \> \cup \< a_{12} \>)^A$ such that we have
\[
A_{\p^j} = \left\{ 
\begin{array}{ll}
\{ a_i \mid i \in I \setminus \{ 1,6 \} \}^A & \text{if $(p,j) = (2,1)$} \\
\{ a_i \mid i \in I \setminus \{ 1 \} \}^A   & \text{if $(p,j) \in \{(2,2), (3,1)\}$} \\
\{ a_i \mid i \in I \}^A                     & \text{otherwise}.
\end{array} \right.
\]
In all cases, if $a \in A_\s$, then $a^e \in A_{\s^e}$. 

Now consider parts~(ii)--(iv). These are all similar, so we just give the details for part~(ii). Here $X \in \{B_2,G_2,F_4\}$ and $\rho$ (so $\s$ and $\s^e$) induces a nontrivial involution $\alpha$ on $W_X$. We will verify that if $a \in A_\s \alpha$ then $a^e \in A_{\s^e} \alpha$ for all odd $e$. First assume $X = B_2$. By \cite[Table~5.1]{ref:LiebeckSaxlSeitz92}, $Y = S$ and, by \cite[Table~5.2]{ref:LiebeckSaxlSeitz92}, $A_\s = A_{\s^e} = A$, so the result holds.

Next assume $X = G_2$. By \cite[Table~5.1]{ref:LiebeckSaxlSeitz92}, $Y = S$, so $A = W_{G_2}$. Let us consider the dihedral group $D_{24} = \< s, t \mid s^{12}, s^2, s^t = s^{-1} \>$. Then $A = \< s^2, t \> = D_{12}$ and $\alpha$ induces $st$, so $A\alpha = D_{24} \setminus D_{12}$. For $s = u\alpha$, it is easy to check that $\{ |S_{u^j\rho^i}| \mid \text{$j$ odd} \} = \{ q_0+1, q_0 \pm \sqrt{3q_0} + 1 \}$. Therefore, by \cite[Table~5.2]{ref:LiebeckSaxlSeitz92}, $A_\s\alpha = A_{\s^e}\alpha = \{ (u\alpha)^j \mid \text{$j$ odd} \}$, so if $a \in A_\s \alpha$ then $a^e \in A_{\s^e} \alpha$.

Finally assume that $X = F_4$. For now also assume that $Y \neq S$. Here, by \cite[Table~5.1]{ref:LiebeckSaxlSeitz92}, $Y \in \{A_2\tilde{A}_2, B_2^2\}$ and $A_\s = A_{\s^e} = A$, so the result holds. It remains to assume that $Y = S$, where $A = W_{F_4}$. Fix $u \in A$ such that $|u\alpha| = 4$ and $\{ |S_{u^j\rho^i}| \mid \text{$j$ odd} \} = \{ (q_0+1)^2 \}$ and $v \in A$ such that $|v\alpha| = 24$ and $\{ |S_{v^j\rho^i}| \mid \text{$j$ odd} \} = \{ q_0^2 \pm \sqrt{2}q_0^{3/2} + q_0 \pm \sqrt{2q_0} + 1, (q_0 \pm \sqrt{2q_0} + 1)^2 \}$. Then, by \cite[Table~5.2]{ref:LiebeckSaxlSeitz92}, $A_\s\alpha = A_{\s^e}\alpha = \{ (u\alpha)^j, (v\alpha)^j \mid \text{$j$ odd} \}$, so $a^e \in A_{\s^e} \alpha$ if $a \in A_\s \alpha$. The proof is complete.
\end{proof}

\section{Totally deranged elements} \label{s:proof}

We now turn to the proof of Theorem~\ref{thm:derangement}. Sections~\ref{ss:proof_alternating_sporadic}--\ref{ss:proof_ade} prove that the only possible examples appear in Theorem~\ref{thm:derangement}, and Section~\ref{ss:proof_examples_proof} proves that the examples in Theorem~\ref{thm:derangement} really are totally deranged. We formally complete the proof in Section~\ref{ss:proof}.

\subsection{Alternating, sporadic and small groups} \label{ss:proof_alternating_sporadic}

We begin by reducing the proof of Theorem~\ref{thm:derangement} to the almost simple groups of Lie type. We also take this opportunity to handle some small groups of Lie type.

\begin{proposition} \label{prop:computational}
Let $G$ be an almost simple group whose socle is alternating, sporadic, the Tits group or one of the following groups of Lie type
\begin{gather*}
\PSL_2(q) \: q \leq 9, \ \ \ 
\PSL_3(q) \: q \leq 4, \ \ \ 
\PSL_4(q) \: q \leq 3, \ \ \
\PSU_3(q) \: q \leq 4, \ \ \
\PSU_5(2), \ \ \ G_2(3).
\end{gather*}
Then $G$ has no totally deranged elements.
\end{proposition}

\begin{proof}
Let $x \in G$. First assume that $G$ is simple. Then every maximal subgroup is core-free, so $x$ is certainly contained in a core-free maximal subgroup and is thus not totally deranged. 

Next assume that $G = \Sym_n$ with $n \geq 5$. If $x$ stabilises a subset of size $k$, for some $0 < k < n$, then $x$ is contained in a $G$-conjugate of $\Sym_k \times \Sym_{n-k}$ if $n \neq 2k$ or $\Sym_k \wr \Sym_2$ if $n=2k$. If $x$ is an $n$-cycle, then $x$ is contained in a $G$-conjugate of $\Sym_a \wr \Sym_b$ if $n=ab$ with $a,b > 1$ or $\mathrm{AGL}_1(n)$ if $n$ is prime. Each of these overgroups is a maximal subgroup of $G$ (see the main theorem of \cite{ref:LiebeckPraegerSaxl87}), so, in all cases, $x$ is contained in a core-free maximal subgroup and we deduce that $x$ is not totally deranged.

It remains to consider the nonsimple sporadic groups, the remaining extensions of $\Alt_6$, the almost simple group ${}^2F_4(2)$ and the nonsimple extensions of the groups displayed in the statement. In these cases, it is straightforward to verify computationally that every element of $G$ is contained in a core-free maximal subgroup and is thus not totally deranged. We handle the sporadic groups using the Character Table Library \cite{ref:CTblLib} in $\textsf{GAP}$ \cite{ref:GAP}, but the remaining groups can be computed with directly in \textsc{Magma} \cite{ref:Magma}. The code is in \cite{ref:GithubTotallyDeranged}.
\end{proof}

\subsection{Eliminating most possibilities}  \label{ss:proof_standard}

Turning to the groups of Lie type, we begin with Propositions~\ref{prop:standard} and~\ref{prop:graph}, which eliminate all but a few families of groups. The cases not covered here require more involved arguments (see Sections~\ref{ss:proof_bfg} and~\ref{ss:proof_ade}). However, we first record a small generalisation of \cite[Lemma~6.4]{ref:BurnessHarper20}.

\begin{lemma} \label{lem:parabolic}
Let $T$ be a finite simple group of Lie type, let $T \leq G \leq \Inndiag(T)$ and let $x \in G$. If $x$ is not contained in a parabolic subgroup of $G$, then $x$ is a regular semisimple element.
\end{lemma}

\begin{proof}
We prove the contrapositive. Assume $x$ is not regular semisimple. Not being regular semisimple is equivalent to commuting with a unipotent element (see \cite[Corollary~E.III.1.4]{ref:Borel70}, for example). Therefore, fix a unipotent element $u \in G$ such that $x \in C_G(u)$. Now $C_G(u)$ is contained in $N_G(\<u\>)$, which, by the Borel--Tits theorem for finite groups of Lie type (see \cite[Theorem~26.5]{ref:MalleTesterman11}, for example), is contained in a parabolic subgroup of $G$.
\end{proof}

%\begin{lemma} \label{lem:maximal}
%Let $G$ be a finite group, let $N \leqn G$ and let $H \leq G$. Assume that $H \cap N$ is a maximal subgroup of $N$ and $|H:H \cap N| = |G:N|$. Then $H$ is a maximal subgroup of $G$.
%\end{lemma}
%
%\begin{proof}
%Let $H \leq M \leq G$. Then $G = MN$ as $|MN:N| \geq |HN:N| = |H:H \cap N| = |G:N|$. Also $H \cap N \leq M \cap N \leq N$, so $M \cap N = H \cap N$ or $M \cap N = N$. In the former case, $|H:H \cap N| = |G:N| = |MN:N| = |M:M \cap N| = |M:H \cap N|$, so $M = H$. In the latter case, $N \leq M$, so $G = MN = M$. Therefore, $H$ is maximal in $G$.
%\end{proof}

\begin{proposition} \label{prop:standard}
Let $T$ be a finite simple group of Lie type, let $T \leq G \leq \Aut(T)$ and let $x \in G$. Assume that one of the following holds:
\begin{enumerate}
\item $T$ is untwisted or very twisted, and $G \leq \< \Inndiag(T), \p \>$
\item $T \in \{ A^\e_m(p^f) \, \text{($m \geq 2$)}, \ D^\e_m(p^f) \, \text{($m \geq 4$)}, \ E^\e_6(p^f) \}$ and $x \in \Inndiag(T)\g\p^j$ for $\e = (-)^{f/j}$, where $G$ does not contain triality if $T = D_4(p^f)$
\item $T = D_4(p^f)$ and $x \in \Inndiag(T)\t\p^j$ with $3 \div \frac{f}{j}$, or $T = {}^3D_4(p^f)$ and $x \in T\t\p^j$ with $3 \ndiv \frac{f}{j}$.
\end{enumerate}
Then $x$ is not totally deranged.
\end{proposition}

\begin{proof}
We may assume that $G$ is not in Proposition~\ref{prop:computational}. We will first identify a core-free subgroup $H \leq G$ that contains $x$, and then we will prove that $H$ is a maximal subgroup of $G$. If $x \in \Inndiag(T)$, then we have some control over the maximal overgroups of $x$, but for the general case, we will apply Shintani descent (see Section~\ref{ss:p_shintani}), which gives an element $x_0 \in \Inndiag(T_0)$ for a related group $T_0$, which we use to inform us about our original element $x$ (in particular, $T_0 = T$ and $\<x_0\> = \<x\>$ if $x \in \Inndiag(T)$, see Remark~\ref{rem:shintani}(ii)). 

Let us establish some notation. We will fix a simple algebraic group $X$ and two Steinberg endomorphisms $\s_1$ and $\s_2$ such that $X_{\s_1} = \Inndiag(T)$ and $x \in \Inndiag(T)\ws_2$. The specific choices for $(X,\s_1,\s_2)$ are given in Table~\ref{tab:standard}. (We will define $T_0$ below.)

\begin{table}[b]
\[
\begin{array}{ccccccc}
\hline 
\text{row} & X                 & \s_1       & \s_2        & T        & T_0        & \text{conditions} \\
\hline
1          & \text{any}        & \p^f       & \p^j        & X(q)     & X(q_0)     & \text{none}       \\
2          & B_2,\, F_4,\, G_2 & \rho^f     & \rho^j      & {}^2X(q) & {}^2X(q_0) & \text{$f$ odd}    \\
3          & A_m,\, D_m,\, E_6 & \g\p^f     & \g\p^j      & {}^2X(q) & {}^2X(q_0) & \text{$e$ odd}    \\
4          & A_m,\, D_m,\, E_6 & \p^f       & \g\p^j      & X(q)     & {}^2X(q_0) & \text{$e$ even}   \\
5          & D_4               & \t\p^f     & \t^e\p^j    & {}^3X(q) & {}^3X(q_0) & 3 \ndiv e         \\
6          & D_4               & \p^f       & \t\p^j      & X(q)     & {}^3X(q_0) & 3 \div e          \\
\hline
\end{array}
\]
\caption{Notation for the proof of Proposition~\ref{prop:standard} (here $q=p^f$, $q_0 = p^j$, $j$ divides $f$, $e=f/j$).}\label{tab:standard}
\end{table}

We claim that Table~\ref{tab:standard} exhausts all possibilities in the statement. To see this, first note that Rows~3 to~6 cover parts~(ii) and~(iii). Now consider part~(i). If $T$ is untwisted, then by replacing $x$ by another generator of $\< x \>$ if necessary, we may assume that $x \in \Inndiag(T)\wp^j$ where $j$ divides $f$, and $\wp^j = \ws_2$ in Row~1. Similarly, if $T$ is very twisted, then since $\<\wp^j\> = \<\wrho^j\>$, we may assume that $x \in \Inndiag(T)\wrho^j$ where $j$ divides $f$, and $ \wrho^j = \ws_2$ in Row~2.

Let $F\: X_{\s_1}\ws_2 \to X_{\s_2}\ws_1$ be the Shintani map of $(X,\s_1,\s_2)$. Since $\s_1 = \s_2^{f/j}$, we have $\ws_1 = 1$. Write $T_0 = X_{\s_2}$ and $F\: \Inndiag(T)\ws_2 \to \Inndiag(T_0)$. Observe that, except in Rows~4 and~6, $T_0$ has the same type as $T$ except it is defined over the subfield $\F_{q_0}$ of $\F_q$. Noting that $x \in \Inndiag(T)\ws_2$, define $x_0 = F(x) \in \Inndiag(T_0)$. We find a core-free maximal subgroup $H$ of $G$ containing $x$ by first considering subgroups $H_0$ of $G_0 = \Inndiag(T_0)$ containing $x_0$.

\emph{\textbf{Case~1.} $x_0$ is contained in a parabolic subgroup of $G_0$.}\nopagebreak

Let $H_0$ be a maximal parabolic subgroup of $G_0$ containing $x_0$. Then $H_0 = Y_{\s_2}$ for a maximal $\p$-stable parabolic subgroup of $Y \leq X$. Theorem~\ref{thm:shintani_subgroups} implies that $x$ is contained in an $\< X_{\s_1}, \ws_2 \>$-conjugate of $\< Y_{\s_1}, \ws_2 \>$. Let $H = N_G(Y_{\s_1})$, so $x \in H$. Since $Y$ is $\p$-stable, $\wp$ normalises $Y_{\s_1}$. In Rows~1 and~2, $|H:H \cap \Inndiag(T)| = |H:H \cap Y_{\s_1}| = f/i = |G:{G \cap \Inndiag(T)}|$, where $\< \Inndiag(T), G \> = \< \Inndiag(T), \wp^i \>$. Hence, $H$ is maximal in $G$ as $H \cap \Inndiag(T)$ is maximal in $\Inndiag(T)$. In Rows~3--6, $X \in \{A_m, D_m, E_6\}$ and $T_0$ is twisted, so $Y$ is actually $\<\p,\g\>$-stable, so again $|H: H \cap Y_{\s_1}| = |G: G \cap \Inndiag(T)|$ and $H = N_G(Y_{\s_1})$ is maximal.

\vspace{0.3\baselineskip}

\emph{\textbf{Case~2.} $x_0$ is not contained in a parabolic subgroup of $G_0$.}\nopagebreak

By Lemma~\ref{lem:parabolic}, $x_0$ is a (regular) semisimple element, so it is contained in a maximal torus of $G_0$ and hence in $H_0 = Y_{\s_2}$ for a connected reductive maximal rank subgroup $Y \leq X$. We may assume that $Y$ is maximal among such subgroups, so $N_{G_0}(H_0)$ is a maximal subgroup of $G_0$. If $X$ is classical, then to simplify our maximality arguments later in the proof we will exhibit a specific choice of $H_0$. (Later we will see why it is useful to record the type of $Y$).

\emph{\textbf{Case~2a.} $T_0 = \PSL_n(q_0)$.}\nopagebreak

Since $x_0$ is not in a parabolic subgroup, $x_0$ is irreducible, so is contained in a field extension subgroup $H_0$ of type $\GL_{n/k}(q_0^k)$ for the least prime $k$ dividing $n$ (so $Y$ has type $\GL_{n/k}^k$).

\emph{\textbf{Case~2b.} $T_0 = \PSU_n(q_0)$.}\nopagebreak

First assume that $x_0$ is irreducible. Then $n$ is odd and $x_0 \in H_0$ of type $\GU_{n/k}(q_0^k)$ for the least prime $k$ dividing $n$ (so $Y$ has type $\GL_{n/k}^k$). Now assume that $x_0$ is reducible. Then $x_0$ stabilises a $k$-space $U$ for some $0 < k \leq n/2$. Since $x_0$ is in no parabolic subgroup, $U$ must be nondegenerate. Since $x_0$ is semisimple, we may assume that $x_0$ acts irreducibly on $U$, which implies that $U$ has odd dimension $k$, so $x_0 \in H_0$ of type $\GU_k(q_0) \times \GU_{n-k}(q_0)$ (so $Y$ has type $\GL_k \times \GL_{n-k}$, perhaps with $k=n-k$).

\emph{\textbf{Case~2c.} $T_0 = \PSp_n(q_0)$.}\nopagebreak

If $q_0$ is even, then $x_0 \in H_0$ of type $\Omega^+_n(q_0)$ or $\Omega^-_n(q_0)$ (so $Y$ has type $\Omega_n$). Now assume that $q_0$ is odd. Here we proceed as in the previous case. First assume that $x_0$ is irreducible. Then $x_0 \in H_0$ of type $\Sp_{n/k}(q_0^k)$ for the least prime $k$ dividing $n/2$ (so $Y$ has type $\Sp_{n/k}^k$). Now assume that $x_0$ is reducible. Then $x_0$ stabilises a $k$-space $U$ for some $0 < k \leq n/2$. Since $x_0$ is in no parabolic subgroup, $U$ must be nondegenerate, so $x_0 \in H_0$ of type $\Sp_k(q_0) \times \Sp_{n-k}(q_0)$ (so $Y$ has type $\Sp_k \times \Sp_{n-k}$, perhaps with $k=n-k$).

\emph{\textbf{Case~2d.} $T_0 = \POm^\e_n(q_0)$ for some $\e \in \{+,\circ,-\}$.}\nopagebreak

First assume that $x_0$ is irreducible. In this case, $T_0 = \POm^-_{2m}(q_0)$ and $x_0 \in H_0$ of type $\SO^-_m(q_0^2)$ if $m$ is even (so $Y$ has type $\Omega_m^2$) or $\GU_m(q_0)$ if $m$ is odd (so $Y$ has type $\GL_m^2$). Now assume that $x_0$ is reducible. Since $x_0$ is in no parabolic subgroup, $x_0$ stabilises a proper nonzero nondegenerate subspace $U$. Since $x_0$ is semisimple, we may assume that $x_0$ acts irreducibly on $U$, which implies that $U$ is a minus-type space of even dimension $k$, so $x_0 \in H_0$ of type $\SO^-_k(q_0) \times \SO^{-\e}_{n-k}(q_0)$ (so $Y$ has type $\Omega_k \times \Omega_{n-k}$, perhaps with $k=n-k$).

\emph{\textbf{Case~2e.} $T_0 = {}^3D_4(q_0)$.}\nopagebreak

Let $S_0$ be a maximal torus of $T_0$ that contains $x_0$. If $|S_0| \in \{ (q_0^2 \pm q_0 + 1)^2, q_0^4-q_0^2+1 \}$, then $x_0 \in H_0 = S_0$ (so $Y$ is a maximal torus). Otherwise, consulting \cite[Table~1.1]{ref:DeriziotisMichler87} for example, since $x_0$ is not contained in a parabolic subgroup of $T_0$, we have $|S_0| = (q_0+1)(q_0^3+1)$, which means that $x_0 \in H_0$ of type $\SL_2(q) \times \SL_2(q^3)$ (so $Y$ has type $\SL_2^4$).

\vspace{0.3\baselineskip}

In summary, $x_0$ is contained in $H_0 = Y_{\s_2}$ for a closed connected $\s_2$-stable subgroup $Y \leq X$, so Theorem~\ref{thm:shintani_subgroups} implies that $x$ is contained in an $\<X_{\s_1}, \ws_2\>$-conjugate of $\<Y_{\s_1},\ws_2\>$. Let $H = N_G(Y_{\s_1})$. We now show that $H$ is maximal, except for one case that we can easily handle. In all cases, if $n \leq 12$, then to verify the claims on maximality we refer the reader to the relevant table in \cite[Chapter~8]{ref:BrayHoltRoneyDougal} rather than to the stated reference in \cite[Chapter~3]{ref:KleidmanLiebeck}.

\emph{\textbf{Case~2A.} $T = \PSU_n(q)$.}\nopagebreak

Here $T_0 = \PSU_n(q_0)$ and $Y$ has type $\GL_k \times \GL_{n-k}$ or $\GL_{n/k}^k$ ($n$ odd and $k$ prime), so $H$ is maximal \cite[Table~3.5B]{ref:KleidmanLiebeck}. 

\emph{\textbf{Case~2B.} $T = \PSp_n(q)$.}\nopagebreak

Here $T_0 = \PSp_n(q_0)$ and $Y$ has type $\Omega_n$ if $q$ is even and type $\Sp_k \times \Sp_{n-k}$ or $\Sp_{n/k}^k$ ($k$ prime) if $q$ is odd, so $H$ is maximal \cite[Table~3.5C]{ref:KleidmanLiebeck}. 

\emph{\textbf{Case~2C.} $T = \Omega_n(q)$ for odd $n$.}\nopagebreak 

Here $T_0 = \Omega_n(q_0)$ and $Y$ has type $\Omega_k \times \Omega_{n-k}$, so $H$ is maximal \cite[Table~3.5D]{ref:KleidmanLiebeck} (referring to the exceptional case in that table, if $q=3$, then $\s_1 = \s_2 = \p$, so $Y_{\s_1} = Y_{\s_2} = H_0$, which has type $\SO_k^-(3) \times \SO_{n-k}(3)$, see Case~2d).

\emph{\textbf{Case~2D.} $T = \POm^-_n(q)$.}\nopagebreak

Here $T_0 = \POm^-_n(q_0)$ and $Y$ has type $\GL_{n/2}^2$ ($n/2$ odd) or $\Omega_k \times \Omega_{n-k}$ ($k$ even), so $H$ is maximal \cite[Table~3.5F]{ref:KleidmanLiebeck} (if $q \leq 3$, then $H$ does not have type $\SO^-_{n-2}(q) \times \SO^+_2(q)$ as above). 

\emph{\textbf{Case~2E.} $T = \PSL_n(q)$.}\nopagebreak 

Here there are two possibilities, corresponding to Rows~1 and~4 of Table~\ref{tab:standard}, namely $T_0 = \PSL_n(q_0)$ (with $\s_2 = \p^j$) and $T_0 = \PSU_n(q_0)$ (with $\s_2 = \g\p^j$ and $f/j > 0$ even). Inspecting Cases~2a and~2b above, we see that we can divide into two subcases.

First assume that $Y$ has type $\GL_{n/k}^k$  where $k$ is the least prime dividing $n$. Now $H$ is maximal \cite[Table~3.5A]{ref:KleidmanLiebeck} (if $q \leq 3$, then $H$ does not have type ${\GL_{n/k}(q) \wr \Sym_k}$), except that if $q=4$ and $n$ is prime, then when $x \not\in \Inndiag(T)$ the subgroup $H$ could have type $\GL_1(4) \wr \Sym_n$ and not be maximal. Here, if $H$ is not maximal, then it is contained in a subgroup $\widetilde{H}$ of type $\PSU_n(2)$ (see \cite[Table~3.5H]{ref:KleidmanLiebeck} and \cite[Proposition~2.3.6]{ref:BrayHoltRoneyDougal}), which is maximal (see \cite[Table~3.5A]{ref:KleidmanLiebeck}), so we replace $H$ with $\widetilde{H}$ in this case.

We may now assume that $Y$ has type $\GL_k \times \GL_{n-k}$ where $k < n$ and $T_0 = \PSU_n(q_0)$. As $T_0 = \PSU_n(q_0)$, we know that $T\ws_2 = T\g\wp^j \in G/T$, so $G \not\leq \< \Inndiag(T), \wp\>$. Therefore, inspecting \cite[Tables~3.5A \&~3.5H]{ref:KleidmanLiebeck}, we see that $H$ is maximal.

\emph{\textbf{Case~2F.} $T = \POm^+_n(q)$.}\nopagebreak

Here there are three possibilities, corresponding to Rows~1, 4 and~6 of Table~\ref{tab:standard}, namely $T_0 = \POm^+_n(q_0)$ (with $\s_2 = \p^j$), $T_0 = \POm^-_n(q_0)$ (with $\s_2 = \g\p^j$ and $f/j$ even) and $T_0 = {}^3D_4(q_0)$ (with $\s_2 = \t\p^j$ and $e = f/j$ divisible by $3$). Assume for now that $T_0 \neq {}^3D_4(q_0)$. Inspecting Case~2d above, we see that we can divide into two subcases.

First assume $Y$ has type $\Omega_k \times \Omega_{n-k}$ where $k < n$ is even, so $H$ is maximal \cite[Table~3.5E]{ref:KleidmanLiebeck}. 

We may now assume that $Y$ has type $\Omega_{n/2}^2$ where $n/2$ is even or $\GL_{n/2}^2$ where $n/2$ is odd and, in both cases, $T_0 = \POm^-_n(q_0)$. As $T_0 = \POm^-_n(q_0)$, we know that $T\ws_2 = T\g\wp^j \in G/T$, so $G \not\leq \< \Inndiag(T), \wp\>$. Moreover, we know that $f/j$ is even, so, by Lemma~\ref{lem:subgroups_power}, $H$ has type $\SO^+_{n/2}(q) \wr \Sym_2$ or $\GL_{n/2}(q).2$. By consulting \cite[Tables~3.5E \&~3.5H]{ref:KleidmanLiebeck}, we see that $H$ is maximal. 

It remains to assume that $T_0 = {}^3D_4(q_0)$. From Case~2e above, we know that $Y$ has type $\SL_2^4$ or is a maximal torus. In the former case, $H$ has type $\SL_2(q)^4$. For the latter case, let $S$ be a maximally split $\s_2$-stable maximal torus of $X$ and write $Y = S^y$. Seeking to apply the theory from Section~\ref{ss:p_subgroups}, let $A = N_{W_X}(W_Y)/W_Y$ and let $\a$ be the automorphism of $A$ induced by $\s_2$. Since $|S_{y_{\s_2}\s_2}| = |Y_{\s_2}| \in \{ (q_0^2 \pm q_0 + 1)^2, q_0^4-q_0^2+1 \}$, we deduce that $W_Y\bar{y}_{\s_2}\a \in A\alpha$ has order $3$, $6$ or $12$ (see \cite[Section~7.5]{ref:Gager73} for example). Now $3$ divides $e = f/j$, so $(W_Y\bar{y}_{\s_2}\a)^e \in A$ has order $1$, $2$ or $4$. By Lemma~\ref{lem:subgroups_power}, $|Y_{\s_1}| = |Y_{\s_2^e}| = |S_{(y_{\s_2}\s_2)^e}| \in \{ (q \pm 1)^4, (q^2+1)^2 \}$. In all cases, by \cite[Table~8.50]{ref:BrayHoltRoneyDougal}, $H$ is maximal since $G \not\leq \< \Inndiag(T), \g, \wp \>$.

\emph{\textbf{Case~2G.} $T$ is exceptional.}\nopagebreak 

For now exclude the case in Row~4 of Table~\ref{tab:standard}. Recall that we write $H_0 = Y_{\s_2}$  and we know that $N_{G_0}(H_0)$ is a maximal subgroup of $G_0 = \Inndiag(T_0) = X_{\s_2}$. Applying Lemma~\ref{lem:subgroups_closed}, we deduce that $N_{\Inndiag(T)}(Y_{\s_1})$ is a maximal subgroup of $\Inndiag(T) = X_{\s_1}$. Now consulting \cite[Tables~5.1 \&~5.2]{ref:LiebeckSaxlSeitz92}, we deduce that $H = N_G(Y_{\s_1})$ is a maximal subgroup of $G$. 

It remains to consider $T = E_6(q)$ and $T_0 = {}^2E_6(q_0)$. Arguing as in the proof of Lemma~\ref{lem:subgroups_closed}, using \cite[Tables~5.1 \&~5.2]{ref:LiebeckSaxlSeitz92}, it is easy to check that since $N_{G_0}(H_0)$ is a maximal subgroup of $G_0$, we also have that $H = N_G(Y_{\s_1})$ is a maximal subgroup of $G$. Here, the only subtlety is that when $Y$ has type $D_5$, while $H_0 = Y_{\s_2} = \POm^-_{10}(q_0) \times (q_0+1)$ is maximal in $\Inndiag(E_6(q_0)) = X_{\s_2}$, the subgroup $Y_{\s_1} = \POm^+_{10}(q) \times (q-1)$ is not maximal in $\Inndiag(E_6(q)) = X_{\s_1}$, but $H = N_G(Y_{\s_1})$ is maximal in $G$ since $T\ws_2 = T\g\wp^j \in G/T$ and hence $G \not\leq \< \Inndiag(T), \wp \>$.

In all cases, we have identified a core-free maximal subgroup $H$ of $G$ that contains $x$, so $x$ is not a totally deranged element of $G$.
\end{proof}

\begin{proposition} \label{prop:graph}
Let $T$ be a finite simple group of Lie type, let $T \leq G \leq \Aut(T)$ and let $x \in G$. Assume that one of the following holds:
\begin{enumerate}
\item $T \in \{ A_m(p^f) \, \text{($m \geq 2$)}, \ D_m(p^f) \, \text{($m \geq 4$)}, \ E_6(p^f) \}$ and $x \in \Inndiag(T)\g\p^j$ for $f/j$ odd, where $G$ does not contain triality if $T = D_4(p^f)$
\item $T \in \{ {}^2A_m(p^f) \, \text{($m \geq 2$)}, \ {}^2D_m(p^f) \, \text{($m \geq 4$)}, \ {}^2E_6(p^f) \}$ and $x \in \Inndiag(T)\p^j$ for $j \div f$
\item $T = D_4(p^f)$ and $x \in \Inndiag(T)\t\p^j$ with $3 \ndiv \frac{f}{j}$
\item $T = {}^3D_4(p^f)$ and $x \in T\p^j$ with $j \div f$.
\end{enumerate}
Then $x$ is not totally deranged.
\end{proposition}

\begin{proof}
In (i) and (ii), if $T = D_m^\pm(p^f)$, then \cite[Theorem~5.8(iii)]{ref:Harper21} implies that $x$ is contained in $N_G(H)$ where $H$ is the stabiliser in $T$ of a $1$- or $2$-space of the natural module for $T$, so, consulting \cite[Tables~3.5E \&~3.5F]{ref:KleidmanLiebeck}, we deduce that $x$ is contained in a core-free maximal subgroup of $G$. For the rest of the proof we will assume that $T \neq D_m^\pm(p^f)$ in (i) and (ii).

To unify notation, write $\a = \g$ if $T$ is $A_m^\pm(p^f)$ or $E_6^\pm(p^f)$ and write $\a = \t$ otherwise; in all cases write $d=|\alpha|$. Fix a simple algebraic group $X$ and two Steinberg endomorphisms $\s_1$ and $\s_2$ such that $X_{\s_1} = \Inndiag(T)$ and $x \in \Inndiag(T)\ws_2$. The specific choices for $(X,\s_1,\s_2)$ are given in Table~\ref{tab:graph}. Write $e=f/j$ and observe that $\s_1 = \a\s_2^e$, so $\ws_1 = \a$.

We now define $F\: \Inndiag(T)\ws_2 \to X(q_0)\a$. For Row~1, $F$ is the Shintani map of $(X,\s_1,\s_2)$. For Row~2, we use the notation of Remark~\ref{rem:shintani}(iii). Let $E_0\: \Inndiag(T)\ws_2 \to {}^dX(q_0)\a^{-e}$ be the Shintani map of $(X,\p^f,\a\p^j)$ and choose $F_0\:\Inndiag(T)\ws_2 \to {}^dX(q_0)\a$ satisfying $g \mapsto E_0(g)^\pm$, so $F_0 \in \{ E_0, E_0'\}$. For $i > 0$, let $F_i$ be the Shintani map of $(X,\a^i\p^j,\a^{i+1}\p^j)$. Let $F$ be $F_0F_1'$ if $d=2$ and $F_0F_1'F_2'$ if $d=3$. In all cases, $x \in \Inndiag(T)\ws_2$, so define $x_0 = F(x) \in X(q_0)\a$. We will write $\Inndiag(T_0) = X(q_0)$ and $G_0 = \< X(q_0), \a \>$.

\begin{table}[t]
\[
\begin{array}{ccccccc}
\hline 
\text{row} & X                   & \s_1   & \s_2   & X_{\s_1} & X_{\s_2}   & \text{conditions} \\
\hline
1          & A_m, \, D_4, \, E_6 & \a\p^f & \p^j   & {}^dX(q) & X(q_0)     &                   \\
2          & A_m, \, D_4, \, E_6 & \p^f   & \a\p^j & X(q)     & {}^dX(q_0) & d \ndiv e         \\
\hline
\end{array}
\]
\caption{Notation for the proof of Proposition~\ref{prop:graph} (here $q=p^f$, $q_0 = p^j$, $j$ divides $f$, $e=f/j$).}\label{tab:graph}
\end{table}

\emph{\textbf{Case~1.} $x_0^d$ is unipotent.}\nopagebreak

First assume that $p=d$, so $x_0$ is a $p$-element of $G_0$. Let $P_0$ be an $\a$-stable parabolic subgroup of $\Inndiag(T_0)$. Then $N_{G_0}(P_0) = \<P_0,\a\>$ contains a Sylow $p$-subgroup of $G_0$, so $x_0$ is contained in a suitable $G_0$-conjugate of $N_{G_0}(P_0)$. Therefore, Theorem~\ref{thm:shintani_subgroups} implies that $x$ is contained in the normaliser of a maximal parabolic subgroup of $G$.

Now assume that $p \neq d$. In this case, $x_0 = su = us$ where $s$ has order $d$ and $u$ is a unipotent element of $\Inndiag(T_0)$. We will fix a Borel subgroup $B$ of $X$ so that $s$ induces a graph automorphism on the corresponding root system; note that $C_B(s)^\circ$ is a Borel subgroup of $C_X(s)^\circ$. Now, since $u \in C_X(s)^\circ$, we know that there exists $g \in C_G(s)$ such that $u \in (C_B(s)^\circ)^g = C_{B^g}(s)^\circ \leq B^g$ and $(B^g)^s = (B^s)^g = B^g$. Therefore, $u$ is contained in an $s$-stable Borel subgroup of $X$ and hence an $s$-stable Borel subgroup $B_0$ of $\Inndiag(T_0)$. We now deduce that $x_0 = su$ is contained in $N_{G_0}(P_0)$ where $P_0$ is a maximal $s$-stable parabolic subgroup of $\Inndiag(T_0)$ containing $B_0$. As before, Theorem~\ref{thm:shintani_subgroups} implies that $x$ is contained in the normaliser of a maximal parabolic subgroup of $G$.

\emph{\textbf{Case~2.} $x_0^d$ is not unipotent.}\nopagebreak

Here we fix a power $y_0$ of $x_0$ that is a semisimple element of prime order in $\Inndiag(T_0)$. Let $H_0 = C_{G_0}(y_0)$, so $x_0 \in H_0$. By \cite[Theorem~4.2.2(j)]{ref:GorensteinLyonsSolomon98}, $y_0^{T_0} = y_0^{\Inndiag(T_0)}$, which means that $\< T_0, H_0 \> \geq \< T_0, C_{\Inndiag(T_0)}(y) \> = \Inndiag(T_0)$. Additionally, $x_0 \in \Inndiag(T_0)\a$ centralises $y_0$, so, in fact, $\< T_0, H_0 \> = G_0$. Let $M_0$ be a maximal subgroup of $G_0$ that contains $H_0$. Since $\< T_0, M_0 \> \geq \< T_0, H_0 \>$, we know that $\< T_0, M_0 \> = G_0$, so $T_0 \not\leq M_0$. Moreover, the fact that $H_0 \leq M_0$ implies that $M_0 = Y_{\p^j}$ for a closed $\p^j$-stable maximal rank subgroup $Y$ of $X$. Now applying Theorem~\ref{thm:shintani_cosets} gives that $x$ is contained in a core-free maximal subgroup of $G$.
\end{proof}

\begin{remark}\label{rem:graph}
There is an alternative proof of Case~1 of Proposition~\ref{prop:graph} when $T_0 = \PSL_n(q_0)$. Write $V = \F_{q_0}^n$ for the natural module for $T_0$. Since $x_0^2$ is unipotent, there is a $1$-space $U$ of $V$ stabilised by $x_0^2$. Let $H_0$ be the stabiliser in $T_0$ of $U$. Then $x_0$ normalises $H_0 \cap H_0^{x_0}$. Now $H_0^{x_0}$ is the stabiliser in $T$ of an $(n-1)$-space $W$ of $V$. If $U \subseteq W$, then $H_0\cap H_0^{x_0}$ is a subgroup of type $P_{1,n-1}$; otherwise, $V = U \oplus W$, so $H_0 \cap H_0^{x_0}$ has type $\GL_1(q_0) \times \GL_{n-1}(q_0)$. In either case $x_0$ is contained in $M_0 = N_{G_0}(H_0 \cap H_0^{x_0})$ and we proceed by Shintani descent in the usual way.
\end{remark}

\subsection{\boldmath The groups $B_2(q)$, $G_2(q)$ and $F_4(q)$} \label{ss:proof_bfg}

We now turn to the cases not covered by Propositions~\ref{prop:standard} and~\ref{prop:graph}, beginning with the almost simple groups with socle $T \in \{ B_2(2^f), G_2(3^f), F_4(2^f) \}$.

\begin{proposition} \label{prop:bfg}
Let $T \in \{B_2(2^f),\, G_2(3^f),\, F_4(2^f) \}$ and let $G = \< T, \rho^i \>$ for an odd divisor $i$ of $f$. Let $x \in G$. Assume that $(G,x)$ is not in Theorem~\ref{thm:derangement}. Then $x$ is not totally deranged.
\end{proposition}

Before proving Proposition~\ref{prop:bfg}, let us record two preliminary results. We include the second, not just to reduce the length of the proof of Proposition~\ref{prop:bfg}, but also because we will adopt the arguments in the proof a few more times in Section~\ref{ss:proof_ade} and we wish to explain the general strategies in detail once and for all.

\begin{lemma} \label{lem:bfg_b2b2}
Let $T_0 = F_4(q_0)$ with $q_0=2^j$ and let $M_0 \leq T_0$ have type $\Sp_4(q_0)^2$. Then $M_0$ has two subgroups of order $(q_0^2-1)^2$ that are not $T_0$-conjugate.
\end{lemma}

\begin{proof}
Let $K_0 = K_1 \times K_2$ where $K_1,K_2$ are maximal tori of $\Sp_4(q_0)$ of orders $(q_0-1)^2$ and $(q_0+1)^2$, and let $L_0 = L_1 \times L_2$ where $L_1,L_2$ are non-conjugate maximal tori of $\Sp_4(q_0)$ of order $q_0^2-1$. It suffices to show that $K_0$ and $L_0$ are not $T_0$-conjugate. Note that $N_{T_0}(K_0)/K_0 \geq N_{M_0}(K_0)/K_0 = D_8^2$ (where $D_8$ is the dihedral group of order $8$) and $N_{M_0}(L_0)/L_0 = 2^4$. We claim that $N_{M_0}(L_0) = N_{T_0}(L_0)$ (from which it follows that $N_{T_0}(L_0)/L_0 = 2^4$ and $L_0$ is not $T_0$-conjugate to $K_0$). Clearly, $N_{T_0}(L_0)$ is contained in a maximal subgroup $N_0$ of $T_0$ of maximal rank. Inspecting the possibilities for $N_0$ (parabolic subgroups and reductive subgroups of maximal rank in \cite{ref:LiebeckSaxlSeitz92}), we see that the only such subgroup to contain a subgroup of order $(q_0^2-1)^2.2^4$ is $N_0 = \Sp_8(q_0)$, but $N_{M_0}(L_0) = N_{N_0}(L_0)$, so $N_{M_0}(L_0) = N_{T_0}(L_0)$, as claimed. 
\end{proof}

\begin{lemma} \label{lem:bfg_mixed}
Let $T_0 = F_4(q_0)$ with $q_0 = 2^j$ and let $x_0 \in T_0$ be mixed. Then $x_0$ is contained in one of the following types of maximal rank maximal subgroups of $T_0$
\[
P_{2,3}, \quad \SL_3(q_0) \circ \SL_3(q_0), \quad \SU_3(q_0) \circ \SU_3(q_0), \quad \Sp_4(q_0) \wr \Sym_2, \quad \Sp_4(q_0^2).2.
\]
\end{lemma}

\begin{proof}
Let $X$ be the algebraic group $F_4$ and let $\s$ be the Steinberg endomorphism $\p^j$. Let $S$ be a $\s$-stable maximally split maximal torus of $X$. Write $\Phi$ for the root system of $X$ with respect to $S$ and let $\a_1, \a_2, \a_3, \a_4$ be the simple roots of $\Phi$. For $i \in \{1,2,3,4\}$, let
\[
S_i = \{ h \in S \mid \text{$\a_j(h) = 1$ for all $j \in \{ 1,2,3,4 \} \setminus \{ i \}$} \},
\] 
which is a $1$-dimensional subtorus. Note that $S = S_1 \times S_2 \times S_3 \times S_4$ and $S_\s = (S_1)_\s \times (S_2)_\s \times (S_3)_\s \times (S_4)_\s$, with $|(S_i)_\s| = q_0-1$.

Write $x_0 = su = us$ where $s \neq 1$ is semisimple, $u \neq 1$ is unipotent. Note that $C_{T_0}(x_0) \neq T_0$ since $s \neq 1$ and $C_{T_0}(x_0)$ is not a torus since $u \neq 1$. The conjugacy classes of semisimple elements of $T_0$, together with their centralisers, are given in \cite[Tables~II \&~III]{ref:Shinoda74}, but we prefer to use the convenient presentation of the conjugacy classes of centralisers of semisimple elements of $T_0$ given in \cite{ref:Luebeck}. Inspecting the possibilities for $C_{T_0}(x_0)$, we see that $C_{T_0}(x_0) = Y_\s$ for a maximal rank closed subgroup $Y < X$. Moreover, by \cite[Proposition~14.7]{ref:MalleTesterman11} for example, $x_0 = su \in Y^\circ_\s$. Write $H_0 = Y^\circ_\s$ and write $\mathcal{M}_0$ for the set of maximal subgroups of $T_0$ of one of the five types given in the statement. Except in one specific case (where we proceed differently), we will identify $M_0 \in \mathcal{M}_0$ containing $H_0$ and thus also containing $x_0$.

\emph{\textbf{Case~1.} $Y^\circ$ has type $A_2 A_2$.}\nopagebreak

Here $H_0$ has type $\SL_3(q_0) \circ \SL_3(q_0)$ or $\SU_3(q_0) \circ \SU_3(q_0)$, so is certainly contained in a subgroup in $\mathcal{M}_0$.

\emph{\textbf{Case~2.} $Y^\circ$ has type $S_1 B_2 S_4$.}\nopagebreak

Here $H_0$ is the centraliser of a subgroup $Z_0$ of order $(q_0-1)^2$, $(q_0+1)^2$, $q_0^2-1$ or $q_0^2+1$ (for each order $H_0$ is determined uniquely up to $\<T_0,\rho\>$-conjugacy, even though there are two possibilities up to $T_0$-conjugacy when $|Z_0|=q_0^2-1$), and $H_0 = \Sp_4(q_0) \times Z_0$ in each case. Let $M_0$ have type $\Sp_4(q_0)^2$. Then $M_0$ has a subgroup $Z_0'$ of order $|Z_0|$ with centraliser of order $|H_0|$, and since no subgroups of $T_0$ of order $|Z_0|$ have centraliser order properly divisible by $|H_0|$, we deduce that $|C_{T_0}(Z_0')| = |C_{M_0}(Z_0')| = |H_0|$. Therefore, $H_0 = C_{T_0}(Z_0)$ and $C_{T_0}(Z_0')$ are $\<T_0,\rho\>$-conjugate, so $H_0$ is contained in a subgroup $M_0 \in \mathcal{M}_0$ of type $\Sp_4(q_0)^2$.

\emph{\textbf{Case~3.} $Y^\circ$ has type $A_2 S_3 A_1$, $A_2 S_3 S_4$, $A_1 S_2 S_3 A_1$ or $A_1 S_2 S_3 S_4$.}\nopagebreak

In \cite{ref:Luebeck}, under \texttt{elements of other class types in center}, for each semisimple centraliser $C_{T_0}(y)$, we have a list of elements $z_1, \dots, z_k$ such that $z_i \in Z(C_{T_0}(y))$ and, conversely, if $z \in Z(C_{T_0}(y))$, then there exists some $1 \leq i \leq k$ such that $C_{T_0}(z)$ is $T_0$-conjugate to $C_{T_0}(z_i)$. As $z \in Z(C_{T_0}(y))$ implies that $C_{T_0}(y) \leq C_{T_0}(z)$, it is easy to use this information to deduce that one semisimple centraliser is contained in another. By considering each possibility for $C_{T_0}(s)$ with these choices of $Y$ (these are, without loss of generality, \texttt{i = 7, 13, 14, 18} in \cite{ref:Luebeck}), we deduce that $C_{T_0}(s) \leq Z_{\s}$ where $Z$ has type $A_2 A_2$ or $S_1 B_2 S_4$ (\texttt{i = 4, 15}). Applying the conclusions from the previous two cases, we deduce that $H_0 \leq M_0$ for $M_0 \in \mathcal{M}_0$.

\emph{\textbf{Case~4.} $Y^\circ$ has type $B_3 S_4$.}\nopagebreak 

We handle this case differently. Let $B$ be the Borel subgroup of $X$ containing $S$ defined as $B = \< S, U_\alpha \mid \alpha \in \Phi^+ \>$, and let $U$ be the unipotent radical of $B$. Let $P$ be the $P_{2,3}$ parabolic subgroup $P = \< S, U_\alpha \mid \alpha \in \Phi^+ \cup \Phi_{1,4} \>$, and let $L$ be the Levi complement of $P$, so 
\[
L = \< S, U_\alpha \mid \alpha \in \Phi_{1,4}\> = A_1 \times S_2 \times S_3 \times A_1.
\] 
Observe that $U \leq B \leq P$. Noting that $P$ is $\rho$-stable, write $M_0 = P_\s$. Observe that 
\[
H_0 = C_{T_0}(s) = B_3(q_0) \times \< s \> = \< S_1, S_2, S_3, U_\alpha \mid \alpha \in \Phi_{1,2,3} \>_\s \times \< s \>.
\] 
Note that 
\[
s \in \< S, U_\alpha \mid \alpha \in \Phi_4 \>_\s = (S_1 \times S_2 \times S_3)_\s \times A_1(q_0) \leq L_\s \leq P_\s.
\] 
Moreover, by conjugating in $H_0$ if necessary, we may assume that 
\[
u \in \< U_\alpha \mid \alpha \in \Phi^+_{1,2,3}\>_\s \leq U_\s \leq P_\s.
\] 
Therefore, $x_0 = su \in P_\s = M_0$. 

In all cases, $x_0 \in M_0$ for some $M_0 \in \mathcal{M}_0$, as claimed.
\end{proof}

With these results in place, we can now continue with the main proof.

\begin{proof}[Proof of Proposition~\ref{prop:bfg}]
By Proposition~\ref{prop:computational}, assume that $T \neq G_2(3)$. By replacing $x$ by another generator of $\<x\>$ if necessary, assume that $x \in T\wrho^l$ where $i$ divides $l$ and $l$ divides $2f$. Fix $X \in \{ B_2,\, G_2,\, F_4 \}$ and $p \in \{2,3\}$ such that, writing $q=p^f$, we have $T = X(q)$. Fix $\s_1 = \p^f$ and $\s_2 = \rho^l$, so $X_{\s_1} = T$ and $x \in T\ws_2$. Let $F\:X_{\s_1}\ws_2 \to X_{\s_2}\ws_1$ be the Shintani map of $(X,\s_1,\s_2)$. Write $T_0 = X_{\s_2}$ and let $x_0 = F(x) \in T_0$. Since $\s_1 = \s_2^{2f/l}$, Remark~\ref{rem:shintani}(ii) gives us $|x| = e|x_0|$ where $e=2f/l$. 

Let $S$ be a $\rho$-stable maximally split maximal torus of $X$. Let $W_X$ be the Weyl group $N_X(S)/S$, and for $g \in N_X(S)$ write $\bar{g} = Sg \in W_X$. 

\emph{\textbf{Case~1.} $l$ is odd.}\nopagebreak

Here $T_0 = {}^2X(q_0)$ for $q_0 = p^l$ and $F\:T\wrho^l \to T_0$. 

\emph{\textbf{Case~1a.} $x_0$ is semisimple.}\nopagebreak

In this case, $x_0$ is contained in a maximal torus $H_0$ of $T_0$. Let $t \in X$ such that $S^t$ is $\s_2$-stable and $H_0 = (S^t)_{\s_2}$. Theorem~\ref{thm:shintani_subgroups} implies that $x$ is contained in a $G$-conjugate of $\< H, \ws_2 \> \leq N_G(H)$ where $H = (S^t)_{\s_1}$. We will show that $N_G(H)$ is a maximal subgroup of $G$, or, in some cases with $X = F_4$, find an another core-free maximal overgroup of $x$. Write $s = t_{\s_2}$, so that $H_0 = (S^t)_{\s_2} = S_{s\s_2}^t$ and $H = (S^t)_{\s_1} = S_{(s\s_2)^{2f/l}}$. Recall $\s_2 = \rho^l$ induces an involution $\alpha$ on $W_X$. 

First assume that $X \in \{B_2, G_2\}$. Here, $W_X = \< a, b \mid a^{2p}, b^2, a^b = a^{-1} \> \cong D_{4p}$ (recall that $p=2$ if $X=B_2$ and $p=3$ if $X=G_2$). Since $\<a\>^\alpha = \<a\>$, we see that 
\[
(\bar{s}\s_2)^{2f/l} = (\bar{s}\rho^l)^{2f/l} = (\bar{s}\bar{s}^\alpha)^{f/l} \p^f \in \<a\>\s_1.
\] 
Therefore, $|H| \neq q^2-1$ (see \cite[(4.4) \& (5.2)]{ref:Gager73} for example), so, consulting \cite[Tables~8.14 \& 8.42]{ref:BrayHoltRoneyDougal}, we see that $N_G(H)$ is maximal unless $T = \Sp_4(4)$ and $|H| = (4-1)^2$, in which case $N_G(H) \leq \widetilde{H} = \< \Sp_4(2), \wrho \>$, which is a core-free maximal subgroup of $G$ containing $x$ (this is easy to verify directly, but the details are also given in the proof of \cite[Proposition~7.2.7]{ref:BrayHoltRoneyDougal}).

Now assume that $X = F_4$.We consider each class of maximal torus in ${}^2F_4(q_0)$ in turn (for the list of these 11 classes, see \cite{ref:Shinoda75}). First consider $|H_0| = q_0^2 \pm \sqrt{2}q_0^{3/2} + q_0 \pm \sqrt{2q_0} + 1$. By \cite[Table~7.3]{ref:Gager73}, $\bar{s}\bar{s}^\alpha$ is a Weyl group element $w \in W_{F_4}$ of order $12$. Therefore, since 
\[
(\bar{s}\s_2)^{2f/l} = w^{f/l}\p^f \in \<w\>\s_1,
\] 
we deduce that $|H|$ is 
\[
|S_{(s\s_2)^{2f/l}}| \in \{ (q \pm 1)^2,\, (q^2+1)^2, \, (q^2 \pm q +1)^2,\, q^4-q^2+1 \}.
\] 
If $q > 2$, then it follows from \cite[Table~5.2]{ref:LiebeckSaxlSeitz92} that $N_G(H)$ is maximal (noting that if $q=4$, then $f=2$ and $l=1$, which means that $|w^{f/l}|=6$, so, in particular, $|H| \neq (4-1)^2$). Now assume that $T = F_4(2)$, so $f=l=1$ and $|w^{f/l}|=12$. In this case, $|H| = 2^4-2^2+1 = 13$ and $N_G(H) \leq \widetilde{H} = \< {}^2F_4(2), \wrho \>$, which is a core-free maximal subgroup of $G$ containing $x$ (this information is provided in \cite[Tables~1(b) \&~2]{ref:NortonWilson89}).

For the remaining possibilities for $H_0$, we will identify an alternative core-free maximal overgroup of $x$. If $|H_0| \in \{ (q_0 \pm 1)^2,\, q_0^2 \pm 1 \}$, then $x_0$ is contained in a subgroup $M_0$ of $T_0$ of type $\Sp_4(q_0)$ and if $|H_0| \in \{ (q_0-1)(q_0 \pm \sqrt{2q_0} +1),\, (q_0 \pm \sqrt{2q_0} +1)^2 \}$ then $x_0$ is contained in a subgroup $M_0$ of type ${}^2B_2(q_0)^2$; in both cases, Theorem~\ref{thm:shintani_subgroups} implies that $x$ is contained in $M = N_G(K)$ where $K$ has type $\Sp_4(q)^2$ or $\Sp_4(q^2)$, and in either case, $M$ is maximal in $G$ (see \cite[Table~5.1]{ref:LiebeckSaxlSeitz92}). The final possibility is that $|H_0| = q_0^2-q_0+1$, and here $x_0$ is contained in a subgroup $M_0$ of type $\SU_3(q_0)$, so Theorem~\ref{thm:shintani_subgroups} implies that $x$ is contained in $M = N_G(K)$ where $K$ has type $\SL^\pm_3(q) \circ \SL^\pm_3(q)$, which is maximal in $G$ (see \cite[Table~5.1]{ref:LiebeckSaxlSeitz92}).

\emph{\textbf{Case~1b.} $x_0$ is not semisimple.}\nopagebreak

By Lemma~\ref{lem:parabolic}, $x_0$ is contained in a maximal parabolic subgroup $H_0$ of $T_0$. Now $H_0 = Y_{\s_2}$ for a maximal $\rho$-stable parabolic subgroup $Y \leq X$, so Theorem~\ref{thm:shintani_subgroups} implies that $x$ is contained in an $\< X_{\s_1}, \ws_2 \>$-conjugate of $\< Y_{\s_1}, \ws_2 \>$. Let $M = N_G(Y_{\s_1})$, so $x \in M$. Since $Y$ is $\rho$-stable, $\wrho^i \in M$, so $|N_G(Y_{\s_1}):Y_{\s_1}| = 2f/i = |G:T|$ and we deduce that $M = N_G(Y_{\s_1})$ is maximal in $G$.

\vspace{0.3\baselineskip}

\emph{\textbf{Case~2.} $l$ is even.}\nopagebreak

Here $\rho^l = \p^j$, where $j=l/2$, so $T_0 = X(q_0)$ for $q_0 = p^j$ and $F\:T\p^j \to T_0$. 

\emph{\textbf{Case~2a.} $x_0$ is semisimple.}\nopagebreak

First assume that $X \in \{B_2, G_2\}$ and that $x_0$ is contained in a maximal torus $H_0$ of $T_0$ of order different from $q_0^2-1$. Fix $t \in X$ such that $S^t$ is $\s_2$-stable and $H_0 = (S^t)_{\s_2}$ and write $s = t_{\s_2}$, so $H_0= S^t_{s\s_2}$. Theorem~\ref{thm:shintani_subgroups} implies that $x$ is contained in a $G$-conjugate of $N_G(H)$ where $H = (S^t)_{\s_1} = S^t_{(s\s_2)^{f/j}}$. Since $\s_2 = \p^j$ acts trivially on $W_X$, $H = S^t_{s^{f/j}\s_1}$. Writing 
\[
W_X = \< a, b \mid a^{2p}, b^2, a^b = a^{-1} \> \cong D_{4p},
\]
as $|H_0| \neq q_0^2-1$, $\bar{s} \in \< a \>$ (see \cite[(4.4) \& (5.2)]{ref:Gager73}), so $\bar{s}^{f/j} \in \< a \>$ and $|H| \neq q^2-1$. As in Case~1a, $N_G(H)$ is maximal or $T = \Sp_4(4)$, $|H| = 9$ and $N_G(H) \leq \widetilde{H} = \< \Sp_4(2), \wrho \>$, which is maximal.

Next assume that $X \in \{ B_2, G_2 \}$ and that every maximal torus of $T_0$ that contains $x_0$ has order $q_0^2-1$. This implies that the order $r$ of $x_0$ satisfies $r \div (q_0^2-1)$ and $r \ndiv (q_0 \pm 1)$. If $X = G_2$, then $x_0$ is contained in a maximal subgroup of $T_0$ of type $\SL_2(q_0) \circ \SL_2(q_0)$, so by Theorem~\ref{thm:shintani_subgroups}, $x$ is contained in a maximal subgroup of $G$ of type $\SL_2(q) \circ \SL_2(q)$ (see \cite[Table~8.42]{ref:BrayHoltRoneyDougal}). Now assume that $X = B_2$. If $f/i$ is even, then arguing as above, $\bar{s} = 1$, so $|H| = (q-1)^2$ and $N_G(H)$ is maximal or $T = \Sp_4(4)$, $|H| = (4-1)^2$ and $N_G(H) \leq \widetilde{H} = \< \Sp_4(2), \wrho \>$ is maximal. Therefore, we will assume that $f/i$ is odd. Since $x$ is not in Theorem~\ref{thm:derangement}, either there exists a prime divisor $k$ of $j$ not dividing $f/j$ such that $r \div (q_0^{2/k}-1)$, or there there exists a prime divisor $k$ of $f/j$ such that $k \ndiv r$ or $k \div (q_0^2-1)/r$. In the former case, $x_0$ is contained in a $T_0$-conjugate of $B_2(q_0^{1/k})$, which by Corollary~\ref{cor:shintani_subfields}, means that $x$ is contained in a $T$-conjugate of $N_G(B_2(q^{1/k}))$. In the latter case there exists $z \in T_0$ such that $z^r = x_0$, which by Corollary~\ref{cor:shintani_subfields}, means that $x$ is contained in a $T$-conjugate of $N_G(B_2(q^{1/k}))$. In both cases, $x$ is contained in a core-free maximal subgroup of $G$ (see \cite[Table~8.14]{ref:BrayHoltRoneyDougal}).

Now assume that $X = F_4$ and fix a maximal torus $H_0$ of $T_0$ containing $x_0$. The $T_0$-classes of maximal tori of $T_0$ given in \cite[p.145]{ref:Shinoda74}. We will consider each possibility for $H_0$ in turn and identify a subgroup $M_0$ of $T_0$ that contains $H_0$. It will be useful to note that there are two $\<T_0,\rho\>$-classes of maximal tori of order $(q_0^2-1)^2$, but otherwise $|H_0|$ determines $H_0$ up to $\<T_0,\rho\>$-conjugacy. 

First assume that $|H_0| = (q_0^2-1)^2$. Here there are two possibilities for $H_0$ up to $T_0$-conjugacy (and $\<T,\rho\>$-conjugacy), but by Lemma~\ref{lem:bfg_b2b2}, both types are contained in a subgroup $M_0$ of type $\Sp_4(q_0)^2$. Now assume that $|H_0| \neq (q_0^2-1)^2$, so $H_0 \leq M_0$ provided that $M_0$ has a maximal torus of order $|H_0|$. If 
\[
|H_0| \in \{ (q_0^2 + \e q_0+1)^2,\, (q_0^3-\e)(q_0 \pm 1) \},
\] 
then let $M_0$ have type $\SL^\e_3(q_0) \circ \SL^\e_3(q_0)$, if 
\[
|H_0| \in \{ (q_0 \pm 1)^4,\, (q_0 \pm 1)^3(q_0 \mp 1),\, (q_0 \pm 1)^2(q_0^2+1),\, (q_0^2+1)^2,\, q_0^4-1 \},
\] 
then let $M_0$ have type $\Sp_4(q_0)^2$, if $|H_0| = q_0^4+1$, then let $M_0$ have type $\Sp_4(q_0^2)$ and if $|H_0| = q_0^4-q_0^2+1$, then let $M_0 = H_0$.

Therefore, in all cases, $M_0$ is $T_0$-conjugate to $Y_{\s_2}$ where $Y$ is a a closed connected $\rho$-stable subgroup $Y$ of $X$ of type $A_2\tilde{A}_2$ or $B_2^2$ or a maximal torus. Therefore, by Theorem~\ref{thm:shintani_subgroups}, $x$ is contained in a $G$-conjugate of $M = N_G(Y_{\s_1})$, which we claim is a core-free maximal subgroup of $G$ (except for some very small values of $q$). If $|H_0| \neq q_0^4-q_0^2+1$, then $Y$ has type $A_2\tilde{A}_2$ or $B_2^2$, so $Y_{\s_1}$ has one of the following types $\SL_3(q) \circ \SL_3(q)$, $\SU_3(q) \circ \SU_3(q)$, $\Sp_2(q)^2$ or $\Sp_2(q^2)$, all of which yield core-free maximal subgroups $M = N_G(Y_{\s_1})$ (see \cite[Table~1]{ref:LiebeckSaxlSeitz92}). Now assume that $|H_0| = q_0^4-q_0^2+1$. Here, we fix $t \in X$ such that $Y = S^t$, so $H_0 = (S^t)_{\s_2} = S_{s\s_2}^t$, where we write $s = t_{\s_2}$. By \cite[Table~5.2]{ref:Gager73}, we may assume that $\bar{s} \in W_{F_4}$ has order 12. Now $Y_{\s_1} = (S^t)_{\s_2} = S^t_{s^{f/j}\s_2}$, which means that 
\[
|Y_{\s_1}| = |S_{s\s_1}| \in \{ (q \pm 1)^2,\, (q^2+1)^2, \, (q^2 \pm q +1)^2,\, q^4-q^2+1 \},
\] 
so, as in Case~1(a), either $M = N_G(Y_{\s_1})$ is maximal, or $T = F_4(2)$, $|H| = 2^4-2^2+1 = 13$ and $N_G(H) \leq \widetilde{M} = \< {}^2F_4(2), \wrho \>$ is maximal.

\emph{\textbf{Case~2b.} $x_0$ is unipotent.}\nopagebreak

In this case, $x_0$ is contained in a Borel subgroup $H_0$ of $T_0$. If $X \in \{B_2,G_2\}$, then let $M_0 = H_0$ and if $X = F_4$, then let $M_0$ be a parabolic subgroup of $T_0$ of type $P_{1,4}$ or $P_{2,3}$ that contains $H_0$. In either case, $x_0 \in M_0$ and $M_0$ is $T_0$-conjugate to $Y_{\s_2}$ for a maximal $\rho$-stable parabolic subgroup of $Y \leq X$. Theorem~\ref{thm:shintani_subgroups} implies that $x$ is contained in an $\< X_{\s_1}, \ws_2 \>$-conjugate of $\< Y_{\s_1}, \ws_2 \>$. Let $M = N_G(Y_{\s_1})$, so $x \in M$. Since $Y$ is $\rho$-stable, $\wrho^i \in M$, so $|N_G(Y_{\s_1}):Y_{\s_1}| = 2f/i = |G:T|$ and we deduce that $M = N_G(Y_{\s_1})$ is maximal in $G$.

\emph{\textbf{Case~2c.} $x_0$ is mixed.}\nopagebreak

Let us make some general comments about the case $X \in \{B_2,G_2\}$. The maximal subgroups of $T_0$ are given in \cite[Table~8.14]{ref:BrayHoltRoneyDougal} if $T_0 = B_2(q_0)$ and \cite[Table~8.42]{ref:BrayHoltRoneyDougal} if $T_0 = G_2(q_0)$. The conjugacy classes of $T_0$ are given in \cite[Table~IV-1]{ref:Enomoto72} if $T_0 = B_2(q_0)$ and \cite[Table~2]{ref:Enomoto70} if $T_0 = G_2(q_0)$. Consulting these lists, we see that if $u \in T_0$ is any nontrivial unipotent element that commutes with a nontrivial semisimple element, then $|C_{T_0}(u)|_{p'}$ divides $q_0^2-1$ and for any two tori $S^+$ and $S^-$ of orders $q_0-1$ and $q_0+1$ contained in $C_{T_0}(u)$, any mixed element of $T_0$ is $\<T_0,\rho\>$-conjugate to an element $us$ where $s \in S^+ \cup S^-$.

First assume that $X = G_2$. Let $H_0$ be a $\rho$-stable subgroup of $T_0$ of type $\SL_2(q_0) \circ \SL_2(q_0)$. Then $H_0$ contains a nontrivial unipotent element that commutes with subgroups of order $q_0-1$ and $q_0+1$, so $x_0$ is contained in a $T_0$-conjugate of $H_0$. Writing $H_0 = Y_{\s_2}$ for a $\rho$-stable subgroup of $Y$ of type $A_1\tilde{A}_1$, by Theorem~\ref{thm:shintani_subgroups}, $x$ is contained in a $G$-conjugate of $N_G(Y_{\s_1})$, which is a maximal core-free subgroup of $G$ (of type $\SL_2(q) \circ \SL_2(q)$).

Next assume that $X = B_2$. Let $H_0^\e$ be a $\rho$-stable subgroup of $T_0$ of type $\O^\e_2(q_0)^2$. Then $H_0^\e$ contains a nontrivial unipotent element that commutes with a subgroup of order $q_0-\e$, so $x_0$ is contained in a $T_0$-conjugate of $H_0^\e$ for some $\e \in \{+,-\}$. Writing $H_0 = Y_{\s_2}$ for a $\rho$-stable subgroup of $Y$ of type $\Omega_2^2$, by Theorem~\ref{thm:shintani_subgroups}, $x$ is contained in a $G$-conjugate of $N_G(Y_{\s_1})$, which is a maximal core-free subgroup of $G$ (of type $\O^\pm_2(q) \wr \Sym_2$ or $\O^-_2(q^2).2$).

Finally assume that $X = F_4$. Here we make use of Lemma~\ref{lem:bfg_mixed}. In particular, $x_0$ is contained in a $T_0$-conjugate of $Y_{\s_2}$ where $Y$ is a a closed connected $\rho$-stable subgroup $Y$ of $X$ (either $P_{1,4}$ or $P_{2,3}$ parabolic or reductive maximal rank of type $A_2\tilde{A}_2$ or $B_2^2$). Therefore, by Theorem~\ref{thm:shintani_subgroups}, $x$ is contained in a $G$-conjugate of $M = N_G(Y_{\s_1})$, which is a core-free maximal subgroup of $G$ (either $P_{1,4}$ or $P_{2,3}$ parabolic or maximal rank of type  $\SL_3(q) \circ \SL_3(q)$, $\SU_3(q) \circ \SU_3(q)$, $\Sp_4(q) \wr \Sym_2$ or $\Sp_4(q^2).2$), see \cite{ref:LiebeckSaxlSeitz92} for maximality. This completes the proof.
\end{proof}

\subsection{\boldmath The groups $A_m(q)$, $D_m(q)$ and $E_6(q)$} \label{ss:proof_ade}

We now turn to the remaining cases not covered by Propositions~\ref{prop:standard} and~\ref{prop:graph}, all of which have socle $T \in \{ \PSL_n(q) \ \text{($n \geq 2$)}, \ \POm^+_{2m}(q) \ \text{($m \geq 4$)}, \ E_6(q) \}$.

\begin{proposition} \label{prop:a}
Let $T = \PSL_n(q)$ and $T \leq G \leq \Aut(T)$ with $G \not\leq \< \Inndiag(T), \p \>$. Let $x \in G$. Assume that $x \in \< \Inndiag(T), \p \>$. Then $x$ is not totally deranged.
\end{proposition}

\begin{proof}
Write $q=p^f$. By replacing $x$ by another generator of $\<x\>$ if necessary, assume that $x \in \PGL_n(q)\wp^j$ for some divisor $j$ of $f$. Fix the simple algebraic group $X = A_{n-1}$ and the Steinberg endomorphisms $\s_1 = \p^f$ and $\s_2 = \p^j$. Let $F\: \Inndiag(T)\wp^j \to \Inndiag(T_0)$ be the Shintani map of $(X,\s_1,\s_2)$ where $T_0 = \PSL_n(q_0)$ for $q_0 = p^j$, and let $x_0 = F(x)$.

If $x_0$ is contained in a subgroup of $\Inndiag(T_0)$ of type $\GL_k(q_0) \times \GL_{n-k}(q_0)$, for some $0 < k < n$, then let $M_0$ be such a subgroup. Now assume that $x_0$ is not contained in such a subgroup. Then, by considering the rational canonical form of $x_0$, we deduce that the minimal polynomial and characteristic polynomial are both equal to $\psi^{n/\deg(\psi)}$ for an irreducible polynomial $\psi \in \F_{q_0}[X]$. If $\deg(\psi) = 1$, then $x_0$ is conjugate to an element with a single Jordan block (of size $n$) and is thus contained in a Borel subgroup of $\Inndiag(T_0)$ and consequently a subgroup $M_0$ of type $P_{1,n-1}$. If $\deg(\psi) = d > 1$, then $x_0$ is conjugate to an element arising from an element with a single Jordan block of size $n/d$ embedded via a degree $d$ field extension of $\F_{q_0}$ and is thus contained in a subgroup of type $\GL_{n/d}(q_0^d)$ and consequently a subgroup $M_0$ of type $\GL_{n/k}(q_0^k)$ where $k$ is the least prime divisor of $d$. 

In all cases, $x_0 \in M_0$ and $M_0$ is $\Inndiag(T_0)$-conjugate to $Y_{\s_2}$ for a closed connected subgroup $Y$ of $X$ (of type $\GL_k \times \GL_{n-k}$, $P_{1,n-1}$ or $\GL_{n/k}^k$). By Theorem~\ref{thm:shintani_subgroups}, $x$ is contained in a $G$-conjugate of $M = N_G(Y_{\s_1})$, which is a core-free maximal subgroup of $G$ (of type $\GL_k(q) \times \GL_{n-k}(q)$, $P_{1,n-1}$, $\GL_{n/k}(q) \wr \Sym_k$ or $\GL_{n/k}(q^k).k$), see \cite[Table~3.5A]{ref:KleidmanLiebeck} using the fact that $G \not\leq \< \Inndiag(T), \wp \>$.
\end{proof}

\begin{proposition} \label{prop:e}
Let $T = E_6(q)$ and $T \leq G \leq \Aut(T)$ with $G \not\leq \< \Inndiag(T), \p \>$. Let $x \in G$. Assume that $x \in \< \Inndiag(T), \p \>$. Then $x$ is not totally deranged.
\end{proposition}

\begin{proof}
As usual, write $q=p^f$, assume that $x \in \Inndiag(T)\wp^j$ for a divisor $j$ of $f$ and let $F\: \Inndiag(T)\wp^j \to \Inndiag(T_0)$ be the Shintani map of $(X,\s_1,\s_2) = (E_6, \p^f, \p^j)$ where $T_0 = E_6(q_0)$ and $q_0 = p^j$. Write $x_0 = F(x)$. 

If $x_0$ is unipotent, then $x_0$ contained in a Borel subgroup and therefore a $P_4$ parabolic subgroup of $T_0$. 

From now we will assume that $x_0 \in C_{\Inndiag(T_0)}(s)$ for a nontrivial semisimple element $s \in G_0$. We proceed as we did in the proof of Lemma~\ref{lem:bfg_mixed}. Consulting \cite{ref:Luebeck}, note that $C_{\Inndiag(T_0)}(s) = Y_{\s_2}$ for a closed $\s_2$-stable subgroup $Y \leq X$, and recall that $x_0 \in M_0 = Y^\circ_{\s_2}$. For now assume that $M_0$ does not have order $q_0^6+q_0^3+1$ ( \texttt{[i,j,k] = [21,1,14]} in \cite{ref:Luebeck}). Then arguing as in Case~3 of Lemma~\ref{lem:bfg_mixed}, using the centraliser inclusions given in \cite{ref:Luebeck}, we deduce that $M_0 \leq Z_{\s_2}$ where $Z$ has one of the following types 
\[
A_5A_1, \ \  A_2^3, \ \ D_5S, \ \ A_4A_1S, \ \ A_2^2A_1S, \ \ A_5S, \ \ A_2^2S^2, \ \ D_4S^2
\]
(corresponding to \texttt{i = 2, 3, 4, 5, 6, 8, 12, 14}). Arguing as in Case~2 of Lemma~\ref{lem:bfg_mixed}, if $Y$ has type $A_4A_1S$, $A_2^2A_1S$, $A_5S$, $A_2^2S^2$ or $M_0$ has order $q_0^6+q_0^3+1$ (\texttt{i =  5, 6, 8, 12} or \texttt{[i,j,k] = [21,1,14]}), then $M_0 \leq Z_{\s_2}$ where $Z$ has type $A_5A_1$, $A_2^3$, $D_5S$ or $D_4S^2$. (If $p \geq 5$, then we can carry out these two steps in one, via centraliser inclusions, but we must be more careful in general since $Z_{\s_2}$ is not a semisimple centraliser if $(Y,p)$ is $(A_5A_1,2)$ or $(A_2^3,3)$.) 

Therefore, in all cases, $x_0$ is contained in a $T_0$-conjugate of $Y_{\s_2}$ where $Y$ is a closed connected $\s_2$-stable subgroup $Y$ of $X$ of type $P_4$, $A_5A_1$, $A_2^3$, $D_5S$ or $D_4S^2$. Therefore, by Theorem~\ref{thm:shintani_subgroups}, $x$ is contained in a $G$-conjugate of $Y_{\s_1}$. By \cite{ref:LiebeckSaxlSeitz92}, either $M = N_G(Y_{\s_1})$ is maximal  or $Y_{\s_1}$ has type $\POm^-_8(q) \times (q^2-1)$ and is contained in a maximal subgroup of type $\POm^+_{10} \times (q-1)$. Either way, $x$ is contained in a core-free maximal subgroup of $G$, so $x$ is not totally deranged.
\end{proof}

We remind the reader that we defined the expression \emph{contains triality} in Remark~\ref{rem:groups_automorphism}(ii).

\begin{proposition} \label{prop:d}
Let $T = \POm^+_{2m}(q)$ for an integer $m \geq 4$ and let $T \leq G \leq \Aut(T)$. Assume that $G \not\leq \<\Inndiag(T), \p \>$ and that $G$ does not contain triality. Let $x$ be an element of $G$. Assume that $x \in \< \Inndiag(T), \p \>$ and that $x$ does not feature in Theorem~\ref{thm:derangement}. Then $x$ is not totally deranged.
\end{proposition}

\begin{proof}
If $m \geq 5$, then $G \leq \Aut(T) = \< \PGO^+_{2m}(q), \wp\>$. Now consider the case $m=4$. Since $G$ does not contain triality, we know that $G \leq \< \Inndiag(T), \wp, \g\t^i \>$ for an integer $i$. By replacing $G$ by $G^{\t^i}$, we obtain $G \leq \< \Inndiag(T), \wp, \g \> = \< \PGO^+_8(q), \wp \>$ while maintaining $x \in \< \Inndiag(T), \wp \>$. Therefore, in all cases, for the remainder of the proof, we may assume that $G \leq \< \PGO^+_{2m}(q), \wp \>$.

As usual, write $q=p^f$, assume that $x \in \Inndiag(T)\wp^j$ for a divisor $j$ of $f$ and let $F\: \Inndiag(T)\wp^j \to \Inndiag(T_0)$ be the Shintani map of $(X,\s_1,\s_2) = (D_m, \p^f, \p^j)$ where $T_0 = \POm^+_{2m}(q_0)$ and $q_0 = p^j$. Write $x_0 = F(x)$. Write $G_0 = \Inndiag(T_0)$ and let $V = \F_{q_0}^{2m}$ be the natural module for $G_0$. Note that $|x| = f/j \cdot |x_0|$.

Write $x_0 = su = us$ for elements $s,u \in G_0$ such that $s$ is semisimple and $u$ is unipotent. Let $U = \{ v \in V \mid vu = v \}$, noting that $U$ is nonzero since $u$ is unipotent. Since $Us = U$ and $s$ is semisimple, we may write $U = \oplus_{i=1}^{d} W_i$ where $s$ stabilises, and acts irreducibly on, each $W_i$ and $0 < \dim{W_1} \leq \dots \leq \dim{W_d} < n$, noting that $W_d$ is proper since $G_0$ contains no elements that act irreducibly on $V$. Let $W = W_1$. Since $s$ acts irreducibly on $W$, we deduce that $W$ is either totally singular or nondegenerate. Now $x_0 = su$ stabilises $W$, and let $H_0$ be the stabiliser of $W$ in $G_0$, so $H_0$ is parabolic if $W$ is totally singular and has type $\O^-_k(q_0) \times \O^-_{2m-k}(q_0)$ if $W$ is nondegenerate. In the latter case, as $u$ acts trivially on $W$ and $s$ acts irreducibly on $W$, $x_0 = su$ is contained in a subgroup of type $\SO^-_k(q_0) \times \SO^-_{2m-k}(q_0)$.

First assume that $W$ is not a totally singular $m$-space. In this case, $H_0 = Y_{\s_2}$ for a closed connected $\s_2$-stable subgroup $Y$ of $X$, where $Y$ is the stabiliser of a totally singular $k$-space with $k < m$ or a nondegenerate $k$-space with $k$ even. By Theorem~\ref{thm:shintani_subgroups}, $x$ is contained in the stabiliser $M_0$ in $G$ of a totally singular $k$-space of $\F_q^{2m}$ with $k < m$ or a nondegenerate $k$-space of $\F_q^{2m}$ with $k$ even. Either way, by \cite[Table~3.5E]{ref:KleidmanLiebeck}, $M = N_G(M_0)$ is a core-free maximal subgroup of $G$, noting that if $q \in \{2,3\}$, then $M_0$ does not have type $\SO^+_2(q) \times \SO^+_{2m-2}(q)$.

Next assume that $W$ is a totally singular $m$-space. Therefore, $\dim{U} \in \{m,2m\}$, so either $U = W$ (with $u \neq 1$) or $U = W \oplus W^* = V$ (with $u = 1$). For now assume that $U = V$. Here $x_0 = s$ is semisimple, and since $x_0$ acts irreducibly on $W$, we deduce that $x_0 = g \oplus g^{-\tr}$ with respect to the decomposition $V = W \oplus W^*$, with $g$ acting irreducibly on $W$. In particular, $x_0$ is semisimple of order $r$ that satisfies $r \div (q_0^m-1)$ and $r \ndiv (q_0^{m/k}-1)$ for all prime divisors $k$ of $m$. If $m$ is even and $r \div 2(q_0^{m/2}+1)$, then $x_0$ also stabilises a nondegenerate minus-type $m$-space (and its orthogonal complement), so we return to the case in the previous paragraph. Therefore, when $U=V$, we can assume that $r \ndiv 2(q_0^{m/2}+1)$ if $m$ is even. 

Now assume that $U = W$. Since $u \in \Inndiag(T_0) \cap \PO^+_{2m}(q_0)$, by \cite[Theorem~11.43]{ref:Taylor92}, $u$ has an even-dimensional $1$-eigenspace, so $m = \dim{W} = \dim{U}$ is even. We claim that $u$ has Jordan form $[J_2^m]$. To see this, observe that $u$ has no $J_1$ blocks since it has no nonsingular $1$-eigenvectors, and since it has a $m$-dimensional $1$-eigenspace there must be $m$ distinct Jordan blocks. If $p$ is odd, then this determines $u$ up to $G_0$-conjugacy. If $p=2$, then since $u$ is centralised by an element $s$ acting irreducibly on a $m$-space, $u$ is again determined up to $G_0$-conjugacy ($u$ is an $a_m$ involution, in the notation of \cite{ref:AschbacherSeitz76}). Without loss of generality, we may write $W = \< e_1, \dots, e_m \>$ and $W^* = \< f_1, \dots, f_m \>$. Now $C_{G_0}(u) = Q{:}D$ where $|Q| = q_0^{m(m-1)/2}$ and $D = \Sp_m(q_0)$, which acts diagonally on $W \oplus W^*$. Therefore, $s = g \oplus g^{-\tr}$ with respect to the decomposition $V = W \oplus W^*$, with $g \in \Sp_{2m}(q_0)$ acting irreducibly on $W$, so the order $r$ of $s$ satisfies $r \div (q_0^{m/2}+1)$ and $r \ndiv (q_0^{m/2}-1)$ (note that $g = g^{-\tr}$ since $g \in \Sp_{2m}(q_0)$). 

Continue to assume that $U=W$, but now also assume that $p=2$. For now, in addition, assume that $f/j$ is even. From the description of $x_0$, we see that $x_0$ is contained in a field extension subgroup $\Omega^+_m(q_0^2)$, so $x_0 \in Y_{y\s_2} = Y_{y\s_2}\ws_1$ (using the fact that $f/j$ is even) for some $\s_2$-stable subgroup $Y \leq X$ of type $\Omega_m^2$. Therefore, by Theorem~\ref{thm:shintani_cosets}, we know that $x \in Y_{\s_1}y\ws_2$, where $y$ interchanges the two factors of $Y$, so $x$ is contained in a subgroup of $G$ of type $\O^\pm_m(q) \wr \Sym_2$, which is maximal by \cite[Table~3.5E]{ref:KleidmanLiebeck}. Now assume that $f/j$ is odd. Here we note that there exists a subgroup chain $H_0 \leq M_0 \leq G_0$ where $H_0 = \Omega^-_m(q_0)$ and $M_0 = \Omega^+_m(q_0^2)$ such that $s \in H_0$. Moreover, we can choose these subgroups such that $u \in N_{G_0}(M_0) \setminus M_0$. In particular, $x_0 \in N_{G_0}(M_0)$ of type $\Omega^+_m(q_0^2).2$, so for a suitable $\s_2$-stable subgroup $Y \leq X$ of type $\Omega_m^2$, we have $x_0 \in Y_{u\s_2}\ws_1$ (using the fact that $f/j$ is odd). Therefore, by Theorem~\ref{thm:shintani_cosets}, we know that $x \in Y_{\s_1}u\ws_2$, so $x$ is contained in a subgroup of $G$ of type $\O^-_m(q) \wr \Sym_2$, which, again, is maximal by \cite[Table~3.5E]{ref:KleidmanLiebeck}.

To summarise, for the remainder of the proof we may assume that $x_0=su=us$, where $r = |s|$ satisfies $r \div (q_0^m-1)$ and $r \ndiv (q_0^{m/2}-1)$ if $m$ is even, and either $u=1$ and $r \ndiv 2(q_0^{m/2}+1)$ if $m$ even, or $p$ is odd, $m$ is even, $u$ has Jordan form $[J_2^m]$ and $r \div (q_0^{m/2}+1)$. Let us write $\s = \s_2$ and $e = f/j$, so $q=q_0^e$.

Let $k$ be a prime divisor of $m$, and assume that $m/k$ is even if $u \neq 1$. Our assumptions imply that $x_0$ is contained in a field extension subgroup $H_0 = \SO^+_{2m/k}(q_0^k)$. Let $Y$ be a closed $\s_2$-stable subgroup of type $\Omega_{2m/k}^k$ that contains a maximally split maximal torus of $X$, so $H_0 = Y_{y\s}$ for a $k$-cycle $y$ acting on the factors of $Y$. We apply this observation in two cases.

First, assume that there exists an odd prime divisor $k$ of $m$ such that $(e,k)=1$. Note that if $u \neq 1$, then $m$ is even, so $m/k$ is even. In this case, $Y_{y\s} = Y_{y\s}y^{-e}\ws^e$ and $|y^{-e}|=k$ since $(e,k)=1$. Now, by Theorem~\ref{thm:shintani_cosets}, $x$ is contained in $Y_{y^{-e}\s^e}y\s$, which implies that $x$ is contained in a subgroup of type $\O^+_{2m/k}(q^k).k$, which is maximal by \cite[Table~3.5E]{ref:KleidmanLiebeck}, noting that $k$ is odd. Therefore, we now assume that every odd prime divisor of $m$ divides $e$.

Next assume that $(e,m) > 1$ and let $k$ be the largest prime divisor of $(e,m)$. Note that if $u \neq 1$, then $m$ is even, so if $k$ is odd, then $m/k$ is even, and if $k=2$, then $m/2$ is even, since otherwise, $m$ would have an odd prime divisor not dividing $e$, which is a contradiction to our working assumption. In this case, $Y_{y\s} = Y_{y\s}\ws^e$, noting that $k$ divides $e$. Now, by Theorem~\ref{thm:shintani_cosets}, $x$ is contained in $Y_{\s^e}y\s$, which implies that $x$ is contained in a subgroup of type $\O^+_{2m/k}(q) \wr \Sym_k$, which is maximal by \cite[Table~3.5E]{ref:KleidmanLiebeck}. Therefore, from now on we will assume that $(e,m) = 1$.

Combining the assumptions that every odd prime divisor of $m$ divides $e$ and also that $(e,m)=1$, establishes that $m$ is a power of $2$ and $e$ is odd.  Since $x$ is not in Theorem~\ref{thm:derangement}, either there exists a (necessarily odd) prime divisor $k$ of $j$ not dividing $e$ such that $r \div (q_0^{m/k}-1)$, or there there exists a (necessarily odd) prime divisor $k$ of $e$ such that $k \ndiv r$ or $k \div (q_0^m-1)/r$. By Corollary~\ref{cor:shintani_subfields}, both of these conditions guarantee that $x$ is contained in a subfield subgroup of type $\O^+_{2m}(q^{1/k})$, which yields a core-free maximal subgroup containing $x$ (see \cite[Table~3.5E]{ref:KleidmanLiebeck}, noting that $k$ is odd). Therefore, we have deduced that $x$ is contained in a core-free maximal subgroup of $G$, so $x$ is not totally deranged.
\end{proof}

\begin{proposition} \label{prop:d4}
Let $T = \POm^+_8(q)$ and $T \leq G \leq \Aut(T)$. Assume that $G$ contains triality. Let $x \in G$. Assume that $x \in \< \PGO^+_8(q), \p \>$ and $x$ is not in Theorem~\ref{thm:derangement}. Then $x$ is not totally deranged.
\end{proposition}

\begin{proof}
Write $q=p^f$. By replacing $x$ by another generator of $\<x\>$ if necessary, assume that $x \in \PGO^+_8(q)\wp^j$ for a divisor $j$ of $f$. Fix the simple algebraic group $X = D_m$ and the Steinberg endomorphisms $\s_1 = \p^f$ and $\s_2 \in \{ \p^j, \g\p^j \}$, so $\Inndiag(T) = X_{\s_1}$ and $x \in \Inndiag(T)\ws_2$. Let $F\: \Inndiag(T)\ws_2 \to \Inndiag(T_0)\ws_1$ be the Shintani map of $(X,\s_1,\s_2)$ where $T_0 = \POm^\e_8(q_0)$ for $q_0 = p^j$ and $\e \in \{+,-\}$, and let $x_0 = F(x)$. Note that if $x \in \Inndiag(T)\wp^j$, then $x_0 \in \Inndiag(\POm^+_8(q_0))$, and if $x \in \Inndiag(T)\g\wp^j$, then $x_0 \in \Inndiag(\POm^-_8(q_0))\g^{f/j}$. Write $G_0 = \PGO^\e_8(q)$ and let $V = \F_{q_0}^8$ be the natural module for $G_0$. Write $x_0 = su = us$ for elements $s,u \in G_0$ such that $s$ is semisimple and $u$ is unipotent. Note that $|x| = f/j \cdot |x_0|$.

If $x_0$ is unipotent, then $x_0$ contained in a Borel subgroup and hence a $P_2$ parabolic subgroup of $G_0$. Therefore, by Theorem~\ref{thm:shintani_subgroups}, $x$ is contained in a maximal $P_2$ parabolic subgroup of $G$.

From now on, we will assume that $s$ is nontrivial. For now assume that there exists $s_0 \in \< s \>$ of prime order such that $C_{\Inndiag(T_0)}(s_0) \leq Y_{\s_2}$ for a closed $\s_2$-stable subgroup $Y \leq X$ of type $A_2S^2$ or $A_1^4$, with $Y_{\s_2}$ of type $\SL_2(q_0)^4$ and $\s_2 = \p^j$ in the latter case. Then $x_0 \in Y_{\s_2}\ws_1$, so Theorem~\ref{thm:shintani_subgroups} implies that $x$ is contained in a maximal subgroup of $G$ of type $\O^\pm_2(q) \times \GL^\pm_3(q)$ or $\O^+_4(q) \wr \Sym_2$. Similarly, we next assume that there exists $s_0 \in \< s\>$ of prime order such that $C_{\Inndiag(T_0)}(s_0) \leq Y_{\s_2}$ for a $\s_2$-stable maximal torus $Y \leq X$ such that $Y_{\s_2}$ has order $(q_0^2+1)^2$ with $\s_2 = \p^j$, or order $q_0^4+1$ with $\s_2 = \g\p^j$ and $f/j$ even. Then $x_0 \in Y_{\s_2}\ws_1$ and $\s_1 = \s_2^{f/j}$, so Theorem~\ref{thm:shintani_subgroups} implies that $x$ is contained in $N_G(H)$ where $H$ is a maximal torus of $G$ of order $(q^2+1)^2$ or $(q \pm 1)^4$; in all cases, $N_G(H)$ is maximal (see \cite[Table~8.50]{ref:BrayHoltRoneyDougal}).

Considering the possible centralisers of semisimple elements of prime order in $\PGO^\pm_8(q_0)$ (see \cite[Tables~B.10--B.12]{ref:BurnessGiudici16} for example), we deduce that it only remains to consider the case where $x_0 \in \Inndiag(T_0)$ for $T_0 = \POm^+_8(q_0)$ and $s \in \< s'\>$ where $s' \in \Inndiag(T_0)$ is an element of order $q_0^4-1$. Write $H_0 = C_{\Inndiag(T_0)}(s)$ and $\s = \s_2 = \p^j$. 

First assume that $|s|$ divides $q_0 \pm 1$, so $H_0$ has type $\GL^\pm_4(q_0)$. Here we proceed as in Case~4 of Lemma~\ref{lem:bfg_mixed}, using the same notation for tori. Let $B$ be the Borel subgroup of $X$ defined as $B = \< S, U_\alpha \mid \alpha \in \Phi^+ \>$, and let $U$ be the unipotent radical of $B$. Let $P$ be the $P_2$ parabolic subgroup $P = \< S, U_\alpha \mid \alpha \in \Phi^+ \cup \Phi_{1,3,4} \>$, and let $L$ be the Levi complement of $P$, so $L = \< S, U_\alpha \mid \alpha \in \Phi_{1,3,4}\>$. Observe that $H_0 = A^\pm_3(q_0) \times \< s \> = \< S_1, S_2, S_3, U_\alpha \mid \alpha \in \Phi_{1,2,3} \>_\s \times \< s \>$ and $s \in \< S, U_\alpha \mid \alpha \in \Phi_4 \>_\s = (S_1 \times S_2 \times S_3)_\s \times A_1(q_0) \leq L_\s \leq P_\s$. By conjugating in $H_0$ if necessary, we may assume that $u \in \< U_\alpha \mid \alpha \in \Phi^+_{1,2,3}\>_\s \leq U_\s \leq P_\s$, so $x_0 = su \in P_\s$. As before, $x$ is contained in a maximal $P_2$ parabolic subgroup of $G$.

Next assume that $|s|$ divides $q_0^2-1$ but does not divide $q_0 \pm 1$, so $H_0$ has type $\GL_2(q_0^2).2$ if $|s|$ divides $2(q_0 \pm 1)$ (so, in particular, $p \neq 2$) and $\GL_2(q_0^2)$ otherwise. In either case, $x_0$ is contained in the subgroup of $H_0$ of type $\GL_2(q_0^2)$. In particular, this means that $x_0$ stabilises a totally singular $2$-space of $V_0$, so $x_0$ is contained in a $P_2$ parabolic subgroup of $G_0$. Therefore, by Theorem~\ref{thm:shintani_subgroups}, $x$ is contained in a maximal $P_2$ parabolic subgroup of $G$.

Now assume that $|s|$ divides $2(q_0^2+1)$. Since $x$ is not in Theorem~\ref{thm:derangement}, either $u=1$ or $p=2$. In either case, $|x_0|$ divides $2(q_0^2+1)$, and we will write $y_0 = x_0$ if $|x_0|$ divides $q_0^2+1$ and $y_0 = x_0^2$ otherwise, so $|y_0|$ divides $q_0^2+1$ in both cases. Now fix a maximally split $\s$-stable maximal torus $Y \leq X$ and fix $\a \in N_X(Y)$ such that $|Y_{\a\s}| = (q_0^2+1)^2$ and $y_0 \in Y_{\a\s}$. If $x_0 = y_0$, then $x_0 \in Y_{\a\s}(\a\ws)^e$, so by Theorem~\ref{thm:shintani_cosets}, $x \in Y_{(\a\s)^e}\a\ws$, which means that $x$ is contained in $N_G(H)$ where $H$ is a torus of order $(q^2+1)^2$ or $(q \pm 1)^4$, which in all cases is maximal. Now assume that $x_0 \neq y_0$. Since $x_0$ and $y_0$ commute, we see that $y_0$ interchanges the two factors of $Y_{\a\s}$, so we may write $y_0 \in Y_{\a\s}t(\a\ws)^e$ for $t \in Y$. Theorem~\ref{thm:shintani_cosets} implies that $x \in Y_{t(\a\s)^e}\a\ws$. Since $Y_{t(\a\s)^e}$ has order $(q^2-1)^2$ or $(q^2+1)^2$, we deduce that $x$ is contained in a maximal subgroup of of $G$ of type $\O^+_4(q) \wr \Sym_2$ or $\O^-_2(q^2) \wr \Sym_2$ (see \cite[Table~8.50]{ref:BrayHoltRoneyDougal}).

Finally, it remains to assume that $s$ has order dividing $q_0^4-1$ but not dividing $q_0^2-1$ or $2(q_0^2+1)$, which, in particular, forces $x_0 = s$. For now assume further than $e$ is even. Then $x_0 \in Y_\s$ for a closed $\s$-stable maximal torus $Y$ of $X$ such that $Y_\s$ has order $q_0^4-1$. Theorem~\ref{thm:shintani_subgroups} implies that $x$ is contained in $Y_{\s^e}\ws$. Since $e$ is even, $Y_{\s^e}$ has order $(q^2-1)^2$ or $(q-1)^4$, which means that $x$ is contained in a maximal subgroup of $G$ of type $\O^+_4(q) \wr \Sym_2$ or $\O^+_2(q) \wr \Sym_4$. 

Now assume that $e$ is odd. We complete the proof just as for Proposition~\ref{prop:d}. Since $x$ is not in Theorem~\ref{thm:derangement}, either there exists an odd prime divisor $k$ of $j$ not dividing $e$ such that $r \div (q_0^{4/k}-1)$, or there there exists an odd prime divisor $k$ of $e$ such that $k \ndiv r$ or $k \div (q_0^4-1)/r$.  Therefore, Corollary~\ref{cor:shintani_subfields} implies that $x$ is contained in a subgroup of type $\O^+_8(q^{1/k})$, which is maximal since $k$ is odd (see \cite[Table~8.50]{ref:BrayHoltRoneyDougal}). In conclusion, $x$ is contained in a core-free maximal subgroup of $G$, so $x$ is not totally deranged.
\end{proof}

\subsection{\boldmath Verifying the examples} \label{ss:proof_examples_proof}

We now verify that if $(G,x)$ appears in part~(iii) of Theorem~\ref{thm:derangement}, then $x$ actually is a totally deranged element of $G$. Here we need more comprehensive information on the subgroup structure of $G$. Let $T \leq G \leq \Aut(T)$ for $G$ and $T$ as in Theorem~\ref{thm:derangement}. Complete information is given in \cite[Table~8.14]{ref:BrayHoltRoneyDougal} if $T = \Sp_4(2^f)$ and in \cite[Table~8.50]{ref:BrayHoltRoneyDougal} if $T = \POm^+_8(q)$. For $m \geq 4$ and $\POm^+_{2m}(q) \leq G \leq \PGO^+_{2m}(q)$, we use the (incomplete) description of the maximal subgroups of $G$ in \cite{ref:KleidmanLiebeck}, namely that each core-free maximal subgroup is either a geometric subgroup in \cite[Table~3.5E]{ref:KleidmanLiebeck} or an absolutely irreducible almost simple group in the class $\S$. The class $\S$ is defined formally in \cite[p.3]{ref:KleidmanLiebeck}, and we will make use of the detailed information on almost simple groups given in \cite[Chapter~5]{ref:KleidmanLiebeck}. 

\begin{lemma} \label{lem:ex_b2}
Let $T = \Sp_4(q)$ for $q=p^f$, and let $T \leq G \leq \Aut(T)$. Assume that $(G,x)$ is in Theorem~\ref{thm:derangement}. Then $x \not\in N_G(H)$ for any $H \leq T$ of type $\Sp_4(q^{1/k})$ for prime $k \div f$ or ${}^2B_2(q)$ for odd $f$.
\end{lemma}

\begin{proof}
By replacing $x$ by another generator of $\<x\>$ if necessary, we may assume that $x \in T\wp^j$ for a divisor $j$ of $f$. Let $F\: T\wp^j \to T_0$ be the Shintani map of $(\Sp_4, \p^f, \p^j)$ where $T_0 = \Sp_4(q_0)$ and $q_0 = 2^j$. Write $x_0 = F(x)$, noting that $|x_0| = |x^{f/j}|$.

First assume that $M = N_G(H)$ where $H$ has type $\Sp_4(q^{1/k})$ for a prime $k \div f$. For now assume further that $k$ divides $f/j$. Then the conditions on $|x_0|$ imply that there is no $z \in T_0$ such that $x_0 = z^k$. Now assume that $k$ does not divide $f/j$ (so $k$ divides $j$). We claim that $x_0$ is not contained in a $T_0$-conjugate of $\Sp_4(q_0^{1/k})$. To see this, suppose otherwise. Then $|x_0|$ must divide either $q_0^{2/k}+1$ or $q_0^{2/k}-1$, but the first of these is impossible since $(q_0^{2/k}+1,q_0^2-1) = 1$ and the second is excluded by the statement of Theorem~\ref{thm:derangement}. Therefore, in both cases, Corollary~\ref{cor:shintani_subfields} (with $\sigma = \p$, $m=f$ and $l = j$) shows that $x \not\in M$.

Now assume that $M = N_G(H)$ where $H$ has type ${}^2B_2(q)$, so $f$ is odd. Suppose that $x_0$ is contained in a $T_0$-conjugate of ${}^2B_2(q_0)$. Then, as before, $|x_0|$ must divide $q_0^2+1$ or $q_0-1$, but the first is impossible since $(q_0^2+1,q_0^2-1)=1$ and the second is excluded. Again, Corollary~\ref{cor:shintani_subfields} (this time with $\sigma = \rho$, $m=2f$ and $l = 2j$) shows that $x \not\in M$.
\end{proof}

\begin{lemma} \label{lem:ex_d}
Let $T = \POm^+_{2m}(q)$ for $q=p^f$ and $2m = 2^l \geq 8$, and let $T \leq G \leq \Aut(T)$. Assume that $(G,x)$ is in Theorem~\ref{thm:derangement}. Then $x \not\in N_G(H)$ for any $H \leq T$ satisfying one of the following
\begin{enumerate}[1.]
\item $H$ is a reducible subgroup of type \vspace{-6pt}

\begin{enumerate}[{\rm (a)}]
\item the stabiliser of a totally singular $k$-space for $1 \leq k < m$,
\item $\O^\e_k(q) \times \O^\e_{2m-k}(q)$ for $1 \leq k < m$ and $\e \in \{+,-\}$
\item $\Sp_{2m-2}(q)$ for odd $q$
\end{enumerate}
\item $H$ has type $\O^\e_{2m/k}(q) \wr \Sym_k$ for $k \div 2m$ and $\e \in \{+,-\}$
\item $H$ is a subfield subgroup of type \vspace{-6pt}

\begin{enumerate}[{\rm (a)}]
\item $\O^+_{2m}(q^{1/k})$ for prime $k \div f$
\item $\O^-_{2m}(q^{1/2})$ for $2 \div f$ 
\item ${}^3D_4(q^{1/3})$ for $m=4$ and $3 \div f$
\end{enumerate}
\item $H$ is an $\S$-type subgroup.
\end{enumerate}
\end{lemma}

\begin{proof}
By replacing $x$ by another generator of $\<x\>$ if necessary, we may assume that $x$ is contained in $\Inndiag(T)\wp^j$ for a divisor $j$ of $f$. Let $F\: \Inndiag(T)\wp^j \to \Inndiag(T_0)$ be the Shintani map of $(X,\s^e,\s) = (\POm_{2m}, \p^f, \p^j)$ where $T_0 = \POm^+_{2m}(q_0)$ and $q_0 = p^j$ (so $e=f/j$). Write $x_0 = F(x)$ noting that $|x_0| = |x^e|$. In particular, $x_0 = su$ with the notation from Theorem~\ref{thm:derangement}. Write $G_0 = \Inndiag(T_0)$ and let $V = \F_{q_0}^{2m}$ be the natural module for $G_0$. 

\emph{\textbf{Cases~1 \&~2.} $H$ is a reducible or imprimitive subgroup.}\nopagebreak

Suppose that $x \in N_G(H)$. Fix a closed $\s$-stable subgroup $Y \leq X$ such that $H = Y_{\s^e} \cap T$. Now $x \in Y_{\s^e}\ws$, so fix $\a \in Y$ such that $x \in Y^\circ_{\s^e}\a\ws$. In the case $Y = \O_{2m/k} \wr \Sym_k$ we may assume that $\s$ does not permute the $k$ factors of $Y^\circ$. By Theorem~\ref{thm:shintani_cosets}, we deduce that $x_0 \in Y^\circ_{\a\s}\ws^e$. The only proper nonzero subspaces of $\F_{q_0}^{2m}$ stabilised by $x_0^a$ (for any even $a$) are totally singular $m$-spaces (here we are using the fact that $u \neq 1$ if $p=2$). Therefore, the only possibility is that $Y^\circ_{\a\s}$ has type $\Omega^+_{2m/k}(q_0^k)$ for some $k$ dividing $m$. This means that $Y$ has type $\O_{2m/k} \wr \Sym_k$ and $\a \in \Sym_k$ is a $k$-cycle. In particular, since $e$ is odd, we deduce that $x_0 \in Y^\circ_{\a\s}\b$ where $\b \in \Sym_k$ is a $k$-cycle. Said otherwise, $x_0 \in \Omega^+_{2m/k}(q_0^k)\psi$ where $\psi$ induces field automorphism of order $k$. The usual Shintani descent argument implies that $\Omega^+_{2m/k}(q_0)$ must contain an element of order $|x_0^k|$. However, since $k \leq m$, we know that $|x_0^k|$ is divisible by a primitive prime divisor of $q_0^m-1$, which is a contradiction.  

\vspace{0.3\baselineskip}

\emph{\textbf{Case~3.} $H$ is a subfield subgroup.}\nopagebreak

First assume that $H$ has type $\O^+_{2m}(q^{1/k})$ for a prime divisor $k$ of $f$. If $k$ divides $f/j$, then the conditions on $|x_0|$ imply that there is no element $z \in \Inndiag(T_0)$ such that $x_0 = z^k$. Now assume that $k$ does not divide $f/j$ (so $k$ divides $j$). We claim that $x_0$ is not contained in a subgroup of $\Inndiag(T_0)$ of type $\O^+_{2m}(q_0^{1/k})$. (To see this, suppose otherwise. Then $s$ stabilises a dual pair of totally singular $m$-spaces of $\F_{q_k}^{2m}$ where $q_k = q_0^{1/k}$, so $|s| \div q_0^{2m/k}-1$, which is a contradiction.) Therefore, in both cases, Corollary~\ref{cor:shintani_subfields} (with $\sigma = \p$, $m = f$ and $l = j$) shows that $x \not\in N_G(H)$.

Next assume that $H$ has type $\O^-_{2m}(q^{1/2})$. We claim that $x_0$ is not contained in a subgroup of $\Inndiag(T_0)$ of type $\O^-_{2m}(q_0^{1/2})$. To see this, suppose otherwise. Then $s$ must act irreducibly on $\F_{q_2}^{2m}$, where $q_2 = q_0^{1/2}$, so $|s| \div q_0^{m/2}+1$. Now the conditions in Theorem~\ref{thm:derangement} imply that $u \neq 1$, but irreducible elements of $\O^-_{2m}(q_0^{1/2})$ commute with no nontrivial unipotent elements, which is a contradiction and the claim is proved. Now, recalling that $f/j$ is odd, by applying Corollary~\ref{cor:shintani_subfields} (with $\sigma = \g\p^{j/2}$, $m=2f/i$ and $l = 2$) we deduce that $x \not\in N_G(H)$. 

Now assume that $m=4$ and $H$ has type ${}^3D_4(q^{1/3})$. In this case, it suffices to note that since $q = q_0^e$ and $e$ is odd, $|{}^3D_4(q^{1/3})|$ is not divisible by a primitive prime divisor $r$ of $q_0^4-1$ since any such $r$ divides $q_0^2+1$ and $q_0^2+1$ is prime to $|{}^3D_4(q^{1/3})|$.

\vspace{0.3\baselineskip}

\emph{\textbf{Case~4.} $H$ is an $\S$-type subgroup.}\nopagebreak

Let $S$ be the socle of $H$. Observe that $|x_0|$ (and hence $|x|$) is divisible by a primitive prime divisor $r$ of $q_0^m-1$. Now $r = am+1$ for some $a \geq 1$, so, if $m \geq 8$, then $r \geq 17$. Moreover, if $m \geq 16$, then $r = 17$ (in which case $m=16$) or $r > 100$. Note also that $|x_0|$ is not prime, so $|x| \geq |x_0| > r$. (If $|x_0|$ were prime, then $x_0 = su$ with $|s|$ prime and $u=1$, but then $|s|$ divides $q_0^{m/2}+1$, which is at odds with the specification of $x_0$ in Theorem~\ref{thm:derangement}.)

\vspace{0.3\baselineskip}

\emph{\textbf{Case~4a.} $S$ is a sporadic group.}\nopagebreak

If $m=4$, then no cases arise \cite[Table~8.50]{ref:BrayHoltRoneyDougal}, and if $m=8$, then $S={\rm M}_{12}$ \cite[Table~7.8]{ref:Rogers17}, but the prime $r \geq 17$ does not divide $|\Aut(S)|$, so again no cases arise. We can now assume that $m \geq 16$. Since $|\Aut(S)|$ is not divisible by any prime greater than $100$, we must have $r=17$ and hence $m=16$. Since $r$ divides $|\Aut(S)|$ we must have $S \in \{ {\rm J}_3, {\rm He}, {\rm Fi}_{23}, {\rm Fi}_{24}', \mathbb{B}, \mathbb{M} \}$. Every nontrivial representation of $S$ has degree greater than 32 \cite[Proposition~5.3.8]{ref:KleidmanLiebeck}, except when $S={\rm J}_3$, but $S$ does not have an irreducible representation of degree exactly 32 in this case either (see \cite[p.83]{ref:ATLAS} and \cite[pp.215--219]{ref:BrauerATLAS}).

\vspace{0.3\baselineskip}

\emph{\textbf{Case~4b.} $S$ is an alternating group $\Alt_d$.}\nopagebreak

First assume that $S$ is embedded via the fully deleted permutation module, so $q=q_0=p$ (see \cite[pp.185--187]{ref:KleidmanLiebeck}). Let $U = \F_p^d$ and write 
\begin{gather*}
U_0 = \{ (a_1,\dots, a_d) \in U \mid a_1 + \dots + a_d = 0 \} \\
W = \{ (a_1, \dots, a_d) \in U \mid a_1 = \dots = a_d \}.
\end{gather*}
Then we may identify $V = U_0/(W \cap U_0)$, so $n=d+1$ if $p \ndiv d$ and $n=d+2$ if $p \div d$. 

For now assume that $u = [J_2^m]$.  The element of $\Alt_d$ corresponding to $u$ has order $p$ and hence cycle type $(p^k,1^{d-pk})$. The Jordan form for the corresponding element in $\GL(U)$ is $[J_p^k,J_1^{d-pk}]$ and hence for the corresponding element in $\GL(U_0)$ it is $[J_p^{k-1},J_{p-1},J_1^{d-pk}]$. If $V \neq U_0$, then the corresponding element of $\GL(V)$ it is $[J_{p-2}]$ if $d=p$, $[J_p^{k-2},J_{p-1}^2]$ if $d=kp > p$ and $[J_p^{k-1},J_{p-1},J_1^{d-pk-1}]$ if $d > kp$; all of these are inconsistent with $u = [J_2^m]$. 

We may now assume that $u=1$. Let $\alpha$ be the element of $\Alt_d$ corresponding to $x_0 = s$ and let $\alpha_0$ be a power of $\alpha$ of order $r$. Since $r \equiv 1 \mod{m}$ and $d \leq 2m+2$ we must have $r \in \{m+1,2m+1\}$. Since $|\alpha| \ndiv (p^{m/2}+1)$, we know that $|\alpha| > r$, so actually $r = m+1$. For now suppose that $\alpha_0$ has cycle type $[r^2]$, so $d=2m+2$ and $p \div d$. Then $\alpha$ has cycle type $[2r]$, which implies that $2r \div (p^m-1)$. In particular, $p \not\in \{2,r\}$, but $d=2r$ is supposed to be divisible by $p$: a contradiction. Now suppose that $\alpha_0$ has cycle type $[r,1^{d-r}]$. Since $d-r \geq m > 0$, the element of $\GL(V)$ corresponding to $\alpha_0$ fixes a $1$-space of $V$, which is impossible as it must act irreducibly on a dual pair of totally singular $m$-spaces. 

It remains to assume that $S$ is not embedded via the fully deleted permutation module. No cases arise when $m=4$ \cite[Table~8.50]{ref:BrayHoltRoneyDougal}. Now assume that $m \geq 8$, so $d \geq r \geq 17$. Combining \cite[Lemma~4.3]{ref:FawcettOBrienSaxl16} and \cite[(6.1)]{ref:Muller16}, we see that since $d \geq 17$ and since the embedding is not the fully deleted permutation module, the degree of the representation must satisfy $2m \geq 2d$. This implies that $r \geq m+1 > d$, which is a contradiction.

\vspace{0.3\baselineskip}

\emph{\textbf{Case~4c.} $S$ is a group of Lie type over $\F_t$ where $(p,t)=1$.}\nopagebreak

For each possible type of $S$, we will show that the bound $r > m$ is contradicted. A lower bound on the degree of a nontrivial representation of $S$ is given in \cite[Theorem~5.3.9]{ref:KleidmanLiebeck}. To obtain upper bounds on the size of the prime divisors of $\Aut(S)$, we note that every prime divisor of $|\Aut(S)|$ is at most $\max\{3,t\}$ or divides one of the cyclotomic polynomials in $t$ given by the $t'$-part of the order formula for $|\Inndiag(T)|$. For exceptional groups and orthogonal groups, this easily yields the desired contradiction. For instance, if $S = E_6(t)$, then $2m \geq t^9(t^2-1)$ but $r$ is at most the greatest cyclotomic polynomial dividing $|\Inndiag(E_6(t))|$ which is $t^6+t^3+1$, which contradicts $r > m$. For the remaining classical groups, more care is required but the arguments are similar in all cases and we present the details for $S = \PSL_d(t)$ where $d$ is even. No examples arise when $m=4$ (see \cite[Table~8.50]{ref:BrayHoltRoneyDougal}), so assume that $m \geq 8$. 

First assume that $d > 2$. On the one hand, $2m \geq t^{d-1}-1$ (from \cite[Theorem~5.3.9]{ref:KleidmanLiebeck}). On the other hand, $r$ divides $(t^i-1)/(t-1)$ for some $1 \leq i \leq d$, and since $r > m \geq \frac{1}{2}(t^{d-1}-1)$, $i=d$. Therefore, $r$ divides $t^{d/2} \pm 1$, so (except for the case $S = \PSL_4(2)$, which was handled in Case~4b as $\Alt_8$) we deduce that \[
r \leq t^{d/2}+1 \leq \tfrac{1}{2} (t^{d-1}-1) \leq m,
\] 
which contradicts $r > m$. 

We now assume that $d=2$, and low rank isomorphisms allow us to assume that $t=8$ or $t \geq 11$. Let us record that every prime divisor of $|\Aut(S)|$ greater than $\log_2{t}$ must divide $t-1$, $t$ or $t+1$. Recall that $2m=2^l$. We claim that $r = t+1$ if $t$ is even, and $r=t$ if $t$ is odd.

To see this, first assume that $t$ is even. From \cite{ref:Burkhardt76}, the degrees of the nontrivial irreducible representations of $S$ are $t-1$, $t$ and $t+1$, so $2m=t$. Since $r \equiv 1 \mod{m}$, fix $a$ such that $r = am+1$, noting that $a \in \{1,2\}$ since $r$ divides $|\Aut(S)|$. If $r = m+1 = t/2+1$, then $r$ does not divide $t-1$, $t$ or $t+1$, so $r = 2m+1 = t+1$.

Now assume that $t$ is odd. Here, the degrees of the nontrivial irreducible representations of $S$ are $\frac{1}{2}(t-\e)$, $t$, $t-1$ and $t+1$, where $\e \in \{1,-1\}$ satisfies $t \equiv \e \mod{4}$. This means that $2m \in \{\frac{1}{2}(t \pm 1), t \pm 1\}$, so $t \in \{ 2^l \pm 1, 2^{l+1} \pm 1 \}$. As $r \equiv 1 \mod{m}$ and $r \leq t+1$ we see that $r \in \{ 2^{l-1}+1, 2^l+1, \frac{3}{2} \cdot 2^l+1, 2^{l+1}+1 \}$. Moreover, since $r$ divides $|\Aut(S)|$, we quickly see that either $r=t$ or $(r,t) \in \{ (2^{l-1}+1,2^l+1), (2^l+1,2^{l+1}+1) \}$ where $r$ divides $t+1$. However, in the latter case, $r$ and $t$ are both Fermat primes, so $l$ and either $l-1$ or $l+1$ is a power of two, which is absurd. Therefore, $r=t$, and we have proved the claim.

It now remains to note that $|x_0| > r$ but $\Aut(S)$ does not contain any elements of order properly divisible by $r \in \{t,t+1\}$: a contradiction. 

\vspace{0.3\baselineskip}

\emph{\textbf{Case~4d.} $S$ is a group of Lie type in characteristic $p$.}\nopagebreak

By \cite[Theorem~5.4.1]{ref:KleidmanLiebeck}, we can fix a closed $\s$-stable simple subgroup $Y \leq X$ such that $H = Y_{\s^e} \cap T$. Now $x \in Y_{\s^e}\ws$, so we can fix $y \in Y$ such that $x \in Y^\circ_{\s^e}y\ws$. By Theorem~\ref{thm:shintani_cosets}, we deduce that $x_0 \in Y^\circ_{y\s}\ws^e$. Let $S_0$ be the socle of $Y^\circ_{y\s}$. We will prove that $x_0 \not\in \Aut(S_0)$. 

Write $S_0 = \Sigma_n(p^{j_0})$ and recall our notation that $x_0 \in \Inndiag(\POm^+_{2m}(q_0))$ where $q_0 = p^j$. By \cite[Theorem~5.4.6 \& Remark~5.4.7]{ref:KleidmanLiebeck}, we have $2m \geq \dim{M}^{j_0/j}$ for a nontrivial irreducible $\F_{p^j}S_0$-module $M$. Let $\a = j_0/j$ and $\b = \max\{i_2 \mid (q_0^{\a i}-1) \div |\Aut(S_0)|\}$, and note that $\a \geq m/\b$ since $x_0 \in \Aut(S_0)$ and $|x_0|$ is divisible by a primitive prime divisor of $q_0^m-1$. Let $\d$ be the smallest degree of a nontrivial irreducible $\F_{p^j}S_0$-module. Therefore, if we can show that $\d^{1/\b} > (2m)^{1/m}$ for all $m \geq \d/2$, then it follows that $2m < \dim{M}^\a$, which is a contradiction. For exceptional groups, it is easy to verify this bound using the degree bounds in \cite[Proposition~5.4.13]{ref:KleidmanLiebeck} and the order formulae in \cite[Table~5.1.B]{ref:KleidmanLiebeck} (together with the fact that $z \mapsto z^{2/z}$ is decreasing on $[\mathrm{e},\infty)$, where $\mathrm{e}$ is the base of the natural logarithm). For instance, if $S_0 = E^\pm_6(q_0^\a)$, then $\d = 27$ and $\b = 8$, so $\d^{1/\b} = 27^{1/8}$, which exceeds $(2m)^{1/m}$ for all $m \geq 8$. Again classical groups require more attention, but since the arguments are similar in all cases and we just give the details for $S_0 = \PSL_d(p^{j_0})$. 

We first handle some cases in small dimension. If $m=4$, then $S_0 = \PSL_3(q_0)$ with $q_0 \equiv 1 \mod{3}$ and $M$ is the $8$-dimensional adjoint module \cite[Table~8.50]{ref:BrayHoltRoneyDougal}, but $\a=1$ and $\b=2$, which contradicts $\a \geq m/\b$. If $m=8$, the by \cite[Table~7.8]{ref:Rogers17} the only possibility is $S_0 = \PSL_2(q_0^4)$, but here the only element orders divisible by a primitive prime divisor of $q_0^8-1$ necessarily divide $q_0^4+1$, which makes this case impossible. From now on we assume that $m \geq 16$. For $S_0 = \PSL_d(q_0)$, we have $\d = d$ and $\b \leq d$, so $\d^{1/\b} \geq d^{1/d}$. If $d \leq 5$, then since $2m \geq 32$, we have 
$
(2m)^{1/m} \leq 32^{1/16} < d^{1/d},
$
a contradiction. Now assume that $d > 5$. If $M$ is the natural module for $S_0$, then $\a > 1$, so $2m \geq d^2$ and we deduce that 
$
(2m)^{1/m} \leq (d^2)^{2/d^2} < d^{1/d},
$
another contradiction. Finally, assume that $M$ is not the natural module, which means that $\dim{M} \geq \frac{1}{2}d(d-1)$ (see \cite[Proposition~5.4.11]{ref:KleidmanLiebeck}). In this case, we have 
$
(2m)^{1/m} \leq (\tfrac{1}{2}d(d-1))^{4/d(d-1)} \leq (\tfrac{1}{2}d(d-1))^{1/d} < \dim{M}^{1/d} \leq \dim{M}^{1/\b},
$
which contradicts $2m \geq \dim{M}^\a$. This completes the proof.
\end{proof}

\begin{proposition} \label{prop:examples}
Let $(G,x)$ be in Theorem~\ref{thm:derangement}. Then $x$ is a totally deranged element of $G$.
\end{proposition}

\begin{proof}
As usual, let $q=p^f$ and assume $x \in \Inndiag(T)\wp^j$ for $j$ dividing $f$ for $e=f/j$ odd.

\emph{\textbf{Case~1.} $T = \Sp_4(q)$ with $q$ even.}\nopagebreak

Let $F\: T\wp^j \to  T_0$ be the Shintani map of $(X,\s_1,\s_2) = (\Sp_4, \p^f, \p^j)$ where $T_0 = \Sp_4(q_0)$ for $q_0 = 2^j$. Write $x_0 = F(x)$ noting that $|x_0| = |x^{f/j}|$. Write $G = \< T, \wrho^i\>$ where $i$ divides $j$. Let $M$ be a core-free maximal subgroup of $G$. If $M$ appears in Lemma~\ref{lem:ex_b2}, then that lemma implies that $x \not\in M$. We now consult \cite[Table~8.14]{ref:BrayHoltRoneyDougal} to identify the remaining possibilities for $M$.

First assume that $M$ is a Borel subgroup, so $M$ is $\< T, \wp^i \>$-conjugate to $\< Y, \wp^i \>_{\s_1}$ for a closed $\rho$-stable Borel subgroup $Y \leq X$. Therefore, if $x \in M$, then Theorem~\ref{thm:shintani_subgroups} implies that $x_0 \in Y_{\s_2}$, a Borel subgroup $[q_0^4]{:}(q_0-1)^2$ of $T_0$, contradicting the conditions on $|x_0|$ in Theorem~\ref{thm:derangement}.

It remains to assume that $M = N_G(H)$ where $H$ is a maximal torus of order $(q-1)^2$, $(q+1)^2$ or $q^2+1$. In this case, $M$ is $\< T, \wp^i \>$-conjugate to $\< Y, \wp^i \>_{\s_1}$ where $Y$ is a maximal torus of $X$. Suppose that $x \in M$, and write $x \in S_{s\s_1}t\ws_2$ where $S$ is a maximally split $\s_2$-stable maximal torus of $X$ and $s, t \in N_T(S)$. Then, by Theorem~\ref{thm:shintani_cosets}, $x_0 \in S_{t\s_2}s\ws_1$. The conditions on $|x_0|$ in Theorem~\ref{thm:derangement}, imply that $|S_{t\s_2}| = q_0^2-1$ and $s\ws_1$ is trivial. Since $|S_{t\s_2}| = q_0^2-1$, we know that $St \in N_T(S)/S \cong D_8$ is a reflection. Since $s\ws_1 = s\ws_2^e = st^e$ is trivial, we know that $s = t^{-e}$ is a reflection, recalling that $e$ is odd, so $|H| = |S_{s\s_1}| = q^2-1$: a contradiction. Therefore, $x$ is not contained in a core-free maximal subgroup of $G$, so $x$ is a totally deranged element of $G$.

\vspace{0.3\baselineskip}

\emph{\textbf{Case~2.} $T = \POm^+_{2m}(q)$ with $2m=2^l$ and $G$ does not contain triality.}\nopagebreak

By replacing $G$ by $G^\t$ or $G^{\t^2}$ when $m=4$ if necessary, we may assume that $G \leq \PGaO^+_{2m}(q)$. Inspecting \cite[Table~3.5.E]{ref:KleidmanLiebeck} (or \cite[Table~8.50]{ref:BrayHoltRoneyDougal} if $m=4$), we see that every core-free maximal subgroup of $G$ appears is in the statement of Lemma~\ref{lem:ex_d}, so by that lemma, $x$ is contained in no core-free maximal subgroup of $G$, so $x$ is totally deranged.

\vspace{0.3\baselineskip}

\emph{\textbf{Case~3.} $T = \POm^+_8(q)$ and $G$ contains triality.}\nopagebreak

Let $F\: \Inndiag(T)\wp^j \to \Inndiag(T_0)$ be the Shintani map of $(X,\s_1,\s_2) = (\POm_8, \p^f, \p^j)$ where $T_0 = \POm^+_8(q_0)$ for $q_0 = p^j$. Write $x_0 = F(x)$ noting that $|x_0| = |x^{f/j}|$. Let $M$ be a core-free maximal subgroup of $G$. If $M$ appears in Lemma~\ref{lem:ex_d}, then that lemma implies that $x \not\in M$. We now consult \cite[Table~8.50]{ref:BrayHoltRoneyDougal} to identify the other possibilities for $M$.

First assume that $M$ is a $P_{1,3,4}$ parabolic subgroup or a subgroup of type $\O_2^\pm(q) \times \GL_3^\pm(q)$. If $x \in M$, then via the usual application of Theorem~\ref{thm:shintani_subgroups}, $x_0$ is contained in a $P_{1,3,4}$ parabolic subgroup of $\Inndiag(T_0)$ or a subgroup of $\Inndiag(T_0)$ of type $\O_2^\pm(q_0) \times \GL_3^\pm(q_0)$. However, both of these possibilities is inconsistent with the order of $x_0$.

Next assume that $M = N_G(H)$ where $H$ is a maximal torus of order $(q^2+1)^2$. Suppose that $x \in M$, and write $x \in S_{s\s_1}t\ws_2$ where $S$ is a maximally split maximal torus and $s,t \in N_{\Inndiag(T)}(S)$. Then by Theorem~\ref{thm:shintani_cosets}, $x_0 \in S_{t\s_2}s\ws_1$. By the conditions on $|x_0|$ in Theorem~\ref{thm:derangement}, we conclude that $|S_{t\s_2}| = q_0^4-1$ and $s\ws_1$ is trivial. Since $|S_{t\s_2}| = q_0^4-1$, by computing in the Weyl group $W_{D_4}$, we know that $t$ has order 4 and is conjugate to its inverse. Since $s\ws_1 = s\ws_2^e = st^e$ is trivial, $s = t^{-e}$. Since $e$ is odd, $s$ is conjugate to $t$, so $|S_{s\s_1}| = q^4-1$, which is a contradiction.

It remains to assume that $M = N_G(H)$ where $H$ has type $2^4.2^6.\PSL_3(2)$. This means that $x^{f/j}$ is contained in a group whose order is divisible by no primes other than $2$, $3$ and $7$, which contradicts $|x^{f/j}| = |x_0|$ being divisible by a primitive prime divisor of $q_0^4-1$. Therefore, $x$ is not contained in a core-free maximal subgroup of $G$, so $x$ is totally deranged.
\end{proof}

\subsection{Completing the proof of Theorem~\ref{thm:derangement}} \label{ss:proof}

Combining the results established in Section~\ref{s:proof} readily gives a proof of Theorem~\ref{thm:derangement}.

\begin{proof}[Proof of Theorem~\ref{thm:derangement}]
By Proposition~\ref{prop:examples}, if $(G,x)$ appears in Theorem~\ref{thm:derangement}, then $x$ is totally deranged. Turning to the converse, assume that $x$ is a totally deranged element of $G$. By Proposition~\ref{prop:computational}, $G$ is an almost simple group of Lie type. By Proposition~\ref{prop:standard}, either $T$ is untwisted and $G \not\leq \< \Inndiag(T), \wp \>$, or $T \in \{ {}^2A_m(q), {}^2D_m(q), {}^2E_6(q), {}^3D_4(q) \}$. Moreover, by Proposition~\ref{prop:bfg}, we may assume that $T \not\in \{B_2(q), G_2(q), F_4(q) \}$.  To summarise, we may divide into the following cases
\begin{enumerate}[(a)]
\item $T \in \{ A_m(q) \, \text{($m \geq 2$)}, \ D_m(q) \, \text{($m \geq 4$)}, \ E_6(q) \}$ and $G \not\leq \< \Inndiag(T), \wp \>$ but $G$ does not contain triality
\item $T = D_4(q)$ and $G$ contains triality
\item $T \in \{ {}^2A_m(q) \, \text{($m \geq 2$)}, \ {}^2D_m(q) \, \text{($m \geq 4$)}, \ {}^2E_6(q) \}$
\item $T = {}^3D_4(q)$
\end{enumerate}
By Propositions~\ref{prop:a}--\ref{prop:d}, we can assume that $x \not\in \< \Inndiag(T), \wp \>$ in case~(a), and $x \not\in \< \Inndiag(T), \g, \wp \>$ in case~(b). We consider cases~(a)--(d) in turn. 

For (a), by replacing $x$ by another generator of $\<x\>$ if necessary, we can assume that $x \in \Inndiag(T)\g\wp^j$ for some $j$ dividing $f$, so combining Propositions~\ref{prop:standard}(ii) and~\ref{prop:graph}(i) shows that $x$ is not totally deranged.

Similarly, for (b), we can assume that $x \in \Inndiag(T)\t\wp^j$ for some $j$ dividing $f$, so the result follows from Propositions~\ref{prop:standard}(iii) and~\ref{prop:graph}(iii).

For (c), by replacing $x$ by another generator of $\<x\>$ if necessary, we can assume that $x \in \Inndiag(T)\wp^j$ for some $j$ dividing $2f$. If $2f/j$ is even, then $j$ divides $f$ and we apply Proposition~\ref{prop:graph}(ii). If $2f/j$ is odd, then $j$ is even and, writing $j_0 = j/2$, we have that $2f/j = f/j_0$ is odd, so $2f/(2f,f+j_0) = 2f/(2f,j)$ and $\< \wp^j \> = \< \wp^{f+j_0} \> = \< \g\wp^{j_0} \>$; in particular, we can assume that $x \in \Inndiag(T)\g\wp^{j_0}$ for odd $f/j_0$ and we apply Proposition~\ref{prop:standard}(ii). 

Arguing in a similar fashion, for (d), we can assume that $x \in \Inndiag(T)\wp^j$ for some $j$ dividing $3f$. If $3$ divides $3f/j$, then $j$ divides $f$ and we apply Proposition~\ref{prop:graph}(iv); otherwise, $3 \div j$ and writing $j_0 = j/3$ we have $\< \wp^j \> = \< \t\wp^{j_0} \>$, so we can assume that $x \in \Inndiag(T)\g\wp^{j_0}$ with $3 \ndiv f/j_0$ and we apply Proposition~\ref{prop:standard}(iii).
\end{proof}

\section{Applications} \label{s:applications}

\subsection{Invariable generation} \label{ss:app_invariable}

In this section, we prove Theorem~\ref{thm:invariable} on invariable generation. We begin with two lemmas, the first of which connects totally deranged elements and invariable generation.

\begin{lemma} \label{lem:invariable}
Let $\< G, a \>$ be a group such that $G \leqn \< G, a \>$. Let $x \in G$. If $\{ x, x^a \}$ is an invariable generating set for $G$, then $x$ is a totally deranged element of $\< G, a \>$. 
\end{lemma}

\begin{proof}
We prove the contrapositive. Assume that $x$ is not a totally deranged element of $\< G, a \>$. Then there exists a maximal subgroup $M < \< G, a \>$ such that $G \not\leq M$ and $x \in M$. Since $M$ is maximal and does not contain $G$, we deduce that $M \not\leq \< G, a^i \>$ for any $i$ such that $\< a^i \> < \< a \>$. Therefore, there exists $g \in G$ such that $ag \in M$. In particular, that $M^{ag} = M$. This means that $x \in M$ and $x^{ag} \in M^{ag} = M$, so $\< x, x^{ag} \> \leq M \cap G < G$, which demonstrates that $\{x, x^a \}$ is not an invariable generating set for $G$. 
\end{proof}

\begin{lemma} \label{lem:invariable_check}
Let $\< G, a \>$ be a finite group such that $G \leqn \< G, a \>$. Let $x \in G$. Then $\{ x, x^a \}$ is an invariable generating set for $G$ if and only if for every maximal subgroup $H$ of $G$, the element $x$ is contained in a $G$-conjugate of at most one of $H$ and $H^a$.
\end{lemma}

\begin{proof}
The set $\{ x^g,x^{ah} \}$ fails to generate $G$ if and only if there exists a maximal subgroup $M$ of $G$ such that $x^g$ and $x^{ah} = x^{h^{a^{-1}}a}$ are both contained in $M$ or equivalently $x$ is contained in both $M^{g^{-1}} = (H^a)^{g^{-1}}$ and $(M^{a^{-1}})^{h^{-a^{-1}}} = H^{h^{-a^{-1}}}$, where $H = M^{a^{-1}}$. The result follows. 
\end{proof}

We may now proceed with the proof of Theorem~\ref{thm:invariable}.

\begin{proof}[Proof of Theorem~\ref{thm:invariable}]
The implication (i)$\implies$(ii) is given by Lemma~\ref{lem:invariable} and (ii)$\implies$(iii) is given by Theorem~\ref{thm:derangement}. Therefore, it suffices to prove that if $(T,a,x)$ appears in part~(iii), then $T$ is invariably generated by $\{x,x^a\}$. (Of course, writing $G = \< T, a\>$, this means that $(G,x)$ appears in Theorem~\ref{thm:derangement} with the additional assumption that $x \in T$.)

\emph{\textbf{Case~1.} $T = \Sp_4(q)$ with $q$ even.}\nopagebreak

By replacing $x$ by some $\Aut(T)$-conjugate if necessary, we assume that $x = y \oplus y^{-\tr}$ with respect to a decomposition $V = \F_q^4 = U \oplus U^*$ where $U$ is a totally singular $2$-space on which $y$ acts irreducibly. This means that $x$ has eigenvalues $\l, \l^q, \l^{-1}, \l^{-q}$ with $|\l|=|y|=|x|$. This implies that $x$ is contained in no $T$-conjugate of $H \in \mathcal{H}_1 = \{ P_1, \ \Sp_2(q) \wr \Sym_2, \ \O^-_4(q) \}$. Moreover, Lemma~\ref{lem:ex_b2} implies that $x$ is contained in no $T$-conjugate of $H \in \mathcal{H}_2$ where  $\mathcal{H}_2$ consists of $\Sp_4(q^{1/k})$ for each prime $k \div f$ and, if $f$ is odd, also ${}^2B_2(q)$. By \cite[Table~8.14]{ref:BrayHoltRoneyDougal}, for every maximal subgroup $H$ of $T$, either $H$ or $H^a$ is $T$-conjugate to a subgroup in $\mathcal{H}_1 \cup \mathcal{H}_2$. Therefore, Lemma~\ref{lem:invariable_check} implies that $\{x,x^a\}$ invariably generates $T$.

\vspace{0.3\baselineskip}

\emph{\textbf{Case~2.} $T = \POm^+_{2m}(q)$ with $2m=2^l$.}\nopagebreak

By replacing $x$ by some $\Aut(T)$-conjugate if necessary, we may assume that $x=su=us$ where $s = y \oplus y^{-\tr}$ with respect to a decomposition $V = \F_q^{2m} = U \oplus U^*$ where $U$ is a totally singular $m$-space on which $y$ acts irreducibly, and $u$ is either trivial or has Jordan form $[J_2^m]$. Let $r$ be a primitive prime divisor of $q^m-1$ that divides $|x|$. 

Let $H$ be a maximal subgroup of $T$ that contains $x$. In particular, $H$ is not in Lemma~\ref{lem:ex_d}. We now consult \cite[Table~3.5.E]{ref:KleidmanLiebeck} (or \cite[Table~8.50]{ref:BrayHoltRoneyDougal} if $m=4$) to identify the remaining possibilities for $H$. We will show that $x$ is not contained in any $T$-conjugate of $H^a$, which completes the proof by Lemma~\ref{lem:invariable_check}. Throughout, let $\Aut_0(T)$ be $\< \Inndiag(T), \p \>$.

First assume that $H$ is the stabiliser of an $m$-space or is a subgroup of type $\GL_m(q)$. If $u=1$, then $U$ and $U^*$ are the only maximal totally singular subspaces stabilised by $x^2$, so $T_U$ and $T_{U^*}$ are the only stabilisers of $m$-spaces containing $x$ and $T_{U \oplus U^*}$ is the only subgroup of type $\GL_m(q)$ containing $x$. If $u \neq 1$, then $T_U$ is the only stabiliser of an $m$-space containing $x$ and no subgroups of type $\GL_m(q)$ contain $x$. 

Let us pause to make some observations that will prove useful later. Let $H$ be the stabiliser of an $m$-space in $T$, which is a maximal subgroup of $T$. Then the $\PGaO^+_{2m}(q)$-class of $H$ splits into exactly two $\Aut_0(T)$-classes (which are also $T$-classes). Given two maximally totally singular subspaces $W_1$ and $W_2$ of $V$, the stabilisers $T_{W_1}$ and $T_{W_2}$ are $\Aut_0(T)$-conjugate (equivalently, $T$-conjugate) if and only if $\dim(W_1 \cap W_2)$ is even (in which case, we say that $W_1$ and $W_2$ have the same \emph{type}), see \cite[Description~4 (on p.30) \& Table~3.5.E]{ref:KleidmanLiebeck}. In particular, since $U$ and $U^*$ have trivial intersection, the stabilisers $T_U$ and $T_{U^*}$ are $\Aut_0(T)$-conjugate. If $a \in \PGaO^+_{2m}(q) \setminus \Aut_0(T)$, then this means that even when $x$ is contained in different stabilisers of an $m$-space, $x$ is not contained in an $\Aut_0(T)$-conjugate $H^a$. If $m=4$ and $a \in \Aut(T) \setminus \PGaO^+_8(q)$, then $H^a$ is either the stabiliser of an $m$-space or a $1$-space, and since $x$ stabilises no $1$-spaces of $V$, again $x$ is not contained in an $\Aut_0(T)$-conjugate of $H^a$. In all cases, as $x$ is not contained in an $\Aut_0(T)$-conjugate of $H^a$, it follows that $\<x\>$ is not $\Aut_0(T)$-conjugate to $\<x^{a^{-1}}\>$, or equivalently, $\<x\>$ is not $\Aut_0(T)$-conjugate to $\<x^a\>$.

Now assume that $H$ is not the stabiliser of an $m$-space. The rest of the proof will focus on showing that $\<x\>^{\Aut(T)} \cap H = \<x\>^{\Aut_0(T)} \cap H$, which we claim completes the proof. To see this, assume that $\<x\>^{\Aut(T)} \cap H = \<x\>^{\Aut_0(T)} \cap H$ and suppose that $x \in (H^a)^g$ for some $g \in T$. Then $(H^a)^g$ contains both $x$ and $x^{ag}$, so, by our assumption, $\<x\>$ and $\<x^a\>^g$ are $\Aut_0(T)$-conjugate, but this contradicts the observation that $\<x\>$ is not $\Aut_0(T)$-conjugate to $\<x^a\>$. Therefore, $x$ is not contained in any $T$-conjugate of $H^a$ and we appeal to Lemma~\ref{lem:invariable_check}. 

First assume that $H$ has type $\GU_m(q)$. Now $H$ has a unique class of cyclic subgroups of order $|x|$, so $\<x\>^{\Aut(T)} \cap H = \<x\>^H$ and, in particular, $\<x\>^{\Aut(T)} \cap H = \<x\>^{\Aut_0(T)} \cap H$.

Next assume that $H$ has type $\Sp_{k_1}(q) \otimes \Sp_{k_2}(q)$ or $\O^{\e_1}_{k_1}(q) \otimes \O^{\e_2}_{k_2}(q)$ where $2m=k_1k_2$. Since $|x|$ is divisible by $r$, we deduce that, without loss of generality, $k_1=m$ and $k_2=2$, and in the orthogonal case, $\e_1=-$. Again $H$ has a unique class of cyclic subgroups of order $|x|$, so $\<x\>^{\Aut(T)} \cap H = \<x\>^{\Aut_0(T)} \cap H$.

Now assume that $H$ has type $2^{1+2l}.\O^+_{2l}(2)$ where $q = p \geq 3$ (recall that $2m=2^l$), so $H$ is an extension of $2^{2l}$ by $\Omega^+_{2l}(2)$ or $\O^+_{2l}(2)$. The assumption that $x \in H$ is highly restrictive. Indeed, since $2^{2l}$ has exponent $2$, we know that $|x^2|$ must divide $\O^+_{2l}(2)$. Now $r$ divides $|x^2|$ and $r \equiv 1 \mod{2^{l-1}}$, so we must have $r=2^{l-1}+1$. Let $x'$ be the image of $x$ in $\O^+_{2l}(2)$. Then $x' = (x_1,x_2) \in \O^-_{2l-2}(2) \times \O^-_2(2)$ where $|x_1|=r=2^{l-1}+1$ and $|x_2| \in \{1,2,3\}$. Suppose that $|x_2| \neq 3$. Then $|x|$ divides $4r$, which divides $2(q^{m/2}+1)$, which is a contradiction. Therefore, $|x_2|=3$. Now $\Omega^-_{2l}(2)$ (and $\O^-_{2l}(2)$) has a unique class of cyclic subgroups of order $3r = (2+1)(2^{l-1}+1)$, which implies that $\<x^2\>^{\Aut(T)} \cap H = \<x^2\>^{\Aut_0(T)} \cap H$. Therefore, it suffices to establish that $x^2$ is not $\Aut_0(T)$-conjugate to $(x^2)^a$. However, $x^2$ satisfies the condition of Theorem~\ref{thm:derangement} in this case (since $3$ does not divide $q^m+1$), so the observation that $x$ is not $\Aut_0(T)$-conjugate to $x^a$ applies to $x^2$ too. This handles subgroups of type $2^{1+2l}.\O^+_{2l}(2)$.

Now assume that $H$ has type $\Sp_{k_1}(q) \wr \Sym_{k_2}$ or $\O^\e_{k_1}(q) \wr \Sym_{k_2}$ where $2m=k_1^{k_2}$. Here, since $|x|$ is divisible by the prime $r > m$ and $k_2 < m$, we deduce that $\Sp_{k_1}(q)$ or $\O^\e_{k_1}(q)$ contains an element of order $r$, but $k_1 < m$, so this is impossible.

All that remains is to consider the case where $H$ has type $\O^+_m(q^2)$. Let $x_\sharp$ be the preimage of $x$ in $\O^+_m(q^2)$, which naturally acts on $V_\sharp = \F_{q^2}^m$. Note that $x_\sharp$ stabilises a maximal totally singular subspace $U_\sharp$ of $V_\sharp$ corresponding to the maximal totally singular subspace $U$ of $V$. Let $y \in x^{\Aut(T)} \cap H$ and let $y_\sharp$ be the preimage of $y$ in $\O^+_m(q^2)$. Let us first address the case where $m=4$ and $y$ is not $\PGaO^+_8(q)$-conjugate to $x$. Here $y$ has exactly two eigenvalues of order dividing $q-1$, so $y_\sharp$ has a unique eigenvalue with order dividing $q-1$, but this is impossible for an element of $H$. Therefore, we may now assume that $y$ is $\PGaO^+_{2m}(q)$-conjugate to $x$. In this case, we can fix a maximal totally singular subspace $W_\sharp$ of $V_\sharp$ stabilised by $y_\sharp$. Now $\dim(U \cap W) = 2 \dim(U_\sharp \cap W_\sharp)$, which is even, so $U$ and $W$ have the same type. Since all maximal totally singular subspaces of $V$ stabilised by $y$ have the same type, we deduce that all maximal subgroups of $T$ of type $P_m$ that contain $y$ are $\Aut_0(T)$-conjugate to $T_U$. Therefore, $y$ is $\Aut_0(T)$-conjugate to $x$, so $\<x\>^{\Aut(T)} \cap H = \<x\>^{\Aut_0(T)} \cap H$. 

We have now shown that for all maximal subgroups $H$ of $T$, the element $x$ is contained in a $T$-conjugate of at most one of $H$ and $H^a$, so Lemma~\ref{lem:invariable_check} completes the proof.
\end{proof}

\subsection{Unique maximal overgroups} \label{ss:app_unique}

We conclude by proving Theorem~\ref{thm:unique} on unique maximal overgroups, beginning with a strong version of the implication (iii)$\implies$(i). Here, and in what follows, for a finite simple group of Lie type $T$, write $\ddot{x} = Tx \in \Out(T)$ if $x \in \Aut(T)$, write $\ddot{A} = \{ \ddot{x} \mid x \in A \}$ if $A \subseteq \Aut(T)$ and write $D = \< \Inndiag(T), \p \>$. 

\begin{proposition} \label{prop:unique}
Let $T$ be a finite simple group of Lie type, $T \leq G \leq \Aut(T)$ and $x \in G$. Assume that $(G,x)$ appears in Theorem~\ref{thm:unique}. Then $\< \Inndiag(T), \p \> \cap G$ is the unique maximal subgroup of $G$ that contains $x$. 
\end{proposition}

\begin{proof}
Let $(G,x)$ be in Theorem~\ref{thm:unique}. By Theorem~\ref{thm:derangement}, $x$ is a totally deranged element of $G$, so $x$ is contained in no core-free maximal subgroup of $G$. Therefore, it suffices to prove that $\ddot{G} \cap \ddot{D}$ is the unique maximal subgroup of $\ddot{G}$ that contains $\ddot{x}$. In all cases, $G \not\leq D$ and replacing $x$ by an $\Aut(T)$-conjugate if necessary, $x \in \Inndiag(T)\wp^j$ for a divisor $j$ of $f$ such that $f/j$ is odd. 

\emph{\textbf{Case~1.} $T = \Sp_4(q)$.}\nopagebreak

The maximal subgroups of $\ddot{G}$ that contain $\ddot{x}$ correspond to the maximal subgroups of $\ddot{G}/\< \ddot{x} \>$. However, $\ddot{G}/\< \ddot{x}\> \cong C_{2j/i}$ and $2j/i$ is a power of $2$, so the unique index two subgroup of $\ddot{G}/\< \ddot{x}\>$ is its unique maximal subgroup. Therefore, $\< \ddot{\p}^i \> = \ddot{G} \cap \ddot{D}$ is the unique maximal subgroup of $\ddot{G}$ containing $\ddot{x}$.

\vspace{0.3\baselineskip}

\emph{\textbf{Case~2.} $T = \POm^+_{2m}(q)$ and $y$ is a duality graph automorphism.}\nopagebreak

Since $\ddot{x}, (\ddot{y}\ddot{\p}^i)^2 \in \ddot{D}$, it is easy to check that $\< \ddot{x}, \ddot{x}^{\ddot{y}\ddot{\p}^i}, (\ddot{y}\ddot{\p}^i)^2 \>$ is an index two subgroup of $\ddot{G}$ wholly contained in $\ddot{D}$, so it is necessarily $\ddot{G} \cap \ddot{D}$. Now $\ddot{G} \cap \ddot{D}$ is certainly a maximal subgroup of $\ddot{G}$ containing $\ddot{x}$. Let $\ddot{M}$ be any maximal subgroup subgroup of $\ddot{G}$ containing $\ddot{x}$. Suppose that $\ddot{M} \neq \ddot{G} \cap \ddot{D}$. Then $\ddot{M} \not\leq \ddot{G} \cap \ddot{D}$, so we may fix $\ddot{g} \in \ddot{M}$ such that $\ddot{g} \not\in \ddot{D}$. Then $\ddot{x}^{\ddot{y}\ddot{\p}^i} = \ddot{x}^{\ddot{g}} \in \ddot{M}$, so $\ddot{K} = \< \ddot{x}, \ddot{x}^{\ddot{y}\ddot{\p}^i} \>$ is a normal subgroup of $\ddot{G}$ contained in $\ddot{M}$. Therefore, the possibilities for $\ddot{M}$ correspond to the maximal subgroups of $\ddot{G}/\ddot{K}$ containing $\ddot{M}/\ddot{K}$. Now $\ddot{G}/\ddot{K}$ is a cyclic group of order dividing $|\ddot{y}\ddot{\p}^i|$, which is a power of $2$, so $\ddot{M} = \< \ddot{K}, (\ddot{y}\ddot{\p}^i)^2 \> = \ddot{G} \cap \ddot{D}$, which contradicts $\ddot{g} \in \ddot{M}$. Therefore, $\ddot{M} = \ddot{G} \cap \ddot{D}$.

\emph{\textbf{Case~3.} $T = \POm^+_8(q)$ and $y$ is a triality graph automorphism.}\nopagebreak

The argument is very similar to in Case~2, so we just sketch it. First note that $\ddot{G} \cap \ddot{D} = \< \ddot{x}, \ddot{x}^{\ddot{y}\ddot{\p}^i}, \ddot{x}^{(\ddot{y}\ddot{\p}^i)^2}, (\ddot{y}\ddot{\p}^i)^3 \>$, which is a maximal subgroup of $\ddot{G}$ containing $\ddot{x}$. Next let $\ddot{M}$ be any maximal subgroup subgroup of $\ddot{G}$ containing $\ddot{x}$, and suppose that $\ddot{M} \neq \ddot{G} \cap \ddot{D}$. Then $\ddot{K} = \< \ddot{x}, \ddot{x}^{\ddot{y}\ddot{\p}^i}, \ddot{x}^{(\ddot{y}\ddot{\p}^i)^2} \>$ is a normal subgroup of $\ddot{G}$ contained in $\ddot{M}$, so the possibilities for $\ddot{M}$ correspond to the maximal subgroups of $\ddot{G}/\ddot{K}$ containing $\ddot{M}/\ddot{K}$. Since $\ddot{G}/\ddot{K}$ is a cyclic group of order dividing $|\ddot{y}\ddot{\p}^i|$, which is a power of $3$, we deduce that $\ddot{M} = \< \ddot{K}, (\ddot{y}\ddot{\p}^i)^3 \> = \ddot{G} \cap \ddot{D}$, which is a contradiction, so $\ddot{M} = \ddot{G} \cap \ddot{D}$.
\end{proof}

\begin{proof}[Proof of Theorem~\ref{thm:unique}]
The equivalence (i)$\iff$(ii) is clear, and the implication (iii)$\implies$(i) is given by Proposition~\ref{prop:unique}. Therefore, the remainder of the proof will be dedicated to proving the implication (ii)$\implies$(iii). Let $(G,x)$ be in Theorem~\ref{thm:derangement} and assume that $\ddot{x}$ is contained in a unique maximal subgroup of $\ddot{G}$. We claim that $(G,x)$ is given in Theorem~\ref{thm:unique}.

\emph{\textbf{Case~1.} $T = \Sp_4(q)$.}\nopagebreak

In this case $\Out(T) = \< \ddot{\rho} \>$. Since $\ddot{G} \not\leq \< \ddot{\p} \>$, we have $\ddot{G} = \< \ddot{\rho}^i \>$ for some odd divisor $i$ of $f$. Since $\ddot{x} \in \ddot{G}$, we must have $i \div j$. The maximal subgroups of $\ddot{G}$ that contain $\ddot{x}$ correspond to the maximal subgroups of $\ddot{G}/\<\ddot{x}\>$. In particular, $\ddot{x}$ is contained in a unique maximal subgroup of $\ddot{G}$ if and only if $2j/i = |\ddot{G}/ \<\ddot{x}\>|$ is a power of $2$.

\vspace{0.3\baselineskip}

\emph{\textbf{Case~2.} $T = \POm^+_{2m}(q)$ and $G$ does not contain triality.} \nopagebreak

If $m > 4$, then $G \leq \PGaO^+_8(q) = \< D, \g \>$ and we write $\alpha  = \g$. If $m=4$, then we can fix $\alpha \in \{\g, \g\t, \g\t^2\}$ such that $G \leq \< D, \alpha\>$. In all cases, $\alpha$ is a duality graph automorphism.

First assume that $\< \ddot{x} \>$ is normal in $\ddot{G}$. The maximal subgroups of $\ddot{G}$ containing $\ddot{x}$ correspond to the maximal subgroups of $\ddot{G}/\<\ddot{x}\>$, so $\ddot{G}/\<\ddot{x}\>$ is a cyclic $2$-group. In particular, the assumption $G \not\leq D$ means that we may write $\ddot{G} = \< \ddot{x}, \ddot{h}\ddot{\a}\ddot{\p}^i \>$ where $h \in \Inndiag(T)$ and $i \div j$. If $j/i$ had an odd prime divisor $r$, then $r$ would divide $|\ddot{G}/\< \ddot{x} \>|$, a contradiction. Therefore, $j/i$ is a power of $2$. This proves the result (with $y=h\a$).

For the rest of this case, we can assume that $p$ is odd and $\< \ddot{x} \> = \< \ddot{\d}\ddot{\p}^j \>$. Here we take a different approach. We know that $\ddot{G} \leq \< \ddot{\d}, \ddot{\a} \> \times \< \ddot{\p} \>$. Assume that the projection of $\ddot{G}$ onto $\< \ddot{\p} \> \cong C_f$ is $\< \ddot{\p}^i \>$ where $i \div f$, so $\ddot{G} \leq \ddot{A}$ where $\ddot{A} = \< \ddot{\d}, \ddot{\a}, \ddot{\p}^i\>$. The assumption that $G \not\leq D$ is equivalent to the projection of $\ddot{G}$ onto $\< \ddot{\d}, \ddot{\a} \> \cong D_8$ not being contained in $\< \ddot{\d}, \ddot{\d}^{\ddot{\a}}\> \cong C_2^2$, and since $\ddot{\d}\ddot{\p}^j \in \ddot{G}$ we deduce that $\ddot{G}$ projects onto $\< \ddot{\d}, \ddot{\a} \>$. Therefore, by Goursat's lemma, either $\ddot{G} = \ddot{A}$ or $f/i$ is even and $\ddot{G} = \{ (x_1,x_2) \in \ddot{A} \mid \ddot{N}x_1 = \<\ddot{\p}^{2i}\>x_2 \}$ where $N$ is one of $\< \ddot{\d}, \ddot{\d}^{\ddot{\a}} \>$, $\< \ddot{\a}, \ddot{\a}^{\ddot{\d}} \>$ or $\< \ddot{\d}\ddot{\a} \>$. Note that in all cases, $\ddot{G} \cap \ddot{D}$ is a maximal (in fact, index two) subgroup of $\ddot{G}$ containing $\ddot{x}$, and it is useful to record that $\ddot{G} \cap \ddot{D}$ is abelian. We will now prove that if $\ddot{G} \cap \ddot{D}$ is the only maximal subgroup of $\ddot{G}$ that contains $\ddot{x}$, then $j/i$ is a power of $2$ and $\ddot{G} = \< \ddot{x}, \ddot{h}\ddot{\a}\ddot{\p}^i \>$ for some $h \in \Inndiag(T)$, as required (again with $y=h\a$).

First assume that $\ddot{G} = \ddot{A}$. If $j > i$, then $\< \ddot{\d}, \ddot{\a}, \ddot{\p}^j\>$ is a proper nonabelian subgroup of $\ddot{G}$ containing $\ddot{x}$, a contradiction, so $j=i$. If $f/i$ is even, then $\< \ddot{\d}\ddot{\p}^i, \ddot{\a} \>$ is a proper nonabelian subgroup of $\ddot{G}$ containing $\ddot{x}$, a contradiction, so $f/i$ is odd. Therefore, $\ddot{G} = \< \ddot{x}, \ddot{\a}\ddot{\p}^i \>$.

From now on we assume that $\ddot{G} < \ddot{A}$. For now assume that $\ddot{N} = \< \ddot{\d}\ddot{\a} \>$, so $\ddot{G} = \< \ddot{\d}\ddot{\p}^i, \ddot{\d}\ddot{\a} \>$. Since $\ddot{x} = \ddot{\d}\ddot{\p}^j \in \ddot{G}$,  $j/i$ is odd. If $j > i$, then $\< \ddot{\d}\ddot{\a}, \ddot{\d}\ddot{\p}^j\>$ is a proper nonabelian subgroup of $\ddot{G}$ containing $\ddot{x}$, a contradiction, so $j=i$. Note that $\ddot{G} = \< \ddot{\d}\ddot{\p}^i, (\ddot{\d}\ddot{\p}^i)\ddot{\d}\ddot{\a} \> = \< \ddot{x}, \ddot{\a}\ddot{\p}^i \>$.

Next assume that $\ddot{N} = \< \ddot{\a}, \ddot{\a}^{\ddot{\d}} \>$, so $\ddot{G} = \< \ddot{\d}\ddot{\p}^i,  \ddot{\a}\>$ and $j/i$ is odd. If $j > i$, then $\< \ddot{\d}\ddot{\p}^j, \ddot{\a} \>$ is a proper nonabelian subgroup of $\ddot{G}$ containing $\ddot{x}$, a contradiction, so $j=i$. Note that $\ddot{G} = \< \ddot{\d}\ddot{\p}^i, (\ddot{\d}\ddot{\p}^i)\ddot{\a} \> = \< \ddot{x}, \ddot{\d}\ddot{\a}\ddot{\p}^i \>$.

Finally assume that $\ddot{N} = \< \ddot{\d}, \ddot{\d}^{\ddot{\a}} \>$, so $\ddot{G} = \< \ddot{\d}, \ddot{\a}\ddot{\p}^i \>$ and $j/i$ is even. If $j/i$ has an odd prime divisor $r$, then $\< \ddot{\d}, \ddot{\a}\ddot{\p}^{ri} \>$ is a proper nonabelian subgroup of $\ddot{G}$ containing $\ddot{x}$, a contradiction, so $j/i$ is a power of $2$. Note that $\ddot{G} = \< \ddot{\d}, \ddot{\a}\ddot{\p}^i \> =  \< \ddot{\d}(\ddot{\a}\ddot{\p}^i)^{j/i}, \ddot{\a}\ddot{\p}^i \> = \< \ddot{x}, \ddot{\a}\ddot{\p}^i \>$.
 
\vspace{0.3\baselineskip}
 
\emph{\textbf{Case~3.} $T = \POm^+_8(q)$ and $G$ contains triality.}\nopagebreak

This case is very similar to Case~2. First assume that $\ddot{x} = \ddot{\p}^j$. The maximal subgroups of $\ddot{G}$ containing $\ddot{x}$ correspond to the maximal subgroups of $\ddot{G}/\<\ddot{x}\>$, so $\ddot{G}/\<\ddot{x}\>$ is a cyclic $l$-group for some prime $l$. In particular, since $G \not\leq \< \Inndiag(T), \g, \p \>$, we may write $\ddot{G} = \< \ddot{x}, \ddot{h}\ddot{\t}\ddot{\p}^i \>$ where $h \in \Inndiag(T)$ and $i \div j$. Observe that $3$ divides $|\ddot{G}/\< \ddot{x}\>|$. If $j/i$ had a prime divisor $r \neq 3$, then $r$ would divide $|\ddot{G}/\< \ddot{x} \>|$, a contradiction. Therefore, $j/i$ is a power of $3$.

For the rest of this case, we can assume that $p$ is odd and $\< \ddot{x} \> = \< \ddot{g}\ddot{\p}^j \>$ for some element $g \in \Inndiag(T) \setminus T$. We know that $\ddot{G} \leq \< \ddot{\d}, \ddot{\g}, \ddot{\t} \> \times \< \ddot{\p} \>$. Assume that the projection of $\ddot{G}$ onto $\< \ddot{\p} \> \cong C_f$ is $\< \ddot{\p}^i \>$ where $i \div f$, so $\ddot{G} \leq \< \ddot{\d}, \ddot{\g}, \ddot{\t}, \ddot{\p}^i\>$. Since $G$ contains triality, $|GD/D|$ is divisible by $3$, so $\ddot{G}$ projects onto at least $\< \ddot{g}, \ddot{\t} \> = \< \ddot{\d}, \ddot{\t}\>$. However, if $\ddot{G}$ projects onto $\< \ddot{\d}, \ddot{\g}, \ddot{\t} \>$, then $\ddot{\g}$ has multiple maximal overgroups in $\ddot{G}$, so $\ddot{G}$ projects onto exactly $\< \ddot{g}, \ddot{\t} \>$. Therefore, either $\ddot{G} = \ddot{A} = \< \ddot{\d}, \ddot{\t}, \ddot{\p}^i\>$ or $3$ divides $f/i$ and $\ddot{G} = \{ (g_1,g_2) \in \ddot{A} \mid \ddot{N}g_1 = \<\ddot{\p}^{3i}\>g_2 \}$ where $N = \< \ddot{\d}, \ddot{\d}^{\ddot{\g}} \>$, so, in other words, $\ddot{G} = \< \ddot{\d}, \ddot{\t}\ddot{\p}^i \>$. In both cases, $\ddot{G} \cap \ddot{D}$ is a maximal (index three) subgroup of $\ddot{G}$ containing $\ddot{x}$, and we record that $\ddot{G} \cap \ddot{D}$ is abelian. 

We will now prove that if $\ddot{G} \cap \ddot{D}$ is the only maximal subgroup of $\ddot{G}$ that contains $\ddot{x}$, then $j/i$ is a power of $3$ and $\ddot{G} = \< \ddot{x}, \ddot{h}\ddot{\t}\ddot{\p}^i \>$ for some $h \in \Inndiag(T)$.

First assume that $\ddot{G} = \< \ddot{\d}, \ddot{\t}, \ddot{\p}^i\>$. If $j > i$, then $\< \ddot{\d}, \ddot{\t}, \ddot{\p}^j\>$ is a proper nonabelian subgroup of $\ddot{G}$ containing $\ddot{x}$, a contradiction, so $j=i$. If $f/i$ is even, then $\< \ddot{\d}\ddot{\p}^i, \ddot{\t} \>$ is a proper nonabelian subgroup of $\ddot{G}$ containing $\ddot{x}$, a contradiction, so $f/i$ is odd. Therefore, $\ddot{G} = \< \ddot{x}\ddot{\p}^i, \ddot{\t} \> = \< \ddot{x}\ddot{\p}^i, (\ddot{x}\ddot{\p}^i)\ddot{\t} \> = \< \ddot{x}, \ddot{x}\ddot{\t}\ddot{\p}^i \>$. 

Now assume that $\ddot{G} = \< \ddot{\d}, \ddot{\t}\ddot{\p}^i \>$, so $j/i$ is divisible by $3$. If $j/i$ has a prime divisor $r \neq 3$, then $\< \ddot{\d}, \ddot{\t}^r\ddot{\p}^{ri} \>$ is a proper nonabelian subgroup of $\ddot{G}$ containing $\ddot{x}$, a contradiction, so $j/i$ is a power of $3$. Note that $\ddot{G} = \< \ddot{\d}, \ddot{\t}\ddot{\p}^i \> = \< \ddot{\d}(\ddot{\t}\ddot{\p}^i)^{j/i}, \ddot{\t}\ddot{\p}^i \> = \< \ddot{x}, \ddot{\t}\ddot{\p}^i \>$.
\end{proof}

\vspace{11pt}

\noindent Scott Harper \newline
School of Mathematics and Statistics, University of St Andrews, KY16 9SS, UK \newline
\texttt{scott.harper@st-andrews.ac.uk}

\begin{furtheracknowledgements*}
In order to meet institutional and research funder open access requirements, any accepted manuscript arising shall be open access under a Creative Commons Attribution (CC BY) reuse licence with zero embargo.
\end{furtheracknowledgements*}

\end{document}